\documentclass[12pt]{amsart}
\usepackage{geometry}   %????????
\usepackage{enumitem} % for \setlist

\usepackage[colorlinks,citecolor = red, linkcolor=blue,hyperindex]{hyperref}
\usepackage{euscript,eufrak,verbatim, mathrsfs}
\usepackage[psamsfonts]{amssymb}
\usepackage{bbm}
\usepackage{graphicx}
\usepackage{float}
\usepackage{float, tikz}
\usepackage{orcidlink}
\usepackage{caption}
\usetikzlibrary{decorations.pathreplacing}
\usepackage{hyperref}
\usetikzlibrary{arrows.meta}
\usepackage{centernot}

\usepackage{extpfeil}

\usepackage[all, cmtip]{xy}

\usepackage{upref, xcolor, dsfont}
\usepackage{amsfonts,amsmath,amstext,amsbsy, amsopn,amsthm}
\usepackage{enumerate}

\usepackage{url}

\usepackage{mathtools}
\usepackage{bookmark}

\usepackage{caption}
\usepackage{euscript}
\usepackage{times}
\usepackage{type1cm}        % activate if the above 3 fonts are
\usepackage{multicol}        % used for the two-column index
\usepackage[bottom]{footmisc}% places footnotes at page bottom
\newcommand{\normmm}[1]{{\left\vert\kern-0.25ex\left\vert\kern-0.25ex\left\vert #1 
		\right\vert\kern-0.25ex\right\vert\kern-0.25ex\right\vert}}

\usepackage{xpatch}

\newtheorem{mainthm}{Theorem}

\newtheorem{mainprop}[mainthm]{Proposition}

\newtheorem{theorem}{Theorem}[section]
\newtheorem*{theorem*}{Theorem B}
\newtheorem{lemma}[theorem]{Lemma}

\newtheorem{proposition}[theorem]{Proposition}

\newtheorem{corollary}[theorem]{Corollary}

\newtheorem*{observation*}{Observation}
\newtheorem{assumption}{Assumption}
\newtheorem*{assumption*}{Assumption}
\newtheorem*{question*}{Question}

\theoremstyle{definition}
\newtheorem{definition}{Definition}
\newtheorem*{definition*}{Definition}

\theoremstyle{remark}

\newtheorem*{remark*}{Remark}

\newcommand{\R}{\mathbb{R}}
\newcommand{\N}{\mathbb{N}}
\newcommand{\Z}{\mathbb{Z}}

\newcommand{\C}{\mathbb{C}}
\newcommand{\E}{\mathbb{E}}

\newcommand{\PP}{\mathbb{P}}

\newcommand{\XX}{P}

\newcommand{\Cov}{\mathrm{Cov}}
\newcommand{\supp}{\mathrm{supp}}

\newcommand{\dist}{\mathrm{dist}}

\newcommand{\an}{\text{\, and \,}}
\newcommand{\anand}{\quad\text{and}\quad}

\makeatletter
\let\oldtocsection=\tocsection
\let\oldtocsubsection=\tocsubsection
\renewcommand{\tocsection}[2]{\hspace{0em}\oldtocsection{#1}{#2}}
\renewcommand{\tocsubsection}[2]{\hspace{2em}\oldtocsubsection{#1}{#2}}
\makeatother
\makeatletter
\xpatchcmd{\@tocline}
{\hfil\hbox to\@pnumwidth{\@tocpagenum{#7}}\par}
{\ifnum#1<0\hfill\else\dotfill\fi\hbox to\@pnumwidth{\@tocpagenum{#7}}\par}
{}{}
\makeatother
\makeatletter
\renewcommand\section{\@startsection{section}{1}{\z@}%
	{-3.5ex \@plus -1ex \@minus -.2ex}%
	{2.3ex \@plus.2ex}%
	{\normalfont\large\bfseries}}%
\makeatother
\numberwithin{equation}{section}

\begin{document}
	\title[Fourier dimensions of classical multiplicative chaos measures]{Harmonic analysis of multiplicative chaos\\Part II: a unified approach to Fourier dimensions}

	\author
	{Zhaofeng Lin}
	\address
	{Zhaofeng LIN: School of Fundamental Physics and Mathematical Sciences, HIAS, University of Chinese Academy of Sciences, Hangzhou 310024, China}
	\email{linzhaofeng@ucas.ac.cn}

	\author%[authorlabel1]
	{Yanqi Qiu}
	\address%[authorlabel1]
	{Yanqi QIU: School of Fundamental Physics and Mathematical Sciences, HIAS, University of Chinese Academy of Sciences, Hangzhou 310024, China}
	\email{yanqi.qiu@hotmail.com, yanqiqiu@ucas.ac.cn}

	\author
	{Mingjie Tan}
	\address
	{Mingjie TAN: School of Mathematics and Statistics, Wuhan University, Wuhan 430072, China}
	\email{mingjie1tan1wuhan@gmail.com}

	%\date{\today}

	\begin{abstract}
	We introduce a unified approach for studying the polynomial Fourier decay of classical multiplicative chaos measures. As consequences, we obtain the precise Fourier dimensions for multiplicative chaos measures arising from the following key models:  the sub-critical 1D and  2D GMC (which in particular resolves the Garban-Vargas conjecture); the sub-critical $d$-dimensional GMC  with $d \ge 3$ when the parameter $\gamma$ is near the critical value; the canonical Mandelbrot random coverings; the canonical Mandelbrot cascades. For various other models, we establish the non-trivial lower bounds of the Fourier dimensions and in various cases we conjecture that they are all optimal and provide the exact values of Fourier dimensions.
	\end{abstract}

	%46B09 Probabilistic methods in Banach space theory
	%60G57 random measures
	%60J80 Branching processes
	%60G46 martingales and classical analysis
	%42A61 Probability methods for one variable harmonic analysis

	\maketitle

	\setcounter{tocdepth}{2}
	\tableofcontents

	\setcounter{tocdepth}{0}
	\setcounter{equation}{0}

	%\setcounter{secnumdepth}{2}
	%\setcounter{tocdepth}{2}

	%\vspace{0.1in}

\section{Introduction}\label{sec-intro}
This paper is the second part  (Part II) of a series of works on the harmonic analysis of multiplicative chaos measures.   Part II  is devoted to the multiplicative chaos on $d$-dimensional unit cube $[0,1]^d$ equipped with the Lebesgue measure as a background measure.  The treatment of multiplicative chaos on  general domains $U\subset \R^d$ equipped with abstract background measure will be included in Part III of this series.   Part II, although containing much more information than Part I, is still self-contained and thus can be read independently of Part I. However, a quick look at Part I will be helpful.

\subsection{Informal description of main results}
The main goal of the current paper  is to  provide a systematic development of the {\it vector-valued martingale method} for analyzing the polynomial Fourier decay of classical multiplicative chaos measures on $[0,1]^d$.   The vector-valued martingale method, initially introduced in \cite{CHQW24} in the setting of Mandelbrot cascades, and then adopted in the setting of one dimensional Gaussian multiplicative chaos  (GMC) in \cite{LQT24},  has been shown to be particularly powerful and straightforward in proving the polynomial Fourier decay of various models of multiplicative chaos:  as consequences, the exact Fourier dimensions for the canonical  Mandelbrot cascades and the classical sub-critical GMC on the unit interval have been established in \cite{CHQW24} and \cite{LQT24} respectively.  

The  ideas  in  \cite{CHQW24} and \cite{LQT24} can be put into a unified theorem (see Theorem~\ref{Fou-Dec-abs} below) for studying  sharp polynomial Fourier decay. This unified theorm is applicable in the classical models of multiplicative chaos  and there seems to be a general phenomenon for multiplicative chaos measures: 
\begin{center}
{\it The Fourier dimension coincides with the correlation dimension.}
\end{center}

More precisely,  as applications of Theorem~\ref{Fou-Dec-abs},  we establish the exact values or non-trivial lower bounds of the Fourier dimensions for classical  multiplicative chaos measures:
\begin{itemize}
\item 1D and 2D GMC: we obtain the exact values of Fourier dimensions of GMC measures on  $[0,1]^d$ for all sub-critical parameters $0<\gamma<\sqrt{2d}$ when $d\in \{1,2\}$ (see Theorem~\ref{thm-gen-12DGMC} below).
\item Higher dimensional GMC ($d\ge 3$): we obtain the exact values of  Fourier dimensions of GMC measures  on $[0,1]^d$   for   $1\leq\gamma<\sqrt{2d}$ when $d=3$  and  for  $\sqrt{2d}-\sqrt{2}\leq\gamma<\sqrt{2d}$ when $d \ge 4$;  for other sub-critical parameters $\gamma$, we obtain  positive Fourier dimensions (see Theorem~\ref{thm-gen-highD} below).
\item Mandelbrot random coverings (MRC):  we obtain the exact values of Fourier dimensions  of canonical MRC measures   and thus new constructions of random Salem sets; we obtain non-trivial lower bounds of the Fourier dimensions in the general setting (see  Theorems~\ref{thm-manrancov-poi} and \ref{salemcanonicalcase-poi} below). 
\item Poisson multiplicative chaos (PMC): we obtain  non-trivial lower bounds of the Fourier dimensions of PMC  measures (see Theorem~\ref{thm-poimulcha-pmc} below). 
\item Generalized Mandelbrot cascades: we introduce a model of generalized Mandelbrot cascades and obtain non-trivial lower bounds of the Fourier dimensions (see  Theorem~\ref{thm-selsim-cas} below). Our lower bounds are optimal due to the work \cite{CHQW24} on the canonical Mandelbrot cascades. 
\end{itemize}

\subsection{A unified approach for polynomial Fourier decay}
The main framework of this series of works is Kahane's $T$-martingale theory. However, in the setting of GMC measures, using a basic comparison result on kernels,   we can go beyond the $T$-martingale setting (see Theorems~\ref{thm-gen-12DGMC} and \ref{thm-gen-highD} below).

\subsubsection{General framework: Kahane's $T$-martingale theory}\label{sec-T-mart}
Let $d\geq1$ be an integer and consider  a sequence of independent stochastic processes indexed by points $\mathbf{t}$ in the $d$-dimensional unit cube $[0,1]^d$:
\begin{align}\label{def-Xm-abs}
	\{\XX_m(\mathbf{t}): \mathbf{t}\in [0,1]^d\}_{m\geq0}  \quad \text{with $\XX_m(\mathbf{t})\geq0$ and $\E[\XX_m(\mathbf{t})]\equiv 1$}.
\end{align}
For each integer $m\ge 0$, define a random measure on $[0,1]^d$ by
\begin{align}\label{def-num-abs}
	\mu_{m}(\mathrm{d}\mathbf{t}):=\Big[\prod_{j=0}^{m}\XX_{j}(\mathbf{t}) \Big] \mathrm{d}\mathbf{t}  \quad  \text{with $\mathrm{d}\mathbf{t}=\prod_{\beta=1}^{d}\mathrm{d}t_\beta$}.
\end{align}
Then $(\mu_{m})_{m\ge 0}$ is a measure-valued martingale with respect to the natural  filtration
\begin{align}\label{def-filtra-abs}
	\mathscr{G}_m:=\sigma\big(\XX_j:0\leq j\leq m\big),\quad m\geq 0.
\end{align}
By Kahane's $T$-martingale theory  \cite{Kah87a},  almost surely,  $\mu_{m}$ converges  in the sense of weak convergence of measures to a limiting random measure $\mu_{\infty}$ (called the associated multiplicative chaos measure):
\begin{align}\label{def-lim-mea-abs}
	\lim_{m\to \infty}\mu_{m}=\mu_{\infty}.
\end{align}

\subsubsection{Key assumptions and non-degeneracy of  multiplicative chaos measures}
We shall make the following key assumptions  on the stochastic processes in \eqref{def-Xm-abs}. Assumption~\ref{assum-indep-abs} is inspired by the independence in Mandelbrot cascade model; Assumption~\ref{assum-Lp0-abs} is inspired by the classical Mandelbrot-Kahane $L^p$-condition on the total mass of the Mandelbrot cascade measures;  Assumption~\ref{assum-alp0-abs} is a  natural condition allowing us to use  the $b$-adic-discrete-time approximation of the stochastic processes in question.  

Let $b\geq2$ be a fixed integer. For each integer $m\geq0$, let $\mathscr{D}_m^b$ denote the family of $b$-adic sub-cubes of $[0,1)^d$ of level/generation $m$: 
\begin{align}\label{def-b-dya-abs}
	\mathscr{D}_m^b:= \Big\{\mathbf{I}\subset [0,1)^d : \mathbf{I} = \prod_{\beta=1}^{d}\Big[\frac{h_\beta-1}{b^m}, \frac{h_\beta}{b^m}\Big) \, \text{with all $h_\beta\in \{1,2,\cdots,b^m\}$} \Big\}. 
\end{align}
For a stochastic process $\{\XX(\mathbf{t}): \mathbf{t}\in [0,1)^d\}$ and any $\mathbf{I}\in \mathscr{D}_m^b$, define  the restricted stochastic process
\[
\XX^\mathbf{I}: = \{\XX(\mathbf{t}): \mathbf{t}\in \mathbf{I}\}.
\]
Then for each $m\ge 0$, the stochastic process $\{\XX(\mathbf{t}): \mathbf{t}\in [0,1)^d\}$ can be naturally identified with a family of restricted stochastic  processes:
\[
\{\XX(\mathbf{t}): \mathbf{t}\in [0,1)^d\} \xlongequal{\text{\,\,identified with\,\,}}\{\XX^\mathbf{I}: \mathbf{I} \in \mathscr{D}_m^b\}.
\]

\begin{assumption}[Sub-exponential partition and independence inside sub-families]\label{assum-indep-abs}
	There exists an integer $k_0\geq0$ such that $\XX_k\not\equiv1$ for any $k>k_0$, and  the stochastic process 
	\[
	\{\XX_k(\mathbf{t}): \mathbf{t}\in [0,1)^d\} \xlongequal{\text{\,\,identified with\,\,}} \{\XX_k^{\mathbf{I}}: \mathbf{I}\in \mathscr{D}_{k-1}^b\}
	\]
	 can be partitioned into a sub-exponential number of sub-families,  each of which consists of independent  stochastic processes.  	That is,   there is a sequence $\{N_m\}_{m\geq0}$ of positive integers satisfying
\begin{align}\label{def-subexp-abs}
	N_m=o(b^{m\varepsilon})\quad\text{as $m\to\infty$ for any $\varepsilon>0$},
\end{align}
and for each $m\geq0$,  there is a partition of $\mathscr{D}_m^b$ into $N_m$-many sub-families, denoted by 
\begin{align}\label{dec-subexp-dec-abs}
	\mathscr{D}_m^b=\bigsqcup_{i=1}^{N_m}\mathscr{D}_{m,i}^{b},
\end{align}
such that,  for all $k>k_0$ and $1\le i\le N_{k-1}$,  the following stochastic processes are jointly independent:
\[
 \XX_{k}^{\mathbf{I}_1},  \, \XX_{k}^{\mathbf{I}_2}, \, \XX_{k}^{\mathbf{I}_3},  \cdots  (\text{with distinct $\mathbf{I}_1, \mathbf{I}_2, \mathbf{I}_3 \cdots \in \mathscr{D}_{k-1,i}^b$}).
\]
\end{assumption}

\begin{remark*}
In all applications in this paper, we only use bounded  sequences $\{N_m\}_{m\ge 0}$. However,  for further reference, we prove our main results in this generality of sub-exponential growth sequences $\{N_m\}_{m\ge 0}$. 
\end{remark*}

\begin{assumption}[Uniform  Mandelbrot-Kahane $L^{p}$-condition]\label{assum-Lp0-abs}
	There exists $p_0\in (1, 2]$  such that
	\[
	\sup_{\mathbf{t}\in[0,1)^d}\mathbb{E}[\XX_j(\mathbf{t})^{p_0}]<\infty   \,\, \text{for all $j\ge 0$} \anand 
	\limsup_{j\to\infty}\sup_{\mathbf{t}\in[0,1)^d}\mathbb{E}[\XX_j(\mathbf{t})^{p_0}]<b^{d(p_0-1)}.
	\]
\end{assumption}

\begin{assumption}[Uniform  $L^{p}$-H\"older condition]\label{assum-alp0-abs}
There exist $p_0\in (1,2]$ and  $\alpha_0\in(0,1]$ such that 
	\[
	\sup_{j\in \N}\sup_{\mathbf{I}\in\mathscr{D}_j^b}\sup_{\mathbf{t},\mathbf{s}\in\mathbf{I} \atop \mathbf{t}\neq\mathbf{s}}\mathbb{E}\Big[\Big| \frac{\XX_j(\mathbf{t})-\XX_j(\mathbf{s})}{b^{j\alpha_0}|\mathbf{t}-\mathbf{s}|^{\alpha_0}} \Big|^{p_0}\Big]<\infty,
	\]
	where $|\mathbf{t}-\mathbf{s}|=\big(\sum_{\beta=1}^{d}|t_\beta-s_\beta|^2\big)^{1/2}$ with $\mathbf{t}=(t_1,\cdots,t_d),\mathbf{s}=(s_1,\cdots,s_d)\in[0,1)^d$.
\end{assumption}

\begin{remark*}
  We shall emphasized that, Assumption~\ref{assum-alp0-abs}  should not be replaced by the following weaker one: 
\[
\limsup_{j\to\infty} \sup_{\mathbf{I}\in\mathscr{D}_j^b}\sup_{\mathbf{t},\mathbf{s}\in\mathbf{I} \atop \mathbf{t}\neq\mathbf{s}}\mathbb{E}\Big[\Big| \frac{\XX_j(\mathbf{t})-\XX_j(\mathbf{s})}{b^{j\alpha_0}|\mathbf{t}-\mathbf{s}|^{\alpha_0}} \Big|^{p_0}\Big]<\infty. 
\]
Indeed, the regularity of $P_j$  for every single $j\ge 0$ may affect the Fourier decay of the resulting multiplicative chaos measure $\mu_{\infty}$.  And,  Assumption~\ref{assum-alp0-abs}  should  not be replaced by the following stronger one: 
\[
\sup_{j\in \N}  \sup_{\mathbf{t},\mathbf{s}\in [0,1)^d \atop \mathbf{t}\neq\mathbf{s}}\mathbb{E}\Big[\Big| \frac{\XX_j(\mathbf{t})-\XX_j(\mathbf{s})}{b^{j\alpha_0}|\mathbf{t}-\mathbf{s}|^{\alpha_0}} \Big|^{p_0}\Big]<\infty, 
\]  
since, in the  setting of the canonical Mandelbrot cascades,  Assumption~\ref{assum-alp0-abs} is automatically satisfied, while the above stronger assumption may be violated. 
\end{remark*}

As usual, we say that the random measure $\mu_\infty$ in \eqref{def-lim-mea-abs} is  non-degenerate if $\PP(\mu_\infty\ne 0)>0$. Note that, uniform integrability of the martingale $\{\mu_m([0,1]^d)\}_{m\ge 0}$ implies the non-degeneracy of $\mu_\infty$. 

\begin{mainprop}[Non-degeneracy]\label{non-degene-abs}
Under Assumptions~\ref{assum-indep-abs} and \ref{assum-Lp0-abs}, for any $1<p\leq p_0$,
\begin{align}\label{Uni-Bou-01-abs}
	\sup_{m\geq0}\mathbb{E}\big[\big(\mu_{m}([0,1]^d)\big)^p \big]<\infty,
\end{align}
and consequently, the multiplicative chaos measure  $\mu_{\infty}$ is non-degenerate.
\end{mainprop}

\subsubsection{A unified theorem for polynomial Fourier decay}
Recall that for a finite positive Borel measure $\nu$ on the unit cube $[0,1]^d$, its Fourier transform is defined by
\[
\widehat{\nu}(\mathbf{x}):=\int_{[0,1]^d} e^{-2\pi i\mathbf{x}\cdot\mathbf{t}}\nu(\mathrm{d}\mathbf{t}),\quad\mathbf{x}=(x_1,\cdots,x_d)\in\mathbb{R}^d, \, \, \text{
where $\mathbf{x}\cdot\mathbf{t}=\sum_{\beta=1}^dx_\beta t_\beta$}.
\] 
A non-zero measure  $\nu$ is said to  have polynomial Fourier decay if there exists a constant $D>0$ such that  $|\widehat{\nu}(\mathbf{x})|^2 = O(|\mathbf{x}|^{-D})$ as $|\mathbf{x}|\to\infty$. The Fourier dimension of $\nu$ is defined by (see, e.g., \cite[Section~8.2]{BSS23} and \cite[Chapter~17]{Kah85b})
\[
\dim_F(\nu) := \sup \Big\{ D \in [0, d): |\widehat{\nu}(\mathbf{x})|^2 = O(|\mathbf{x}|^{-D}) \text{\,\,as\,\,} |\mathbf{x}|\to \infty\Big\}.
\]

Our main theorem is  Theorem~\ref{Fou-Dec-abs} below on the  Fourier dimension  of the multiplicative chaos measure.    The lower bound  in Theorem~\ref{Fou-Dec-abs} is sharp. Indeed, 
in the classical models of multiplicative chaos (including 1 D and 2D GMC, canonical Mandelbrot random coverings, canonical Mandelbrot cascades, etc), combined with the standard upper bound of  Fourier dimension by  correlation or Hausdorff dimensions: $\dim_F(\nu)\le \dim_2(\nu)\le \dim_H(\nu)$, Theorem~\ref{Fou-Dec-abs} provides the exact values of the Fourier dimensions.  

\begin{mainthm}[Polynomial Fourier decay]\label{Fou-Dec-abs}
	Under the Assumptions~\ref{assum-indep-abs}, \ref{assum-Lp0-abs} and \ref{assum-alp0-abs}, almost surely on $\{\mu_{\infty}\neq0\}$, the multiplicative chaos measure $\mu_{\infty}$ has polynomial Fourier decay with
	\[
	\dim_F(\mu_{\infty})\geq  \min\Big\{2 \alpha_0,\sup_{1<p\leq p_0}\frac{2  \Theta(p)}{p\log b}\Big\} >0,
	\]
	where $\Theta(p)$ is the structure function defined by 
\begin{align}\label{def-Theta-abs}
	\Theta(p):=d(p-1)\log b-\log\Big(\limsup_{j\to\infty}\sup_{\mathbf{t}\in[0,1)^d}\mathbb{E}[\XX_j(\mathbf{t})^p]\Big).
\end{align}
\end{mainthm}

\begin{remark*}
The definition \eqref{def-Theta-abs} of $\Theta(p)$  is closely related to the structure function in the theory of Mandelbrot cascades.  
\end{remark*}

\subsection{Gaussian multiplicative chaos}
We consider the Fourier dimensions of the GMC measures on $d$-dimensional unit cubes. The reader is referred to \cite{Kah85a, RV14, LRV15, BKNSW15, CN19, FJ19, GV23, BGKRV24} for more background and related results in Fourier decay of Gaussian multiplicative chaos. 

Informally, GMC measures arise as the exponential of log-correlated Gaussian fields.   
The original construction of GMC in   Kahane's seminal work \cite{Kah85a} will be briefly recalled in \S\ref{sec-sigma-reg}. Kahane's construction relies on a  well-known notion of the $\sigma$-positive type kernel, corresponding to  a special decomposition of the  kernel and thus a special decomposition of the related Gaussian field. Therefore,  Kahane's original GMC is a particular case of his $T$-martingale theory \cite{Kah87a}.    

However, the modern theory of GMC goes beyond the $\sigma$-positive type kernels and relies instead on various mollifications of the log-correlated Gaussian fields, see, e.g., \cite{RV14}. 

In Theorem~\ref{thm-gen-12DGMC} and  Theorem~\ref{thm-gen-highD} below, we shall deal with the Gaussian fields on $[0,1]^d$ with general log-correlated kernels  (not necessarily of $\sigma$-positive type): 
\[
K(\mathbf{t}, \mathbf{s}) = \log_{+}\Big( \frac{1}{|\mathbf{t}-\mathbf{s}|}\Big) + G(\mathbf{t}, \mathbf{s}), \quad \mathbf{t}, \mathbf{s}\in [0,1]^d,
\]
with $G$  being bounded and continuous and  $\log_{+}(x)= \max\{\log x, 0\}$.    We  emphasize that,  a priori,  the GMC measures associated to non-necessarily $\sigma$-positive type kernels do not fall into  the framework of $T$-martingale theory and thus one cannot  directly apply Theorem~\ref{Fou-Dec-abs}.  This difficulty is overcome as follows.   
 First, in the framework of Theorem~\ref{Fou-Dec-abs}, we shall  prove  Proposition~\ref{prop-kernel-ass} below on GMC measures associated to $\sigma$-positive type kernels.   Then,  by a simple comparison of GMC measures on the level of kernels (not the measure-theoretical level), we  derive Theorems~\ref{thm-gen-12DGMC} and  \ref{thm-gen-highD} from Proposition~\ref{prop-kernel-ass} and Theorem~\ref{Fou-Dec-abs}. 

In what follows,  for any integer $d\ge 1$ and any $\gamma\in (0, \sqrt{2d})$,  we define $D_{\gamma,d}\in(0,d)$ by
\begin{align}\label{def-D-gamma}
D_{\gamma,d}: =\left\{
	\begin{array}{cl}
		d-\gamma^2 & \text{if  $0<\gamma<\sqrt{2d}/2$}
		\vspace{2mm}
		\\
		(\sqrt{2d}-\gamma)^2 & \text{if $\sqrt{2d}/2\leq\gamma<\sqrt{2d}$}
	\end{array}\right..
\end{align}

We shall need the following auxiliary result: 
for any integer $d \ge 1$, there exists a continuous  function $f: [0, \infty) \rightarrow [0, \infty)$ with
\[
\supp(f) \subset [0,1], \quad  f(0) = 1 \anand  \sup_{v\in (0,1]} \frac{|f(v) - f(0)|}{v^{2}}<\infty, 
\]
such that the function $\R^d \ni \mathbf{t} \mapsto f(|\mathbf{t}|)$ is positive definite on $\R^d$. In what follows, we denote the class of such functions by $\mathscr{P}_d$: 
\begin{align}\label{def-Pd}
\mathscr{P}_d = \Big\{f: [0,\infty)\rightarrow [0, \infty)\, \big|\,  \text{$f$ satisfies all conditions above}\Big\}. 
\end{align}
 The fact that $\mathscr{P}_d \ne \varnothing$  is an immediate consequence of Lemma~\ref{lem-base-fn} below. Moreover,  Lemma~\ref{lem-base-fn} is proved by  explicit construction:  for any $h\in C^\infty(\R)$ with $h\not\equiv 0$ and $\supp(h)\subset [0,1/4]$,  taking the  self-convolution of $h(|\mathbf{x}|^2)$ and renormalization, we obtain a radial function $\mathbf{x}\mapsto f(|\mathbf{x}|)$ with $f\in \mathscr{P}_d$.

\subsubsection{1D and 2D GMC}
\begin{theorem}[Proof of Garban-Vargas conjecture]\label{thm-gen-12DGMC}
Let $d \in \{1, 2\}$ and consider a positive definite kernel 
\begin{align}\label{G-log}
K(\mathbf{t}, \mathbf{s}) = \log_{+} \Big( \frac{1}{|\mathbf{t}-\mathbf{s}|} \Big) + G(\mathbf{t}, \mathbf{s}), \quad \mathbf{t}, \mathbf{s}\in [0,1]^d,
\end{align}
with $G$  being bounded and continuous. Assume  that there exists  a function $f\in \mathscr{P}_d$ such that 
\begin{itemize}
\item either $G(\mathbf{t}, \mathbf{s})$ has the form
\[
G(\mathbf{t}, \mathbf{s}) = \int_{|\mathbf{t}-\mathbf{s}|}^1 \frac{f(v) -1}{v} \mathrm{d} v; 
\]
\item or more generally,  the function $G$ satisfies 
\begin{align}\label{error-cond}
 \sup_{\mathbf{t}, \mathbf{s}\in [0,1]^d
\atop \mathbf{t}\ne \mathbf{s}} \frac{|G(\mathbf{t},\mathbf{t}) -  G(\mathbf{s},\mathbf{s})|}{|\mathbf{t}-\mathbf{s}|}<\infty, \quad  \sup_{\mathbf{t}, \mathbf{s}\in [0,1]^d
\atop \mathbf{t}\ne \mathbf{s}} \frac{|G(\mathbf{t},\mathbf{t}) + G(\mathbf{s}, \mathbf{s})- 2 G(\mathbf{t},\mathbf{s})|}{|\mathbf{t}-\mathbf{s}|^2}<\infty
\end{align}
and  there is a constant $\lambda>0$ such that the following  kernel  is positive definite on $[0,1]^d \times [0,1]^d$:
\begin{align}\label{error-cond-bis}
R_\lambda(\mathbf{t}, \mathbf{s}) : = \lambda + G(\mathbf{t}, \mathbf{s}) - \int_{|\mathbf{t}-\mathbf{s}|}^1 \frac{f(v)-1}{v} \mathrm{d}v. 
\end{align}
\end{itemize}
Then for any $\gamma\in (0, \sqrt{2d})$, almost surely, the  sub-critical GMC measure $\mathrm{GMC}_{K}^\gamma$ has Fourier dimension 
\begin{align}\label{low-GMC-dim}
\dim_F(\mathrm{GMC}_{K}^\gamma) = D_{\gamma,d}. 
\end{align}
\end{theorem}

\begin{remark*}
One can easily check that the condition \eqref{error-cond} holds for any  $C^2$-smooth symmetric function 
\[
G\in C^2([0,1]^d \times [0,1]^d).
\]
 Moreover, for a fixed $\gamma \in (0, \sqrt{2d})$, for obtaining the almost sure equality \eqref{low-GMC-dim},  one may replace the  condition  \eqref{error-cond} by the following weaker one (recall that $D_{\gamma,d}\in (0,2)$ for $d\in \{1,2\}$ and $\gamma \in (0, \sqrt{2d})$): 
\[
 \sup_{\mathbf{t}, \mathbf{s}\in [0,1]^d
\atop \mathbf{t}\ne \mathbf{s}} \frac{|G(\mathbf{t},\mathbf{t}) -  G(\mathbf{s},\mathbf{s})|}{\sqrt{|\mathbf{t}-\mathbf{s}|^{D_{\gamma,d}}}}<\infty \anand \sup_{\mathbf{t}, \mathbf{s}\in [0,1]^d
\atop \mathbf{t}\ne \mathbf{s}} \frac{|G(\mathbf{t},\mathbf{t}) + G(\mathbf{s}, \mathbf{s})- 2 G(\mathbf{t},\mathbf{s})|}{|\mathbf{t}-\mathbf{s}|^{D_{\gamma,d}}}<\infty.
\]
Note also that the mere boundedness and continuity  of $G$ is inadequate for analyzing the exact polynomial Fourier decay of the GMC measure. Indeed, multiplication by a continuous density on a measure  can profoundly change its  Fourier transform's asymptotic behavior.   Therefore,  to ensure the Fourier dimension of the GMC measure remains invariant under the perturbation, further smoothness constraints on $G$ are indispensable.  
\end{remark*}

\begin{remark*}
When $d=1$,  for the exact log-kernel 
\[
K_{\mathrm{exact}}(t, s)= \log \frac{1}{|t- s|},  \quad  t,  s \in [0,1], 
\]
the almost sure equality \eqref{low-GMC-dim}  is proved in  \cite{LQT24} --the Part I of this series: there we used a slightly different method  via the white noise decomposition of $K_{\mathrm{exact}}$ in Bacry-Muzy \cite{BM03}.
\end{remark*}

\subsubsection{Higher dimensional GMC ($d\ge 3$)}
\begin{theorem}[Proof of Garban-Vargas conjecture for near-critical parameters $\gamma$]\label{thm-gen-highD}
For any integer  $d\ge 3$,  consider a positive definite kernel of the form \eqref{G-log} such that the assumptions in Theorem~\ref{thm-gen-12DGMC} are satisfied. Then for any $\gamma \in (0, \sqrt{2d})$,  almost surely,  the sub-critical GMC measure $\mathrm{GMC}_{K}^\gamma$ satisfies 
\[
0<\min\{2,D_{\gamma,d}\}\leq\mathrm{dim}_{F}(\mathrm{GMC}_{K}^\gamma)\leq D_{\gamma,d}<d.
\]
\end{theorem}

\begin{remark*}
Under the assumptions of Theorem~\ref{thm-gen-highD},  we obtain in particular  the following exact equalities: for  $d = 3$ and  $\gamma \in [1, \sqrt{2d}) = [1, \sqrt{6})$, almost surely, 
	\[
	\dim_F(\mathrm{GMC}_{K}^\gamma)  =\left\{
	\begin{array}{cl}
		d-\gamma^2 & \text{if  $1\le \gamma<\sqrt{2d}/2$}
		\vspace{2mm}
		\\
		(\sqrt{2d}-\gamma)^2 & \text{if $\sqrt{2d}/2\leq\gamma<\sqrt{2d}$}
	\end{array}\right.,
	\]
for $d\geq4$ and $\gamma \in [\sqrt{2d}-\sqrt{2}, \sqrt{2d})$, almost surely, 
	$
	\mathrm{dim}_{F}(\mathrm{GMC}_{K}^\gamma)= (\sqrt{2d}-\gamma)^2.
	$
\end{remark*}

\subsection{Mandelbrot random coverings}
The Mandelbrot random coverings (MRC), also known as the Poisson random coverings,  were first introduced by Mandelbrot  \cite{Man72} and further studied by many authors, e.g., \cite{She72, FFS85, Kah87b, Kah98, Fan02, SS18}. 

We briefly recall the main idea of MRC.  For any countable  subset   $\mathcal{Z}\subset \R\times \R_{+}$, define an open set by taking union of open intervals
\[
\mathcal{U}[\mathcal{Z}]: =  \bigcup_{z\in \mathcal{Z}} I_z, \quad \text{where the open interval $I_z : = (x, x+y)$ for any point $z = (x,y)\in \mathcal{Z}$}.
\]  
The theory of MRC aims to study the open set $\mathcal{U}(\mathcal{Z})$ or its complement  when  $\mathcal{Z}$ is sampled from a translation-invariant Poisson point process on the strip  
\[
 S := \mathbb{R}\times(0,1). 
 \]
 More precisely,  let $\mathrm{PPP} (\omega_\Lambda)$ be the translation-invariant Poisson point process on $\R\times (0,1)$ with intensity 
 \begin{align}\label{def-kappa}
 \omega_\Lambda(\mathrm{d}x\mathrm{d}y) := \mathrm{d}x  \otimes \Lambda(\mathrm{d}y),
 \end{align}
 where $\Lambda$ is a  Radon measure  on $(0,1)$ with infinite total mass $\Lambda((0,1)) = \infty$.   By translation-invariance,  the study  of the random set $\mathcal{U}[\mathrm{PPP} (\omega_\Lambda)]\subset \R$ can usually be reduced to a local position. Therefore,  define a closed random subset of the unit interval $[0,1]$ by  
\[
E_\Lambda: = [0,1]\setminus \mathcal{U}[\mathrm{PPP} (\omega_\Lambda)]=  [0,1]\setminus \bigcup_{z\in \mathrm{PPP} (\omega_\Lambda)} I_z.
\]
The random subset $E_\Lambda\subset [0,1]$ is called the   uncovered set of the Mandelbrot random covering.  We refer the reader to \cite{She72, FFS85} for more background and details of the uncovered set $E_\Lambda$.

The study of the above random set $E_\Lambda$  is closely related to multiplicative chaos. Indeed,  one may define a  multiplicative chaos measure with  support $E_\Lambda$ as follows. Given any $\epsilon\in (0,1)$, set
\[
S_\epsilon := \R\times (\epsilon, 1)
\]
and consider the closed random subset
\begin{align}\label{def-E-set}
E_\Lambda(\epsilon): = [0,1] \setminus \mathcal{U}[\mathrm{PPP} (\omega_\Lambda) \cap S_\epsilon]=[0,1]\setminus \bigcup_{z\in \mathrm{PPP} (\omega_\Lambda) \cap S_\epsilon} I_z.
\end{align}
Define a random measure on the unit interval $[0,1]$ by 
\begin{align}\label{def-muPCepsilon-poi}
	\mathrm{MRC}_\Lambda^\epsilon(\mathrm{d} t): =  \frac{\mathds{1}_{E_\Lambda(\epsilon)} (t) }{\E[\mathds{1}_{E_\Lambda(\epsilon)} (t) ]} \mathrm{d}t.
\end{align}
It turns out that this construction fits perfectly  with Kahane's $T$-martingale theory  \cite{Kah87a} and gives rise to  the Mandelbrot random convering measure  $\mathrm{MRC}_{\Lambda}$ on $[0,1]$ with 
\begin{align}\label{weacon-muPC-poi}
	\lim_{\epsilon\to0^+} \mathrm{MRC}_\Lambda^\epsilon =\mathrm{MRC}_\Lambda \anand  \supp(\mathrm{MRC}_\Lambda)= E_\Lambda. 
\end{align} 

To state our main results on $\mathrm{MRC}_\Lambda$ and $E_\Lambda$, we define two quantities: for any integer $b \ge 2$, set 
\begin{align}\label{def-chi-b}
\chi(b,\Lambda): = \limsup_{j\to \infty}\frac{1}{\log b}\int_{[b^{-j}, b^{-(j-1)})} y \Lambda(\mathrm{d}y)
\anand 
\chi(\Lambda):= \inf_{b\in \N_{\ge 2}}   \chi(b, \Lambda). 
\end{align}

Recall that for a compact subset $E\subset \R$, its Fourier dimension $\dim_{F}(E)$  is defined as 
\begin{align}\label{def-dimF-set}
\dim_F(E): = \sup\Big\{\dim_F(\nu): \text{$\nu$ is a Radon measure supported on $E$} \Big\}. 
\end{align}

\begin{theorem}\label{thm-manrancov-poi}
	Suppose that $
	\chi(\Lambda)<1,
	$
	then almost surely,  the random measure  $\mathrm{MRC}_{\Lambda}$ and the closed set $E_\Lambda$ have positive Fourier dimensions with
	\[
	\dim_F(E_\Lambda) \ge \dim_{F}(\mathrm{MRC}_{\Lambda})\geq1-\chi(\Lambda)>0.
	\]
\end{theorem}

\begin{remark*} 
	We note that Shmerkin and Suomala employed a different approach to derive lower bounds of the Fourier dimensions of Mandelbrot random coverings, as detailed in \cite[Corollary~7.5]{SS18}. 
\end{remark*}

The above lower bounds are sharp and provide the exact values of Fourier dimensions in the canonical MRC setting (see \cite{Fan02} for more background on the canonical case) with $\Lambda$ is given by
\begin{align}\label{canonicalcase-poi}
	\Lambda_\alpha:=\sum_{n=1}^{\infty}\delta_{\alpha/n}, \quad \text{where $\delta_x$ denoting the Dirac mass at $x$  and $\alpha \in (0,1)$}.
\end{align}

Recall  that the Hausdorff dimension (more precisely, the upper Hausdorff dimension) of a Radon measure $\nu$  on $\R$ is defined as
\begin{align}\label{def-Hdim-m}
\dim_{H}(\nu):=\inf\Big\{\dim_{H}(A):A\subset\mathbb{R} \an\nu(\mathbb{R}\setminus A)=0\Big\}.
\end{align}

\begin{theorem}\label{salemcanonicalcase-poi}
	Let $\alpha \in (0,1)$. Then  almost surely, $\mathrm{MRC}_{\Lambda_\alpha}$ is a Salem measure and $E_{\Lambda_\alpha}$ is a Salem set. More precisely, almost surely, we have
	\[
	\dim_{F}(\mathrm{MRC}_{\Lambda_\alpha})=\dim_{H}(\mathrm{MRC}_{\Lambda_\alpha})=\dim_{F}(E_{\Lambda_\alpha})=\dim_{H}(E_{\Lambda_\alpha})=1-\alpha.
	\]
\end{theorem}

\subsection{Poisson multiplicative chaos}
 The Poisson multiplicative chaos (PMC) were investigated in the study of random coverings,  see \cite{FK01, Fan02, BF05} for more details.
 
The construction of PMC is as follows.  For any $\epsilon\in (0,1)$ and $t\in[0,1]$, set
\begin{align}\label{def-Depsil-pmc}
	\mathscr{D}_{\epsilon}(t):=\Big\{(x,y)\in\mathbb{R}\times(0,1)\,\Big|\,\epsilon\leq y<1 \an t-y<x<t\Big\}.
\end{align}
Let $\mathrm{PPP} (\omega_\Lambda)$ be the Poisson point process on the strip $S=\R\times (0,1)$ with intensity $\omega_\Lambda$ given  in  \eqref{def-kappa}.  Fix any $a\in(0,1)$ and define a random measure on the unit interval $[0,1]$ by
\begin{align}\label{def-mu-a-pmc}
	\mathrm{PMC}_{\Lambda}^{a,\epsilon}(\mathrm{d}t):=\frac{a^{\#(\mathrm{PPP} (\omega_\Lambda)\cap\mathscr{D}_{\epsilon}(t))}}{\E[a^{\#(\mathrm{PPP} (\omega_\Lambda)\cap\mathscr{D}_{\epsilon}(t))}]}\mathrm{d}t,
\end{align}
where $\#(\mathrm{PPP} (\omega_\Lambda)\cap\mathscr{D}_{\epsilon}(t))$ denotes the cardinality of $\mathrm{PPP} (\omega_\Lambda)\cap\mathscr{D}_{\epsilon}(t)$.  Kahane's $T$-martingale theory \cite{Kah87a} then gives rise to  a limiting random measure $\mathrm{PMC}_{\Lambda}^{a}$ on $[0,1]$ with
\begin{align}\label{weacon-mu-a-pmc}
	\lim_{\epsilon\to0^+} \mathrm{PMC}_{\Lambda}^{a,\epsilon}=\mathrm{PMC}_{\Lambda}^{a}.
\end{align}
The random measure $\mathrm{PMC}_{\Lambda}^{a}$ is called the PMC measure associated with $\Lambda$ and $a$.

Recall the definitions \eqref{def-chi-b} for $\chi(b,\Lambda)$ and $\chi(\Lambda)$. 

\begin{theorem}\label{thm-poimulcha-pmc}
	Suppose that $\chi(\Lambda)<1$ and $a\in (0,1)$, 
	then almost surely,  the random measure  $\mathrm{PMC}_{\Lambda}^{a}$ has polynomial Fourier decay with
	\[
	\dim_{F}(\mathrm{PMC}_{\Lambda}^{a})\geq1- (1-a)^2   \chi(\Lambda) >0.
	\]
\end{theorem}

\subsection{Generalized Mandelbrot cascades}\label{S-cas-abs}
We now introduce a generalized model of Mandelbrot cascades, which   include  the canonical Mandelbrot cascades as the simplest particular examples. 

Let $\{\mathscr{W}(\mathbf{t}): \mathbf{t}\in [0,1]^d\}$ be a stochastic process indexed by points in $[0,1]^d$ satisfying 
\begin{align}\label{def-Wt-cas}
	\mathscr{W}(\mathbf{t})\geq0\anand\E[\mathscr{W}(\mathbf{t})]\equiv 1\text{ for any $\mathbf{t}\in[0,1]^d$}.
\end{align}
For any fixed integer $b\ge 2$,  recall the $b$-adic structure defined in \eqref{def-b-dya-abs}.  Based on the stochastic process \eqref{def-Wt-cas}, we construct a family  of jointly independent stochastic processes 
\begin{align}\label{def-WIt-cas}
	\Big\{\mathscr{W}_\mathbf{I}: \mathbf{I} \in \bigsqcup_{m=1}^{\infty} \mathscr{D}_m^b\Big\}
\end{align}
such that  for each $b$-adic sub-cube $\mathbf{I}\in\mathscr{D}_m^b$, the stochastic process $\mathscr{W}_{\mathbf{I}}(\mathbf{t})$ is indexed by points $\mathbf{t}\in\mathbf{I}$ and is an independent copy of $\{\mathscr{W}(\mathbf{t}): \mathbf{t}\in [0,1]^d\}$ up to a natural affine transform of $\mathbf{I} \rightarrow [0,1]^d$: 
\begin{align}\label{selsimproper-cas}
	\big\{\mathscr{W}_{\mathbf{I}}(\mathbf{t}): \mathbf{t} \in \mathbf{I}\big\}   \xlongequal{\text{in distribution}} \big\{\mathscr{W} \big(b^m(\mathbf{t}-\ell_{\mathbf{I}})\big):  \mathbf{t}\in\mathbf{I}  \big\},
\end{align}
where $\ell_{\mathbf{I}}$ is the minimum vertex of $\mathbf{I}$ defined by 
\begin{align}\label{ell-I-cas}
	\ell_{\mathbf{I}}:=\Big(\frac{h_1-1}{b^{m}},\frac{h_2-1}{b^{m}},\cdots,\frac{h_d-1}{b^{m}}\Big)\quad\text{if}\quad\mathbf{I} = \prod_{\beta=1}^{d}\Big[\frac{h_\beta-1}{b^{m}}, \frac{h_\beta}{b^{m}}\Big)\in\mathscr{D}_{m}^b.
\end{align}

If we define 
\begin{align}\label{def-XWm-cas}
	\XX_{\mathscr{W},0}(\mathbf{t}):\equiv1\anand\XX_{\mathscr{W},m}(\mathbf{t}):=\sum_{\mathbf{I}\in\mathscr{D}_{m}^b}\mathscr{W}_{\mathbf{I}}(\mathbf{t})\mathds{1}_{\mathbf{I}}(\mathbf{t})\text{ for each $m\geq1$},
\end{align}
then we obtain   a sequence of independent stochastic processes indexed by $\mathbf{t}\in[0,1]^d$ such that
\begin{align}\label{def-Xm-cas}
	\{\XX_{\mathscr{W},m}(\mathbf{t}): \mathbf{t}\in [0,1]^d\}_{m\ge 0}  \quad \text{with $\XX_{\mathscr{W},m}(\mathbf{t})\geq0$ and $\E[\XX_{\mathscr{W},m}(\mathbf{t})]\equiv 1$}.
\end{align}
Then by applying Kahane's construction \eqref{def-num-abs} and \eqref{def-lim-mea-abs} as in \S\ref{sec-T-mart},  we  obtain a generalized Mandelbrot cascade measure $\mathrm{MC}_\mathscr{W}^b$ on the unit cube $[0,1]^d$ associated to  \eqref{def-Xm-cas} and the background measure $\mathrm{d} \mathbf{t}$: 
\begin{align}\label{def-geneMC-cas}
	\mathrm{MC}_\mathscr{W}^b(\mathrm{d}\mathbf{t}): = \lim_{m\to\infty}  \Big[\prod_{j=0}^{m}\XX_{\mathscr{W}, j}(\mathbf{t}) \Big] \mathrm{d}\mathbf{t}.
\end{align}

\begin{theorem}\label{thm-selsim-cas}
Let $\{\mathscr{W}(\mathbf{t}): \mathbf{t}\in [0,1]^d\}$ be given as \eqref{def-Wt-cas} such that  for some $p_0\in (1,2]$ and $\alpha_0\in(0,1]$, 
\begin{align}\label{condi-thm-selsim-cas}
	\sup_{\mathbf{t}\in[0,1)^d}\E[\mathscr{W}^{p_0}(\mathbf{t})]<b^{d(p_0-1)}\anand\sup_{\mathbf{t}, \mathbf{s}\in[0,1)^d, \mathbf{t}\neq\mathbf{s}}\mathbb{E}\Big[\Big|\frac{\mathscr{W}(\mathbf{t})-\mathscr{W}(\mathbf{s})}{|\mathbf{t}-\mathbf{s}|^{\alpha_0}} \Big|^{p_0}\Big]<\infty.
\end{align}
Then the generalized Mandelbrot cascade measure $\mathrm{MC}_\mathscr{W}^b$ is non-degenerate. Moreover,   almost surely on $\{ \mathrm{MC}_\mathscr{W}^b \neq0\}$, the random measure  $\mathrm{MC}_\mathscr{W}^b$ has polynomial Fourier decay with
	\[
	\dim_{F}(\mathrm{MC}_\mathscr{W}^b)\geq \min\Big\{2 \alpha_0,\sup_{1<p\leq p_0 }  \Big[  2d \big(1- \frac{1}{p}) - 2\log_b \Big( \sup_{\mathbf{t}\in[0,1)^d}  \big(\mathbb{E}[\mathscr{W}^{p}(\mathbf{t})]\big)^{\frac{1}{p}}\Big)   \Big]\Big\}>0.
	\]
\end{theorem}

When $\mathscr{W}(\mathbf{t})$ is a random constant function, that is, $\mathscr{W}(\mathbf{t})=W\not\equiv1$ is a random variable satisfying $W\geq0$ and $\E[W]=1$, then the  random measure $\mathrm{MC}_\mathscr{W}^b$ becomes the canonical Mandelbrot cascade  measure on  $[0,1]^d$ induced by the random variable $W$.  In this case, we denote $\mathrm{MC}_\mathscr{W}^b$ as 
\[
\mathrm{MC}_W^b. 
\]
We refer the reader to \cite{Man74a, Man74b, KP76, Kah93, CLS24, CHQW24} for more background and some related results in Fourier decay of the canonical Mandelbrot cascade measures.

\begin{corollary}\label{MC-FouDec-cas}
	Suppose that $\mathbb{E}[W\log W]<d\log b$ and $\E[W^{1+\varepsilon}]<\infty$ for some $\varepsilon>0$, then almost surely on $\{\mathrm{MC}_W^b\neq0\}$,  the canonical Mandelbrot cascade $\mathrm{MC}_W^b$ has polynomial Fourier decay with 
	\[
	\dim_{F}(\mathrm{MC}_W^b)\geq\min\Big\{2,\sup_{1<p\leq p_0} \Big[ 2d \big(1-\frac{1}{p}\big) -   2 \log_b \Big(\big(\E[W^p]\big)^{\frac{1}{p}}\Big)   \Big]   \Big\}> 0,
	\]
	here $p_0\in (1,2]$ is any exponent satisfying $\E[W^{p_0}]<b^{d(p_0-1)}$.
\end{corollary}

\begin{remark*}
From the works \cite{CLS24, CHQW24},  the lower bound of the Fourier dimension  obtained in Corollary~\ref{MC-FouDec-cas} is sharp.  Moreover, Corollary~\ref{MC-FouDec-cas}  implies that  the assumption $\E[W^p]<\infty$ for all $p>0$ used in \cite{CHQW24} can be relaxed to the weaker assumption $\E[W^2]<\infty$. 
\end{remark*}

The following special case was suggested to us by Prof. Xinxin Chen.  Let $d=1$. For any $\sigma>0$, let $\mathscr{W}_\sigma(t)$ be the geometric Brownian motion:
\begin{align}\label{def-GBM}
\mathscr{W}_\sigma(t)=\exp\big(\sigma \mathrm{B}(t)-\frac{\sigma^2}{2}t\big), \quad t\in [0,1],
\end{align}
where $\mathrm{B}(t)$ is the standard Brownian motion starting from the origin $0$.
  
\begin{corollary}\label{GBM-FouDec-cas}
	For any $\sigma \in(0,\sqrt{2\log b})$, the generalized Mandelbrot cascade measure $\mathrm{MC}_{\mathscr{W}_\sigma}^b$ is non-degenerate and has polynomial Fourier decay almost surely. Moreover, almost surely on $\{\mathrm{MC}_{\mathscr{W}_\sigma}^b\neq0\}$,
    \begin{align}\label{def-D-sigma}
    \dim_{F}(\mathrm{MC}_{\mathscr{W}_\sigma}^b)\geq  D_\sigma : =  \left\{
    \begin{array}{cl}
    	1-\frac{\sigma^2}{\log b} & \text{if  $0<\sigma<\frac{\sqrt{2\log b}}{2}$}
    	\vspace{2mm}
    	\\
    	(\sqrt{2}-\frac{\sigma}{\sqrt{\log b}})^2 & \text{if $\frac{\sqrt{2\log b}}{2}\leq \sigma<\sqrt{2\log b}$}
    \end{array}\right..
    \end{align}
\end{corollary}

\subsection{Organization of the paper}
In the introduction  \S\ref{sec-intro}, we state the main results of this paper: the unified results Proposition~\ref{non-degene-abs} on non-degeneracy and  Theorem~\ref{Fou-Dec-abs} on polynomial Fourier decay of the multiplicative chaos; Theorem~\ref{thm-gen-12DGMC} and Theorem~\ref{thm-gen-highD} on the Garban-Vargas conjecture of Gaussian multiplicative chaos; Theorem~\ref{thm-manrancov-poi} and Theorem~\ref{salemcanonicalcase-poi} on the Fourier dimensions  of  Mandelbrot random coverings;  Theorem~\ref{thm-poimulcha-pmc}  and    Theorem~\ref{thm-selsim-cas}   on the Fourier decay  of Poisson multiplicative chaos  and generalized Mandelbrot cascades respectively.   

In \S\ref{S-metvecvalmar-abs} below, we outline the proof of Theorem~\ref{Fou-Dec-abs}. In particular, we recall the Pisier's martingale type inequalities and outline the application of the vector-valued martingale method in the study of Fourier decay of multiplicative chaos.  

In \S\ref{S-dDGMC-abs}, we give the proofs of Theorem~\ref{thm-gen-12DGMC} and Theorem~\ref{thm-gen-highD} for GMC measures.  The main strategy is as follows. We first introduce a notion of $\sigma$-regular kernels  and construct  $*$-scale invariant  $\sigma$-regular kernels.  Then in Proposition~\ref{prop-kernel-ass},  we prove  Assumptions~\ref{assum-indep-abs}, \ref{assum-Lp0-abs} and \ref{assum-alp0-abs} for the exponential Gaussian processes associated with the Gaussian fields with $\sigma$-regular kernels.   Finally, Theorem~\ref{thm-gen-12DGMC} and Theorem~\ref{thm-gen-highD} will be proved as consequences of a special comparison result for GMC measures  (see \S\ref{sec-compare}) and  of  Proposition~\ref{prop-kernel-ass} and Theorem~\ref{Fou-Dec-abs}. 
 
In \S\ref{S-poi-abs}, we derive Theorem~\ref{thm-manrancov-poi} and Theorem~\ref{salemcanonicalcase-poi} for MRC measures from Theorem~\ref{Fou-Dec-abs} . 
  
In \S\ref{S-pmc-abs}, we derive Theorem~\ref{thm-poimulcha-pmc} for PMC measures from Theorem~\ref{Fou-Dec-abs}. 
  
In \S\ref{S-genercas-abs}, we derive Theorem~\ref{thm-selsim-cas} for generalized MC measures from Proposition~\ref{non-degene-abs} and  Theorem~\ref{Fou-Dec-abs}.   And then we prove Corollary~\ref{MC-FouDec-cas} and Corollary~\ref{GBM-FouDec-cas} from Theorem~\ref{Fou-Dec-abs}. 

The detailed proof of Proposition~\ref{non-degene-abs}  is  given in \S\ref{S-Pf-Thm-Nondegen-abs}.

The detailed proof of  our unified Theorem~\ref{Fou-Dec-abs} is given in  \S\ref{S-Pf-Thm-PolyFoudec-abs}. The schematic graph of the proof of Theorem~\ref{Fou-Dec-abs}  is given in Figure~\ref{figure-procedure-proving}.

\subsection*{Acknowledgements} We would like to thank Prof. Xinxin Chen for useful discussion on the model in Corollary~\ref{GBM-FouDec-cas} of the generalized Mandelbrot cascades relating to  the geometric Brownian motion. 
YQ is supported by National Natural Science Foundation of China (NSFC No. 12471145).

\section{Outline of the proof of Theorem~\ref{Fou-Dec-abs}}\label{S-metvecvalmar-abs}
In this section,  we briefly outline the proof of our unified Theorem~\ref{Fou-Dec-abs} via vector-valued martingale method, combined with Littlewood-Paley theory. We will prove Theorem~\ref{Fou-Dec-abs} in \S\ref{S-Pf-Thm-PolyFoudec-abs} following the procedure in Figure \ref{figure-procedure-proving}:
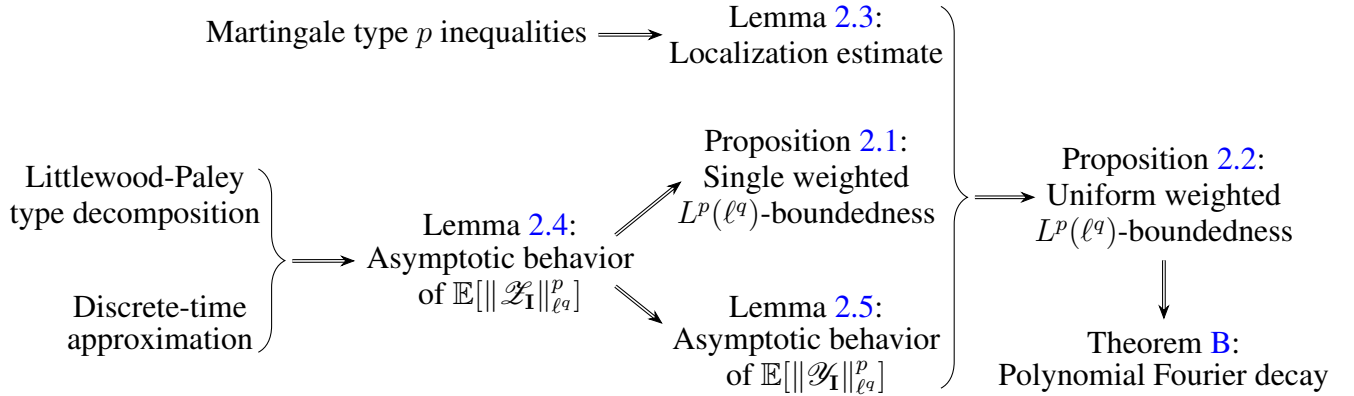
\begin{figure}[h]
	\begin{center}
		\begin{tikzpicture}[node distance=1.5cm and 2cm, >=Stealth]
\node (B) at (-2.57,2.4375) {Martingale type $p$ inequalities};
\draw[->, double] (B) -- (0.95,2.4375);
\node (E1) at (2.85,2.675) {Lemma~\ref{dec-vm-abs}:};
\node (E2) at (2.85,2.2) {Localization estimate};
\node (C1) at (2.85,1) {Proposition~\ref{lem-lq-abs}:};
\node (C2) at (2.85,0.525) {Single weighted};
\node (C3) at (2.85,0.05) {$L^p(\ell^q)$-boundedness};
\node (D1) at (2.85,-1.15) {Lemma~\ref{UB-ZW-abs}:};
\node (D2) at (2.85,-1.625) {Asymptotic behavior};
\node (D3) at (2.85,-2.1) {of $\mathbb{E}[\|\mathscr{Y}_\mathbf{I} \|_{\ell^q}^p]$};
\draw[decorate, decoration={brace, amplitude=10pt}] (4.63,2.8125) -- (4.63,-2.2375);
\node (F1) at (7.58,0.7625) {Proposition~\ref{Uni-Bou-Lplq-abs}:};
\node (F2) at (7.58,0.2875) {Uniform weighted};
\node (F3) at (7.58,-0.1875) {$L^p(\ell^q)$-boundedness};
\draw[->, double] (5.08,0.2875) -- (F2);
\node (G1) at (7.58,-1.625) {Theorem~\ref{Fou-Dec-abs}:};
\node (G2) at (7.58,-2.1) {Polynomial Fourier decay};
\draw[->, double] (F3) -- (G1);
\node (A1) at (-1.2,-0.075) {Lemma~\ref{UB-YZW-abs}:};
\node (A2) at (-1.2,-0.55) {Asymptotic behavior};
\node (A3) at (-1.2,-1.025) {of $\mathbb{E}[\|\mathscr{Z}_\mathbf{I} \|_{\ell^q}^p]$};
\draw[->, double] (0.325,-0.2125) -- (1.2,0.5);
\draw[->, double] (0.325,-0.8875) -- (0.95,-1.3875);
\node (H1) at (-6.03,0.525) {Littlewood-Paley};
\node (H2) at (-6.03,0.05) {type decomposition};			
\node (J1) at (-5.7,-1.15) {Discrete-time};
\node (J2) at (-5.7,-1.625) {approximation};
\draw[decorate, decoration={brace, amplitude=10pt}] (-4.4,0.6625) -- (-4.4,-1.7625);
\draw[->, double] (-3.95,-0.555) -- (A2);
		\end{tikzpicture}
	\end{center}
	\caption{The procedure for proving Theorem~\ref{Fou-Dec-abs}.}
	\label{figure-procedure-proving}
\end{figure}

\subsection{Pisier's martingale type inequalities}	
	We shall use the following well-known fact in the theory of Banach space geometry (see \cite[Proposition~10.36 and Definition~10.41]{Pis16}): 
	\begin{center}
		{\it  For any  $2\le q<\infty$, the Banach space $\ell^q$ has martingale type $p$ for all $1<p\le 2$. }
	\end{center}
	
	More precisely,  for any  $1<p\leq2\leq q<\infty$, there exists a constant $C(p, q)>0$ such that any $\ell^q$-valued martingale  
	$(F_m)_{m\ge 0}$ in $L^p(\PP; \ell^q)$ satisfies
	\begin{align}\label{def-Mtype}
		\E [ \| F_m\|_{\ell^q}^p] \le  C(p, q) \sum_{k =0}^m   \E[ \|F_k- F_{k-1}\|_{\ell^q}^p]
	\end{align}
	with the convention $F_{-1}\equiv 0$. 	The inequality \eqref{def-Mtype} implies in particular that for any family of independent and  centered  $\ell^q$-valued random variables $(G_k)_{k=0}^m$  in $L^p(\PP; \ell^q)$, 
	\begin{align}\label{def-ind-Mtype}
		\E\Big[ \Big\|\sum_{k=0}^m G_k \Big\|_{\ell^q}^p\Big] \le C(p, q)   \sum_{k =0}^m \E[\|G_k \|_{\ell^q}^p].  
	\end{align} 
	
	For further reference, note that when $1<p\leq2$, if $(T_m)_{m\ge 0}$ is a scalar martingale in $L^p(\PP)$, the inequality \eqref{def-Mtype} follows from the classical Burkholder's martingale inequality and reads as 
	\begin{align}\label{def-Mtype-scalar}
		\E [ | T_m|^p] \le  C(p) \sum_{k =0}^m   \E[ |T_k- T_{k-1}|^p]
	\end{align}
	with the convention $T_{-1}\equiv 0$.  Moreover,  the inequality \eqref{def-Mtype-scalar} implies that for any family of independent and  centered random variables $(X_k)_{k=0}^m$  in $L^p(\PP)$, 
	\begin{align}\label{def-ind-Mtype-scalar}
		\E\Big[ \Big|\sum_{k=0}^m X_k \Big|^p\Big] \le C(p)   \sum_{k =0}^m \E[|X_k |^p].  
	\end{align}

\subsection{Vector-valued martingale method}\label{subS-metvecvalmar-abs}
Given a Radon measure $\nu$ on $[0,1]^d$, it is well-known from  Kahane \cite[Chapter~17, Lemma~1]{Kah85b} (see also a convenient form in \cite[Section~1.3]{CHQW24}) that one may reduce the study of the polynomial decay of the Fourier transform of $\widehat{\nu}(\mathbf{x})$ as $|\mathbf{x}| \to\infty$ to that of its Fourier coefficients  $\widehat{\nu}(\mathbf{n})$  for $\mathbf{n}  = (n_1, \cdots, n_d) \in \Z^d$ as $|\mathbf{n}| = (\sum_{\beta=1}^d n_\beta^2)^{1/2}\to\infty$.  

To study the Fourier decay of the  multiplicative chaos measure $\mu_\infty$ defined by \eqref{def-lim-mea-abs} as in \S \ref{sec-T-mart},  we shall study the Fourier transform of the associated  approximating measure-valued martingale sequence  $(\mu_m)_{m\ge 0}$ defined in \eqref{def-num-abs}.   Thanks to Kahane's reduction to Fourier coefficients, we may consider only the restriction $\widehat{\mu_m}|_{\Z^d}$.  The key step in proving the unified Theorem~\ref{Fou-Dec-abs} is the following Proposition~\ref{Uni-Bou-Lplq-abs} on the uniform weighted $L^p(\ell^q)$-boundedness of the weighted $\ell^q$-valued martingale 
\begin{align}\label{key-mart}
 \big\{\widehat{\mu_{m}}|_{\Z^d}\big\}_{m\ge 0} = \big\{(\widehat{\mu_{m}} (\mathbf{n}))_{\mathbf{n}\in \Z^d}\big\}_{m\ge 0} \quad \text{with $\widehat{\mu_m}(\mathbf{n}) = \int_{[0,1]^d}e^{-2\pi i\mathbf{n}\cdot\mathbf{t}}\mu_{m}(\mathrm{d}\mathbf{t})$ for $\mathbf{n}\in\Z^d$}. 
\end{align}

\subsubsection{Uniform weighted $L^p(\ell^q)$-boundedness}
For simplifying notation, in what follows, set 
\begin{align}\label{def-LF-abs}
	\mathcal{L}_F:=\min\Big\{2 \alpha_0,\sup_{1<p\leq p_0}\frac{2  \Theta(p)}{p\log b}\Big\},
\end{align}
where $\Theta(p)$ is given in \eqref{def-Theta-abs}.  Under Assumption~\ref{assum-Lp0-abs}, we have  $\mathcal{L}_F>0$. See Lemma~\ref{lem-LF} below for details. 

\begin{proposition}[Single weighted $L^p(\ell^q)$-boundedness]\label{lem-lq-abs}
	Under the Assumptions~\ref{assum-Lp0-abs} and \ref{assum-alp0-abs}, for any fixed $\tau\in(0,\mathcal{L}_F)$ and any fixed integer $m\geq0$, if $1<p\leq p_0$ and $q>\frac{2d}{2\alpha_0-\tau}$, then
	\[
	\mathbb{E}\Big[\Big\{\sum_{\mathbf{n}\in\mathbb{Z}^{d}}\Big||\mathbf{n}|^{\tau/2}\widehat{\mu_{m}}(\mathbf{n})\Big|^q\Big\}^{p/q}\Big]<\infty.
	\]
\end{proposition}

In other words, for any $\tau\in (0, \mathcal{L}_F)$ and $p, q$ as in Proposition~\ref{lem-lq-abs},  we obtain the following $\ell^q$-valued martingale with finite $p$-moments (for simplifying notation, the subscript $\tau$ will be  omitted): 
\begin{align}\label{vec-val-mar-abs}
\{\mathcal{M}_{m}\}_{m\ge 0}=\{\mathcal{M}_{\tau,m} \}_{m\ge 0}:= \big\{ (|\mathbf{n}|^{\tau/2}\widehat{\mu_{m}}(\mathbf{n}))_{\mathbf{n}\in\mathbb{Z}^{d}}\big\}_{m\ge 0}.
\end{align}
For obtaining Theorem~\ref{Fou-Dec-abs}, we shall establish the uniform $L^p(\ell^q)$-boundedness of the $\ell^q$-valued martingale $\{\mathcal{M}_{m}\}_{m\ge 0}$ in the following Proposition~\ref{Uni-Bou-Lplq-abs}.

\begin{proposition}[Uniform weighted $L^p(\ell^q)$-boundedness]\label{Uni-Bou-Lplq-abs}
	Under the Assumptions~\ref{assum-indep-abs}, \ref{assum-Lp0-abs} and \ref{assum-alp0-abs}, for any fixed $\tau\in(0,\mathcal{L}_F)$, there exist $p$ and $q$ satisfying $1<p\leq p_0 \leq\max\{2,\frac{2d}{2\alpha_0-\tau}\}<q<\infty$ such that 
	\begin{align}\label{sup-Lplq-abs}
		\sup_{m\geq0}\mathbb{E}[\|\mathcal{M}_{m}\|_{\ell^{q}}^{p}]  = \sup_{m\ge 0} 	\mathbb{E}\Big[\Big\{\sum_{\mathbf{n}\in\mathbb{Z}^{d}}\Big||\mathbf{n}|^{\tau/2}\widehat{\mu_{m}}(\mathbf{n})\Big|^q\Big\}^{p/q}\Big]<\infty.
	\end{align}
\end{proposition}

\begin{remark*}
	For the exponents $\tau,p, q$ as in Proposition \ref{Uni-Bou-Lplq-abs}, consider the  weighted $\ell^q$-sequence space 
	\[
		\ell^q_\tau(\Z^d): = \Big\{ (c_\mathbf{n})_{\mathbf{n}\in \Z^d}\,\Big|\, c_\mathbf{n} \in \C \an \| (c_\mathbf{n})\|_{\ell_\tau^q}: = \Big\{\sum_{\mathbf{n}\in \Z^d}  \Big| |\mathbf{n}|^{\tau/2} c_\mathbf{n}\Big|^q  \Big\}^{1/q} <\infty \Big\}.
	\]
	Then the inequality \eqref{sup-Lplq-abs} means that the  martingale \eqref{key-mart} is uniformly bounded in $L^p(\ell^q_\tau(\Z^d))$: 
	\[
	\sup_{m\geq0}\mathbb{E}\big[\big\|\widehat{\mu_m}|_{\Z^d}\big\|_{\ell_\tau^{q}}^{p}\big]<\infty.
	\]
\end{remark*}

\subsubsection{Localization estimate via a random Fourier decoupling inequality}   The main idea in the first step towards Proposition~\ref{Uni-Bou-Lplq-abs} is essentially a global-to-local  estimate (or a localization estimate) of the Fourier transform of $\widehat{\mu_\infty}$ or of $\widehat{\mu_m}$, and Pisier's martingale type $p$ inequality for the Banach space $\ell^q$ plays an important role  in this step.  That is, to estimate the global Fourier transform $\widehat{\mu_m}(\mathbf{n})$, we shall consider the following random quantities  (see \eqref{def-Y-abs} and \eqref{def-Z-abs} below for the definitions of the precise quantities that will be used later) obtained as local Fourier transforms: 
\[
\int_{\mathbf{I}}\Big[\prod_{j=0}^{m}\XX_j(\mathbf{t})\Big]  e^{-2\pi i\mathbf{n}\cdot\mathbf{t}}\mathrm{d}\mathbf{t} \anand  \int_{\mathbf{I}}\Big[\prod_{j=0}^{k-1}\XX_j(\mathbf{t})\Big]  \mathring{\XX}_{k} (\mathbf{t}) e^{-2\pi i\mathbf{n}\cdot\mathbf{t}}\mathrm{d}\mathbf{t} 
\]
with $\mathring{\XX}_{k} = \XX_k -1$ and $\mathbf{I}$ is some $b$-adic sub-cube of $[0,1)^d$ with a size related to $m$ or $k$ in a certain way. 

\begin{remark*}
We find some similarity (probably in a dual form) of our estimate to the decoupling inequalities in harmonic analysis.  
\end{remark*}

We now turn to the precise statement of our localization estimate. 

Recall the definition of the natural filtration $(\mathscr{G}_k)_{k\ge 0}$  defined in \eqref{def-filtra-abs}.  For any integer $k\ge 1$ and  any $b$-adic sub-cube $\mathbf{I}\in \mathscr{D}_{k-1}^b$ defined in \eqref{def-b-dya-abs},  we define a $\mathscr{G}_k$-measurable (but not $\mathscr{G}_{k-1}$-measurable) random vector $\mathscr{Y}_{\mathbf{I}} =\mathscr{Y}_{\tau,\mathbf{I}}:= \big(\mathscr{Y}_\mathbf{I}(\mathbf{n})\big)_{\mathbf{n}\in\mathbb{Z}^{d}}$ by
\begin{align}\label{def-Y-abs}
	\mathscr{Y}_\mathbf{I}(\mathbf{n}):= |\mathbf{n}|^{\tau/2}\int_{\mathbf{I}}\Big[\prod_{j=0}^{k-1}\XX_j(\mathbf{t})\Big]  \mathring{\XX}_{k} (\mathbf{t}) e^{-2\pi i\mathbf{n}\cdot\mathbf{t}}\mathrm{d}\mathbf{t},
\end{align}
where 
\begin{align}\label{def-rin-X-abs}
	\mathring{\XX}_{k}(\mathbf{t}):=\XX_{k}(\mathbf{t})-\E[\XX_{k}(\mathbf{t})]=\XX_{k}(\mathbf{t})-1. 
\end{align}

\medskip
{\flushleft\bf Alarming:}  One should note that, for each $\mathbf{I}\in \mathscr{D}_{k-1}^b$, the random vector $\mathscr{Y}_\mathbf{I}$  defined in \eqref{def-Y-abs} is $\mathscr{G}_k$-measurable, but is not $\mathscr{G}_{k-1}$-measurable.  It is worthwhile to note that, by condition \eqref{def-Xm-abs} of the independent stochastic processes $\{\XX_j\}_{j\ge 0}$ and the definition \eqref{def-rin-X-abs}, we have
\begin{align}\label{Y-I-cond-abs}
	\E[\mathscr{Y}_\mathbf{I}|\mathscr{G}_{k-1}] =0. 
\end{align}
The observation  \eqref{Y-I-cond-abs} is crucial in the proof of the following localization estimate in Lemma~\ref{dec-vm-abs}.

\begin{lemma}[Localization estimate]\label{dec-vm-abs}
	Under the Assumption~\ref{assum-indep-abs}, for any fixed $\tau\in(0,\mathcal{L}_F)$ and $p$, $q$ satisfying $1<p\leq p_0 \leq\max\{2,\frac{2d}{2\alpha_0-\tau}\}<q<\infty$, there exists a constant $C=C(d,b,\tau,p,q)>0$ such that for any $m>k_0$ (here $k_0$ is the integer given as in  Assumption~\ref{assum-indep-abs}),
	\begin{align}\label{local-ineq}
	\mathbb{E}[\|\mathcal{M}_{m}\|_{\ell^q}^p]\leq C \E[\|\mathcal{M}_{k_0}\|_{\ell^q}^p]  +  C \sum_{k=k_0+1}^{m}N_{k-1}^{p-1}\sum_{\mathbf{I}\in\mathscr{D}_{k-1}^b}\mathbb{E}[\|\mathscr{Y}_\mathbf{I} \|_{\ell^q}^p].
	\end{align}
\end{lemma}

\subsubsection{Key lemmas to the localization estimates}
By the localization estimate in Lemma~\ref{dec-vm-abs}, 
to complete the proof of Proposition~\ref{Uni-Bou-Lplq-abs},   the following two tasks should be accomplished:
\begin{itemize}
\item   prove Proposition~\ref{lem-lq-abs} to ensure that  $\E[\|\mathcal{M}_{k_0}\|_{\ell^q}^p]<\infty$;
\item   obtain an appropriate asymptotic behavior of $\mathbb{E}[\|\mathscr{Y}_\mathbf{I} \|_{\ell^q}^p]$ (as the Lebesgue measure $|\mathbf{I}|\to 0$) to ensure the summability of the series on the right hand side of \eqref{local-ineq}.
\end{itemize}
It turns out that the above two tasks can be unified into essentially one single task. Indeed,  the proof of $\E[\|\mathcal{M}_{k_0}\|_{\ell^q}^p]<\infty$ relies on   the asymptotic behavior of $\mathbb{E}[\|\mathscr{Z}_\mathbf{I} \|_{\ell^q}^p]$  for the random vectors $\mathscr{Z}_\mathbf{I}$ defined in a similar way to  the random vector $\mathscr{Y}_{\mathbf{I}}$ in \eqref{def-Y-abs}:  for any integer $m\geq0$ and any $b$-adic sub-cube $\mathbf{I}\in \mathscr{D}_{m}^b$,  the $\mathscr{G}_m$-measurable random vector $\mathscr{Z}_\mathbf{I} =\mathscr{Z}_{\tau,\mathbf{I}}:= \big(\mathscr{Z}_\mathbf{I}(\mathbf{n})\big)_{\mathbf{n}\in\mathbb{Z}^{d}}$ is defined by
\begin{align}\label{def-Z-abs}
	\mathscr{Z}_\mathbf{I}(\mathbf{n}):=|\mathbf{n}|^{\tau/2}\int_{\mathbf{I}}\Big[\prod_{j=0}^{m}\XX_j(\mathbf{t})\Big]e^{-2\pi i\mathbf{n}\cdot\mathbf{t}}\mathrm{d}\mathbf{t}.
\end{align}

Therefore, the following key technical lemmas  (Lemmas \ref{UB-YZW-abs} and \ref{UB-ZW-abs}) on  the asymptotic behaviors of $\mathbb{E}[\|\mathscr{Y}_\mathbf{I} \|_{\ell^q}^p]$ and $\mathbb{E}[\|\mathscr{Z}_\mathbf{I} \|_{\ell^q}^p]$ as  $|\mathbf{I}|\to 0$  lie at the heart of our work.  

By the elementary Lemma~\ref{lem-small-p} below, Assumption~\ref{assum-Lp0-abs} implies that  for sufficiently large integer $j_0>0$,  if $1<p\le p_0\le 2$, then
\[
\sup_{0\leq j\leq j_0}\sup_{\mathbf{t}\in[0,1)^d}\mathbb{E}[\XX_j(\mathbf{t})^{p}]<\infty\anand
\sup_{j>j_0}\sup_{\mathbf{t}\in[0,1)^d}\mathbb{E}[\XX_j(\mathbf{t})^{p}]<b^{d(p-1)}.
\]
In what follows, we fix a large enough integer $j_0>0$ and denote
\begin{align}\label{def-phi-abs}
	\zeta(p)=\zeta_{j_0}(p):=d(p-1)\log b-\log\Big(\sup_{j>j_0}\sup_{\mathbf{t}\in[0,1)^d}\mathbb{E}[\XX_j(\mathbf{t})^p]\Big)>0.
\end{align}

\begin{lemma}\label{UB-YZW-abs}
	Under the Assumptions~\ref{assum-Lp0-abs} and \ref{assum-alp0-abs}, for any fixed $\tau\in(0,\mathcal{L}_F)$, if $1<p\leq p_0$ and $q>\frac{2d}{2\alpha_0-\tau}$, then there exists a constant $C = C(d,b,\tau,p,q,j_0)>0$ such that for any $b$-adic sub-cube $\mathbf{I}\subset [0,1)^d$,  
	\[
	\E[\|\mathscr{Z}_\mathbf{I}\|_{\ell^q}^p] \le C |\mathbf{I}|^{ 1 + \frac{\zeta(p)}{d\log b}-\frac{\tau p}{2d}-\frac{p}{q}}.
\]
In other words, for each $m\geq0$ and any $\mathbf{I}\in \mathscr{D}_{m}^b$,
	\begin{align}\label{Z-exponent}
	\E[\|\mathscr{Z}_\mathbf{I}\|_{\ell^q}^p]\leq C\cdot b^{-m [d + \frac{\zeta(p)}{\log b}-\frac{\tau p}{2}-\frac{dp}{q}] }.
	\end{align}
\end{lemma}

\begin{lemma}\label{UB-ZW-abs}
	Under the Assumptions~\ref{assum-Lp0-abs} and \ref{assum-alp0-abs}, for any fixed $\tau\in(0,\mathcal{L}_F)$, if $1<p\leq p_0$ and $q>\frac{2d}{2\alpha_0-\tau}$, then there exists a constant $C = C(d,b,\tau,p,q,j_0)>0$ such that for any $b$-adic sub-cube $\mathbf{I}\subset [0,1)^d$,  
\[
	\E[\|\mathscr{Y}_\mathbf{I}\|_{\ell^q}^p] \le C |\mathbf{I}|^{ 1 + \frac{\zeta(p)}{d\log b}-\frac{\tau p}{2d}-\frac{p}{q}}. 
\]
	In other words, for each $k\ge 1$ and any $\mathbf{I}\in \mathscr{D}_{k-1}^b$,
	\begin{align}\label{Y-exponent}
	\E[\|\mathscr{Y}_\mathbf{I}\|_{\ell^q}^p] \leq C \cdot b^{-k [d+\frac{\zeta(p)}{\log b}-\frac{\tau p}{2}-\frac{dp}{q}]}.
	\end{align}
\end{lemma}

\begin{remark*}
 Lemma~\ref{UB-ZW-abs} will follow from Lemma~\ref{UB-YZW-abs} by a  simple observation: for any $k\ge 1$ and $\mathbf{I}\in \mathscr{D}_{k-1}^b$,
\[
\mathscr{Y}_\mathbf{I} =\Big( \sum_{\mathbf{J}\in \mathscr{D}_k^b(\mathbf{I})}  \mathscr{Z}_\mathbf{J}\Big)  -  \mathscr{Z}_\mathbf{I}\quad \text{with $\mathscr{D}_{k}^b(\mathbf{I})= \Big\{\mathbf{J}\subset\mathbf{I}:\mathbf{J}\in\mathscr{D}_{k}^b\Big\}$}. 
\]
The importance of the upper estimates \eqref{Z-exponent} and \eqref{Y-exponent} is that,   for any $\tau\in(0,\mathcal{L}_F)$, there exists large enough  integer $j_0$ and   $p,q$ with  $1<p\leq p_0\leq\max\{2,\frac{2d}{2\alpha_0-\tau}\}<q<\infty$ such that   (see  \S \ref{sec-mart-type} for details)
\[
	\frac{\zeta(p)}{\log b}-\frac{\tau p}{2}-\frac{dp}{q} = \frac{\zeta_{j_0}(p)}{\log b}-\frac{\tau p}{2}-\frac{dp}{q}>0.
\]
Hence the quantities $\E[\|\mathscr{Z}_\mathbf{I}\|_{\ell^q}^p]/|\mathbf{I}|$ and $\E[\|\mathscr{Y}_\mathbf{I}\|_{\ell^q}^p]/|\mathbf{I}|$  decay exponentially with respect to the level of the $b$-adic sub-cube $\mathbf{I}\subset [0,1)^d$. 
\end{remark*}

\section{Fourier dimensions of Gaussian multiplicative chaos}\label{S-dDGMC-abs}
This section is devoted to the proofs of Theorems~\ref{thm-gen-12DGMC} and \ref{thm-gen-highD}.  We emphasize that, here in the GMC setting with log-correlated kernels, we go beyond the $T$-martingale setting. The main idea is  that, by using a basic comparison result on log-correlated  kernels in any bounded sets of $\R^d$,  we can reduce the study of  non-necessarily $\sigma$-positive type kernels to very special $\sigma$-positive type kernels. 

The main strategy of the proofs of Theorems~\ref{thm-gen-12DGMC} and \ref{thm-gen-highD} is as follows. We first introduce a notion of $\sigma$-regular kernels  and construct  $*$-scale invariant  $\sigma$-regular kernels in \S\ref{sec-con-reg}. Then in Proposition~\ref{prop-kernel-ass},  we prove  Assumptions~\ref{assum-indep-abs}, \ref{assum-Lp0-abs} and \ref{assum-alp0-abs} for the exponential Gaussian processes associated with the Gaussian fields with $\sigma$-regular kernels.  Finally,  Theorems~\ref{thm-gen-12DGMC} and \ref{thm-gen-highD} will be proved by a simple comparison result for GMC measures  (see \S\ref{sec-compare}) on the level of kernels and  by  Proposition~\ref{prop-kernel-ass} and our unified Theorem~\ref{Fou-Dec-abs}.

\subsection{The notion of $\sigma$-regular kernels}\label{sec-sigma-reg}
Let $K(\mathbf{t}, \mathbf{s})$ be the covariance kernel of a Gaussian field on $[0,1)^d$.     For simplicity,  we always assume that $K$ is stationary, that is (slightly abusing the notation),
\[
K(\mathbf{t}, \mathbf{s})= K(\mathbf{t}-\mathbf{s}), \quad \mathbf{t}, \mathbf{s} \in [0,1)^d. 
\]
 Following Kahane's $\sigma$-positive type condition on the kernels, we assume that 
\begin{align}\label{K-sigma-pos}
K(\mathbf{t}) =  \sum_{j=0}^\infty K_j(\mathbf{t}) \quad \text{for all $\mathbf{t} \in (-1,1)^d \setminus \{\mathbf{0}\}$},
\end{align}
where $K_j$ are positive definite  and non-negative \footnote{In Kahane's seminal works \cite{Kah85a, Kah87a},   the non-negative  assumption on all $K_j$  in  \eqref{K-sigma-pos} implies that the GMC measure constructed as in  \S\ref{sec-T-mart} does not depend on the decomposition of the kernel.   However, we mention that such uniqueness of GMC is not relevant in this work and hence one  may also drop this non-negative assumption.  } bounded continuous functions on  $[-1,1]^d$. 

Let $\psi(\mathbf{t})$ be the centered Gaussian field  on $[0,1)^d$ with covariance kernel $K(\mathbf{t}, \mathbf{s})= K(\mathbf{t} - \mathbf{s})$. The corresponding sub-critical GMC measure $\mathrm{GMC}_K^\gamma$ is informally written as 
\[
\mathrm{GMC}_K^\gamma(\mathrm{d}\mathbf{t})  \xlongequal{\text{informally}}e^{\gamma \psi (\mathbf{t}) - \frac{\gamma^2}{2} \E[\psi(\mathbf{t})^2]} \mathrm{d}\mathbf{t}. 
\]
Here we briefly recall Kahane's original construction of GMC \cite{Kah85a, Kah87a}.   The decomposition \eqref{K-sigma-pos}   gives a decomposition of $\psi (\mathbf{t})$ into the sum of independent centered  Gaussian  processes $\psi_j(\mathbf{t})$ with  covariance kernels $K_j(\mathbf{t}, \mathbf{s})= K_j(\mathbf{t}-\mathbf{s})$.   Following Kahane,  
given a parameter $\gamma>0$, define a sequence of independent stochastic processes: 
\begin{align}\label{Field-dec}
\Big\{\XX_{\gamma, j}(\mathbf{t})=   \exp\big( \gamma \psi_j(\mathbf{t}) - \frac{\gamma^2}{2} \E[\psi_j(\mathbf{t})^2] \big): \mathbf{t} \in [0,1)^d \Big\}_{j\ge 0}.
\end{align}
Then,  as  in \S\ref{sec-T-mart} for general $T$-martingales, the GMC measure $\mathrm{GMC}_K^\gamma$ is defined as the limiting random measure in the sense of weak convergence:
\[
\mathrm{GMC}_K^\gamma(\mathrm{d}\mathbf{t}) =    \lim_{m\to\infty}  \Big[\prod_{j=0}^m  \XX_{\gamma, j}(\mathbf{t}) \Big] \mathrm{d}\mathbf{t}    =  \lim_{m\to\infty}  \Big[\prod_{j=0}^m e^{\gamma \psi_j(\mathbf{t}) - \frac{\gamma^2}{2} \E[\psi_j(\mathbf{t})^2]} \Big] \mathrm{d}\mathbf{t}. 
\]

\begin{definition}[$\sigma$-regularity]\label{def-sigma-reg}
Let $\alpha_0 \in (0, 1]$ and let $b\ge 2$ be an integer.  A stationary kernel $K$  is called {\it $\sigma^b(\alpha_0)$-regular}  if it admits a  decomposition  \eqref{K-sigma-pos} such that  the following conditions (H1), (H2) and (H3) are satisfied (in this case, we say that the  decomposition  \eqref{K-sigma-pos} is $\sigma^b(\alpha_0)$-admissible): 
\begin{itemize}
\item[(H1)] Shrinking support:  There exists $j_0\ge 0$ such that for each $j\ge j_0$, the support of $K_j$ satisfies
\begin{align}\label{supp-Kj}
\supp(K_j) \subset \bar{B}(\mathbf{0}, b^{-j}),
\end{align}
  where $\bar{B}(\mathbf{0}, \delta) \subset \R^d$ denotes the centered  closed Euclidean ball of radius $\delta>0$; 
\item[(H2)] Eventual upper bound of $K_j(\mathbf{0})$: 
\begin{align}\label{Kj-bdd}
\limsup_{j\to\infty} K_j(\mathbf{0})  =   \log b;
\end{align}
\item[(H3)] Rescaled uniform regularity at  the origin $\mathbf{0}$: 
\begin{align}\label{res-0-reg}
\sup_{j\ge 0}\sup_{0<|\mathbf{t}| \le b^{-j}}  \frac{ |K_j(\mathbf{t}) -K_j(\mathbf{0})|}{b^{2j\alpha_0} |\mathbf{t}|^{ 2\alpha_0}}<\infty. 
\end{align}
\end{itemize}
 We will also say that $K$ is {\it sharply-$\sigma^b(\alpha_0)$-regular},  if the above condition \eqref{Kj-bdd} in (H2) is replaced by the following stronger condition (H2'): 
 \begin{itemize}
 \item[(H2')] Eventual constant of $K_j(\mathbf{0})$: there exists $j_0\ge 0$ such that for all $j\ge j_0$,
 \begin{align}\label{Kj-const}
 K_j(\mathbf{0}) = \log b.
 \end{align}
 \end{itemize}
\end{definition}

{\flushleft \bf Terminology:} A stationary kernel $K$ is called  $\sigma(\alpha_0)$-regular if it is $\sigma^b(\alpha_0)$-regular for some integer $b\geq2$, and  called $\sigma$-regular if it is $\sigma(\alpha_0)$-regular for some $\alpha_0\in (0,1]$.  The sharp-$\sigma$-regularity is defined similarly. Hence  we have the following terminologies:
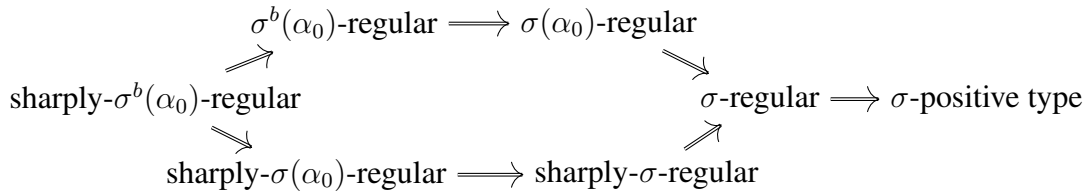
\begin{figure}[h]
\begin{center}
\begin{tikzpicture}
    \node (A) at (0,0) {sharply-$\sigma^b(\alpha_0)$-regular};
    \node (B) at (2.5,1) {$\sigma^b(\alpha_0)$-regular};
    \node (C) at (6,1) {$\sigma(\alpha_0)$-regular};
    \node(G) at (2,-1){sharply-$\sigma(\alpha_0)$-regular};
    \node (D) at (6.5,-1) {sharply-$\sigma$-regular};
    \node (E) at (8,0) {$\sigma$-regular};
    \node (F) at (11,0) {$\sigma$-positive type};
    \draw[->, double] (A) -- (G);
    \draw[->, double] (G) -- (D);
    \draw[->, double] (A) -- (B);
    \draw[->, double] (B) -- (C);
    \draw[->, double] (C) -- (E);
    \draw[->, double] (D) -- (E);
    \draw[->, double] (E) -- (F);
\end{tikzpicture}
\end{center}
\caption{The terminologies of $\sigma$-regular.}
\label{figure-sigma-regular}
\end{figure}

\begin{remark*}
	It is important to note that we only require the rescaled uniform regularity \eqref{res-0-reg} at the single point $\mathbf{0}$. Indeed,   this H\"older condition  cannot  be uniformly extended to any neighborhood of $\mathbf{0}$ (except that $K_j$ is a constant) once $\alpha_0\in (1/2, 1]$.  
\end{remark*}

\begin{remark*}
	One can also define a similar notion of  $\sigma$-regularity for non-stationary kernels. However, although we shall indeed deal with non-stationary kernels,  instead of using  a  notion of $\sigma$-regularity of non-stationary kernels, we are going to use  Assumptions~\ref{assum-indep-abs}, \ref{assum-Lp0-abs} and \ref{assum-alp0-abs} for  the associated sequence of independent exponential Gaussian processes related to certain modified non-stationary kernels (see the definition of  the modified kernels $L_\lambda$ in \eqref{def-L-lambda} and the proofs of Lemmas~\ref{lem-K-lambda} and \ref{lem-3-ass-new} below for details). 
\end{remark*}

\begin{lemma}\label{lem-most-log}
Let  $K$ be a  $\sigma$-regular stationary kernel as in Definition~\ref{def-sigma-reg}. Then 
\begin{align}\label{up-log}
 K(\mathbf{t})  =   \big(1 + o(1)\big) \log \frac{1}{|\mathbf{t}|}, \quad \text{as $\mathbf{t} \to \mathbf{0}$}. 
\end{align}
Moreover, if $K$ is  sharply-$\sigma$-regular, then 
\begin{align}\label{eq-log-bdd}
K(\mathbf{t})= \log \frac{1}{|\mathbf{t}|} + O(1), \quad \text{as $\mathbf{t} \to \mathbf{0}$}. 
\end{align}
\end{lemma}

The existence and  basic construction of $\sigma$-regular kernels will be postponed to \S\ref{sec-con-reg} below. Here we mention that, we are going to construct a natural class of  sharply-$\sigma$-regular kernels of the form 
\[
K(\mathbf{t})= \log_{+}\Big(\frac{1}{|\mathbf{t}|}\Big) + g(|\mathbf{t}|),  \quad \mathbf{t} \in \bar{B}(\mathbf{0},1)\setminus\{\mathbf{0}\},\quad \text{with $g\in C^1([0,1])$ and $\sup_{x\in (0, 1]}\frac{|g(x)-g(0)|}{x^2} <\infty$}. 
\]

We shall also mention that, if $d = 1$, inspired by the proof  that the exact log-kernel is of $\sigma$-positive type (see e.g., Rhodes-Vargas  \cite[Proposition~2.15]{RV14}), we can also show that it is sharply-$\sigma^b(1/2)$-regular.   

\begin{lemma}\label{lem-log-sigma}
Let $d = 1$ and consider the exact log-kernel 
\[
K_{\mathrm{exact}}(t) = \log_{+}\Big(\frac{1}{|t|}\Big), \quad t \in \R. 
\]
Then the kernel $K_{\mathrm{exact}}$ is sharply-$\sigma^b(1/2)$-regular for any integer $b\ge 2$. 
\end{lemma}

\subsection{Consequences of $\sigma$-regular kernels}
\begin{proposition}\label{prop-kernel-ass}
Let $\alpha_0\in (0,1]$ and let  $K$ be a  stationary  kernel admitting the $\sigma^b(\alpha_0)$-admissible  decomposition  \eqref{K-sigma-pos}. Then for all $\gamma \in (0, \sqrt{2d})$, the following sequence defined as in \eqref{Field-dec}: 
\begin{align}\label{def-P-g-j}
\Big\{\XX_{\gamma, j}(\mathbf{t})=   \exp\big( \gamma \psi_j(\mathbf{t}) - \frac{\gamma^2}{2} \E[\psi_j(\mathbf{t})^2] \big): \mathbf{t} \in [0,1)^d \Big\}_{j\ge 0}
\end{align}
satisfies Assumptions~\ref{assum-indep-abs}, \ref{assum-Lp0-abs} and \ref{assum-alp0-abs} for the given $\alpha_0$ and for all  $1< p_0 <  \min\{\frac{2d}{\gamma^2}, 2\}$. 
\end{proposition}

Recall  that for any integer $d\ge 1$ and any $\gamma\in (0, \sqrt{2d})$,  we defined the quantity $D_{\gamma, d}$ in  \eqref{def-D-gamma}  by 
\[
D_{\gamma,d} =\left\{
	\begin{array}{cl}
		d-\gamma^2 & \text{if  $0<\gamma<\sqrt{2d}/2$}
		\vspace{2mm}
		\\
		(\sqrt{2d}-\gamma)^2 & \text{if $\sqrt{2d}/2\leq\gamma<\sqrt{2d}$}
	\end{array}\right..
\]

\begin{proposition}\label{thm-sigma-reg}
Let $\alpha_0\in (0,1]$ and let  $K$ be a $\sigma(\alpha_0)$-regular  stationary kernel on $[0,1)^d$.  Then for all $\gamma\in (0, \sqrt{2d})$,  the GMC measure $\mathrm{GMC}_{K}^\gamma$ is non-degenerate and almost surely has polynomial Fourier decay with
\[
\dim_F(\mathrm{GMC}_K^\gamma) \ge \min\big\{2 \alpha_0,  D_{\gamma, d}\big\} >0.
\]
\end{proposition}

Combining Proposition~\ref{thm-sigma-reg}  and the multifractal analysis of GMC measures (in particular the formula for the correlation dimension, see Bertacco \cite[Theorem~3.1 and Formula~(3.2)]{Ber23}, as well as Rhodes-Vargas  \cite[Section~4.2]{RV14} and   Garban-Vargas \cite[Remark~2]{GV23}), we obtain the exact values of the Fourier dimensions of 1D and 2D sub-critical GMC measures.

\begin{corollary}[1D and 2D GMC]\label{cor-low-dim}
Let $d\in \{1,2\}$ and  let  $K$ be a  stationary kernel of the form 
\[
K(\mathbf{t})= \log_{+}\Big( \frac{1}{|\mathbf{t}|}\Big) + g(|\mathbf{t}|),  \quad \mathbf{t}\in [0,1)^d, \quad\text{with $g$ bounded and continuous}. 
\]
Assume that $K$ is $\sigma(1)$-regular.   Then   for all $\gamma\in (0, \sqrt{2d})$,  almost surely, 
\[
\dim_F(\mathrm{GMC}_K^\gamma)  =  D_{\gamma, d}.
\]
\end{corollary}

\subsection{Proofs of Lemma~\ref{lem-most-log} and Lemma~\ref{lem-log-sigma}}
\begin{proof}[Proof of Lemma~\ref{lem-most-log}]
Fix any $\mathbf{t}$ with $b^{-(m+1)} < |\mathbf{t}|\le b^{-m}$.  By  \eqref{K-sigma-pos} and \eqref{supp-Kj}, we have 
\begin{align}\label{K-sum-finite}
	K(\mathbf{t})= \sum_{j=0}^m K_j(\mathbf{t}). 
\end{align}
Note that  $|\mathbf{t}|\le b^{-m} \le b^{-j}$ for all $0\le j \le m$. Hence by the rescaled uniform regularity condition \eqref{res-0-reg},   there exist constants $C, C'>0$ (indepedent of $m$ and $\mathbf{t}$) such that  
\[
\sum_{j=0}^m |K_j(\mathbf{t})- K_j(\mathbf{0})| \le C  \sum_{j=0}^m  b^{2j \alpha_0} |\mathbf{t}|^{2 \alpha_0}  = C' b^{2m \alpha_0} |\mathbf{t}|^{2\alpha_0}   \le C'. 
\]
Therefore, by \eqref{K-sum-finite}, 
\[
K(\mathbf{t}) = \sum_{j=0}^m K_j(\mathbf{0}) + O(1). 
\] 
Consequently, under the condition \eqref{Kj-bdd}, we obtain the desired equality \eqref{up-log}; while under the stronger condition \eqref{Kj-const}, we obtain the desired equality \eqref{eq-log-bdd}. This completes the proof of Lemma~\ref{lem-most-log}. 
\end{proof}

\begin{proof}[Proof of Lemma~\ref{lem-log-sigma}]
Our proof of Lemma~\ref{lem-log-sigma} relies on  the following simple equality: 
\[
K_{\mathrm{exact}}(t) = \log_{+} \Big(\frac{1}{|t|}\Big)= \int_{[0,1]} (u-|t|)_{+} \nu(\mathrm{d}u),\quad\text{where $\nu(\mathrm{d}u) = \frac{\mathrm{d}u}{u^2} + \delta_1(\mathrm{d}u)$}. 
\]
Fix any integer $b\ge 2$, write 
\[
K_{\mathrm{exact}}(t) = \sum_{j=0}^\infty  \underbrace{\int_{(b^{-(j+1)}, b^{-j}]}  (u-|t|)_{+} \nu(\mathrm{d}u)}_{\text{denoted $K_j(t)$}}.
\]
Clearly, $K_j$ is non-negative and $\supp(K_j) \subset \bar{B}(0, b^{-j})$. Since the function $x\mapsto (1- |x|)_{+}$ is positive definite on $\R$ and $\nu$ is positive, the function $K_j$ is positive definite.   For any integer  $j\ge 1$,  
\[
K_j(0) = \int_{b^{-(j+1)}}^{b^{-j}}  \frac{\mathrm{d}u}{u}  = \log b. 
\]
Moreover, using the H\"older regularity at the origin $0$:
\[
| (u - |t|)_{+} - u |\le |t| \quad \text{for all $u\in [0,1]$},
\]
we obtain, for any $j \ge 1$,  by taking $\alpha_0=1/2$ and $|t|\le b^{-j}$, 
\[
 \frac{ |K_j(t) -K_j(0)|}{b^{2j\alpha_0} |t|^{ 2\alpha_0}} \le   \frac{1}{b^{j} |t|} \int_{b^{-(j+1)}}^{b^{-j}}   | (u - |t|)_{+} - u |  \frac{\mathrm{d}u}{u^2}  \le  \frac{1}{b^{j}} \int_{b^{-(j+1)}}^{b^{-j}}  \frac{\mathrm{d}u}{u^2}  =  b-1. 
\]
For $j=0$, we clearly have 
\[
\sup_{0<|t| \le 1}  \frac{ |K_0(t) -K_0(0)|}{|t|}<\infty.
\]
Consequently, we obtain the desired inequality \eqref{res-0-reg} with $\alpha_0 = 1/2$. This completes the proof that $K_{\mathrm{exact}}$ is sharply-$\sigma^b(1/2)$-regular. 
\end{proof}

\subsection{Proof of Proposition~\ref{prop-kernel-ass}}
Let $\alpha_0\in (0,1]$ and  $b\ge 2$ be an integer. Assume that $K$ is a stationary kernel admitting a given $\sigma^b(\alpha_0)$-admissible  decomposition  \eqref{K-sigma-pos}. Let $\psi_j(\mathbf{t})$ be the sequence of independent  Gaussian  process  with  kernel $K_j(\mathbf{t}, \mathbf{s})= K_j(\mathbf{t}-\mathbf{s})$.  Our goal is to show that, for any $\gamma\in (0,\sqrt{2d})$,  the sequence of independent stochastic processes 
\begin{align}\label{gamma-proc}
\Big\{\XX_{\gamma, j}(\mathbf{t})=   \exp\big( \gamma \psi_j(\mathbf{t}) - \frac{\gamma^2}{2} \E[\psi_j(\mathbf{t})^2] \big): \mathbf{t} \in [0,1)^d \Big\}_{j\ge 0}
\end{align}
satisfies Assumptions~\ref{assum-indep-abs}, \ref{assum-Lp0-abs} and \ref{assum-alp0-abs} for the given $\alpha_0$ and for all  $p_0$ with $1< p_0 < \min\{\frac{2d}{\gamma^2},2\}$.

{\flushleft \it 1). Verification of Assumption~\ref{assum-indep-abs}.}  Recall the notation \eqref{def-b-dya-abs} for $\mathscr{D}_j^b$. For any integer $j\ge 0$ and each $\theta = (\theta_1, \cdots, \theta_d) \in \{0,1\}^d$,  define a sub-family $\mathscr{D}_{j, \theta}^b\subset \mathscr{D}_j^b$ as 
\[
\mathscr{D}_{j, \theta}^b : = \Big\{ \mathbf{I} =  \prod_{\beta=1}^d \Big[\frac{h_\beta-1}{b^j}, \frac{h_\beta}{b^j}\Big):   (h_1, \cdots, h_d) \equiv (\theta_1,\cdots, \theta_d)  \, (\mathrm{mod} 2)\Big\}. 
\]
Clearly, we have 
\begin{align}\label{dec-odd-even}
\mathscr{D}_{j}^b = \bigsqcup_{\theta \in \{0,1\}^d} \mathscr{D}_{j, \theta}^b. 
\end{align}
Moreover, for any $\theta \in \{0, 1\}^d$, inside each  sub-family $ \mathscr{D}_{j, \theta}^b$, all pairs of distinct sub-cubes $\mathbf{I}, \mathbf{I}'$ satisfy 
\begin{align}\label{dist-I-I}
\dist (\mathbf{I}, \mathbf{I}')  = \inf\{|\mathbf{t}-\mathbf{t}'|: \mathbf{t}\in \mathbf{I}\,\,\mathrm{and}\,\,\mathbf{t}' \in \mathbf{I}'\} \ge b^{-j}. 
\end{align}
The above inequality follows from the fact that once $\mathbf{I} \ne  \mathbf{I}'$ are in  the same sub-family $\mathscr{D}_{j,\theta}^b$ with 
\[
\mathbf{I} =  \prod_{\beta=1}^d \Big[\frac{h_\beta-1}{b^j}, \frac{h_\beta}{b^j}\Big) \anand \mathbf{I}' =  \prod_{\beta=1}^d \Big[\frac{h_\beta'-1}{b^j}, \frac{h_\beta'}{b^j}\Big),
\]
and then
\[
\inf_{1\le \beta \le d} |h_\beta- h_\beta'|\ge 2. 
\]
We can now verify Assumption~\ref{assum-indep-abs} for the sequence \eqref{gamma-proc} of independent stochastic processes with respect to  the partition \eqref{dec-odd-even} and the bounded  (thus sub-exponential) sequence $N_j= \#(\{0,1\}^d) = 2^d$.  Indeed, it suffices to show that, for each $\theta\in \{0,1\}^d$ and each pair $(\mathbf{I}, \mathbf{I}')$ of distinct sub-cubes in $\mathscr{D}_{j, \theta}^b$,  the following two stochastic processes  are independent: 
\[
\{\psi_j(\mathbf{t}): \mathbf{t} \in \mathbf{I}\}, \quad \{\psi_j(\mathbf{t}'): \mathbf{t}' \in \mathbf{I}'\}.
\]
 Since they are both Gaussian processes,  we only need to show that for any $\mathbf{t} \in \mathbf{I}$ and $\mathbf{t}' \in \mathbf{I}'$, 
 \[
 \Cov(\psi_j(\mathbf{t}), \psi_j(\mathbf{t}')) = K_j(\mathbf{t}-\mathbf{t}') = 0.
 \]
This equality follows  from   the inequality \eqref{dist-I-I} and the condition \eqref{supp-Kj} on the support of $K_j$ (here we also used the continuity assumption of $K_j$, which implies that $K_j(\mathbf{t}-\mathbf{t}') = 0$ if $|\mathbf{t}-\mathbf{t}'|= b^{-j}$). 
 
{\flushleft \it 2). Verification of Assumption~\ref{assum-Lp0-abs}.} For any $j\ge 0$, by the definition  \eqref{gamma-proc} of $\XX_{\gamma, j}(\mathbf{t})$, we have 
\[
	\E[\XX_{\gamma, j}(\mathbf{t})^{p_0}]   =  \E\Big[\exp\Big( \gamma p_0 \psi_j(\mathbf{t}) - \frac{\gamma^2 p_0}{2} \E[\psi_j(\mathbf{t})^2] \Big)\Big]  =  \exp\Big( \frac{\gamma^2 p_0(p_0-1)}{2} \E[\psi_j(\mathbf{t})^2] \Big)
\]
and hence
\begin{align}\label{P-p-m}
\E[\XX_{\gamma, j}(\mathbf{t})^{p_0}]   =   \exp\Big( \frac{\gamma^2 p_0(p_0-1)}{2} K_j(\mathbf{0}) \Big)<\infty.
\end{align}
Then for any $\gamma\in (0, \sqrt{2d})$ and  $1<p_0 < \min\{\frac{2d}{\gamma^2},2\}$,  by the condition \eqref{Kj-bdd}, 
\begin{align}\label{lim-sup-P}
\limsup_{j\to\infty}\sup_{\mathbf{t}\in[0,1)^d}\mathbb{E}[\XX_{\gamma, j}(\mathbf{t})^{p_0}] = \exp\Big( \frac{\gamma^2 p_0(p_0-1)}{2} \limsup_{j\to\infty} K_j(\mathbf{0}) \Big)  = b^{\frac{\gamma^2 p_0(p_0-1)}{2}} <    b^{d (p_0-1)},
\end{align}
where in the last inequality, we used $\gamma^2 p_0  <2d$.  Therefore,  Assumption~\ref{assum-Lp0-abs} holds. 

{\flushleft \it 3). Verification of Assumption~\ref{assum-alp0-abs}.} Since $1<p_0 < \min\{\frac{2d}{\gamma^2},2\}\le 2$, we only need to show
\begin{align}\label{hol-goal}
\sup_{j\in \N}\sup_{\mathbf{I}\in\mathscr{D}_j^b}\sup_{\mathbf{t},\mathbf{s}\in\mathbf{I} \atop \mathbf{t}\neq\mathbf{s}} \mathbb{E}\Big[\Big| \frac{\XX_{\gamma, j}(\mathbf{t})-\XX_{\gamma, j}(\mathbf{s})}{b^{j\alpha_0}|\mathbf{t}-\mathbf{s}|^{\alpha_0}} \Big|^{2}\Big]<\infty. 
\end{align}
Now let  $j\ge 0$. Take any $\mathbf{I} \in \mathscr{D}_j^b$ and $\mathbf{t}, \mathbf{s}\in \mathbf{I}$, by \eqref{P-p-m}, we have 
\[
\E [ \XX_{\gamma, j}(\mathbf{t})^2] = \E[ \XX_{\gamma, j}(\mathbf{s})^2] =   \exp\big(\gamma^2 K_j(\mathbf{0}) \big). 
\]
We also have 
\[
\E[\XX_{\gamma, j}(\mathbf{t})   \XX_{\gamma, j}(\mathbf{s})]  = \E\Big[ \exp\Big( \gamma \big[\psi_j(\mathbf{t})   + \psi_j(\mathbf{s}) \big] - \frac{\gamma^2}{2} \big(\E[\psi_j(\mathbf{t})^2] + \E[\psi_j(\mathbf{s})^2]\big)  \Big)\Big] 
\]
and then
\[
\E[\XX_{\gamma, j}(\mathbf{t}) \XX_{\gamma, j}(\mathbf{s})] = \exp\big(\gamma^2 \E[\psi_j(\mathbf{t}) \psi_j(\mathbf{s})]\big)=\exp\big(\gamma^2 K_j(\mathbf{t}-\mathbf{s})\big). 
\]
It follows that 
\[
\E\big[\big|  \XX_{\gamma, j}(\mathbf{t})-\XX_{\gamma, j}(\mathbf{s})  \big|^{2}\big]  =   2 \big[ \exp\big(\gamma^2 K_j(\mathbf{0}) \big) - \exp\big(\gamma^2 K_j(\mathbf{t}-\mathbf{s})\big) \big] 
\]
and hence 
\begin{align}\label{2-2-upper}
\begin{split}
&\quad\,\,   \sup_{j\in \N}\sup_{\mathbf{I}\in\mathscr{D}_j^b}\sup_{\mathbf{t},\mathbf{s}\in\mathbf{I} \atop \mathbf{t}\neq\mathbf{s}} \mathbb{E}\Big[\Big| \frac{\XX_{\gamma, j}(\mathbf{t})-\XX_{\gamma, j}(\mathbf{s})}{b^{j\alpha_0}|\mathbf{t}-\mathbf{s}|^{\alpha_0}} \Big|^{2}\Big] 
\\
& \le  2  \exp\big(\gamma^2 \sup_{j\ge 0} K_j(\mathbf{0}) \big) \cdot  \sup_{j\ge 0} \sup_{0<|\mathbf{t}|\le b^{-j}} \frac{ 1 -  \exp\big(-\gamma^2 [ K_j(\mathbf{0})- K_j(\mathbf{t})] \big)}{b^{2j\alpha_0} |\mathbf{t}|^{2\alpha_0}}.
\end{split}
\end{align}
Since $K_j$ is non-negative and positive definite,  for any $\mathbf{t}$, we have 
$
0\le K_j(\mathbf{t}) \le K_j(\mathbf{0})$. Then by \eqref{Kj-bdd}, 
\[
\sup_{j\ge 0}   |K_j(\mathbf{0}) - K_j(\mathbf{t})| \le \sup_{j\ge 0}  K_j(\mathbf{0})<\infty. 
\]
Hence for any $\gamma>0$, there exists a constant $ C(\gamma)>0$ such that  for any $j\ge 0$ and any $|\mathbf{t}|\le b^{-j}$, 
\begin{align}\label{holder-holder}
\big|1 -  \exp\big(- \gamma^2 [K_j(\mathbf{0})- K_j(\mathbf{t})]\big)\big| \le C(\gamma) |K_j(\mathbf{0})- K_j(\mathbf{t})|.
\end{align}
Combining  \eqref{2-2-upper}, \eqref{holder-holder} with the rescaled uniform regularity condition \eqref{res-0-reg}, we obtain the desired inequality \eqref{hol-goal} and thus complete the verification of Assumption~\ref{assum-alp0-abs}.

\subsection{Proof of Proposition~\ref{thm-sigma-reg}}  
We now derive Proposition~\ref{thm-sigma-reg} from the unified Theorem~\ref{Fou-Dec-abs}. 
Let $\alpha_0\in (0,1]$ and let  $K$ be a $\sigma(\alpha_0)$-regular  stationary kernel on $[0,1)^d$. Fix any $\gamma \in (0, \sqrt{2d})$. By Proposition~\ref{prop-kernel-ass}, the sequence of independent stochastic processes defined in \eqref{def-P-g-j} satisfies Assumptions~\ref{assum-indep-abs}, \ref{assum-Lp0-abs} and \ref{assum-alp0-abs} for the given $\alpha_0$ and for all  $1< p_0 <  \min\{\frac{2d}{\gamma^2}, 2\}$. Therefore, by Theorem~\ref{Fou-Dec-abs},  almost surely, we have 
\[
	\dim_F(\mathrm{GMC}_{K}^\gamma)\geq  \min\Big\{2 \alpha_0,\sup_{1<p\leq p_0}\frac{2  \Theta_{\gamma}(p)}{p\log b}\Big\},
\]
where $\Theta_{\gamma}(p)$ is defined as in  \eqref{def-Theta-abs}:
\[
	\Theta_{\gamma}(p):=d(p-1)\log b-\log\Big(\limsup_{j\to\infty}\sup_{\mathbf{t}\in[0,1)^d}\mathbb{E}[\XX_{\gamma,j}(\mathbf{t})^p]\Big).
\]
Since $1<p_0<\min\{\frac{2d}{\gamma^2}, 2\}$ is chosen arbitrarily, 
almost surely, we have 
\[
	\dim_F(\mathrm{GMC}_{K}^\gamma)\geq  \min\Big\{2 \alpha_0,\sup_{1<p\leq  \min\{\frac{2d}{\gamma^2}, 2\} }\frac{2  \Theta_{\gamma}(p)}{p\log b}\Big\}.
\]
Note that by \eqref{lim-sup-P}, 
\begin{align}\label{def-zeta-p-1}
\frac{\Theta_{\gamma}(p)}{p\log b}   =d+\frac{\gamma^2}{2}-\Big(\frac{\gamma^2p}{2}+\frac{d}{p}\Big). 
\end{align}
Then by an elementary computation of the minimal value of $\frac{\gamma^2 p}{2} + \frac{d}{p}$ on $p\in[1, \min\{\frac{2d}{\gamma^2}, 2\}]$, we obtain 
\begin{align}\label{def-zeta-p-2}
\sup_{1<p\leq  \min\{\frac{2d}{\gamma^2}, 2\} }\frac{2  \Theta_{\gamma}(p)}{p\log b} =  \frac{D_{\gamma, d}}{2},
\end{align}
where $D_{\gamma, d}$ is given by the formula \eqref{def-D-gamma}.  This completes the whole proof of Proposition~\ref{thm-sigma-reg}.

\subsection{Proof of Corollary~\ref{cor-low-dim}}
The following lemma is due to Bertacco \cite[Theorem~3.1 and Formula~(3.2)]{Ber23}, as well as Rhodes-Vargas  \cite[Section~4.2]{RV14} and   Garban-Vargas \cite[Remark~2]{GV23}. 

\begin{lemma}\label{Uppbou-FourierDim-dDGMC}
Let $d\ge 1$ be an integer.  Assume that $K$ is a Gaussian  kernel on $[0,1]^d$ of the  form: 
\[
K(\mathbf{t},\mathbf{s})= \log_{+}\Big( \frac{1}{|\mathbf{t}-\mathbf{s}|}\Big) + g(\mathbf{t}, \mathbf{s}), \quad \mathbf{t}, \mathbf{s}\in [0,1]^d,
\]
where $g$ is  bounded continuous on $[0,1]^d\times [0,1]^d$.  	Then for each $\gamma\in(0,\sqrt{2d})$, almost surely, we have
	\[
	\mathrm{dim}_{F}(\mathrm{GMC}_{K}^\gamma)\leq D_{\gamma,d}.
	\]
\end{lemma}

\begin{proof}
	Recall that the correlation dimension $\dim_2(\nu)$ of a measure $\nu$ is defined by (see, e.g., \cite[Lemma~2.6.6 and Definition~2.6.7]{BSS23})
	\begin{align}\label{def-sup-ball-dDGMC}
		\dim_2(\nu): =  \liminf_{\delta\to 0^+}\frac{ \log \Big(\sup \sum_i \nu \big(B(\mathbf{x}_i, \delta)\big)^2 \Big)}{\log \delta},
	\end{align}
	where the supremum is taken over all families of disjoint balls. In \cite{Ber23}, Bertacco  introduced the $L^q$-spectrum of the measure $\nu$ by 
	\[
	\tau_\nu(q): =   \limsup_{\delta \to 0^+}\frac{ \log \Big(\sup \sum_i \nu \big(B(\mathbf{x}_i, \delta)\big)^q \Big)}{-\log \delta}, \quad q\in \mathbb{R},
	\]
	where the supremum is taken over all families of disjoint balls. From Bertacco's definition, we have  (the $\limsup$ becomes $\liminf$ after multiplication by $-1$)
	\begin{align}\label{B-Lq-def-dDGMC}
		-\tau_\nu(q) =   \liminf_{\delta \to 0^+}\frac{ \log \Big(\sup \sum_i \nu \big(B(\mathbf{x}_i, \delta)\big)^q \Big)}{\log \delta}.
	\end{align}
	By comparing  \eqref{def-sup-ball-dDGMC} and \eqref{B-Lq-def-dDGMC}, we get 
	\[
	\dim_2(\nu)= - \tau_\nu(2). 
	\]
	Note that the above equality is a particular case of \cite[Lemma~2.6.6]{BSS23}. 
	
	Bertacco \cite[Theorem~3.1 and Formula~(3.2)]{Ber23}, as well as Rhodes-Vargas  \cite[Section~4.2]{RV14} and   Garban-Vargas \cite[Remark~2]{GV23}, obtained the precise value of the $L^q$-spectrum  $\tau_{ \mathrm{GMC}_K^\gamma}(q)$ of the $d$-dimensional GMC measure  for all $q\in \mathbb{R}$.  It was shown that for any $\gamma\in(0,\sqrt{2d})$,  almost surely,  
	\[
		\dim_2( \mathrm{GMC}_K^\gamma)= - \tau_{ \mathrm{GMC}_K^\gamma}(2)= 
		\left\{
		\begin{array}{cc}
			\xi_{ \mathrm{GMC}_K^\gamma}(2) - d & \text{if $2\le \sqrt{2d}/\gamma$}
			\vspace{2mm}
			\\
			2\xi_{\mathrm{GMC}_K^\gamma}'(\sqrt{2d}/\gamma) & \text{if $2\ge  \sqrt{2d}/\gamma$} 
		\end{array}
		\right., 
	\]
	where $\xi_{ \mathrm{GMC}_K^\gamma}(q)$ (see \cite[Formula~(2.5)]{Ber23}) is the power law spectrum of $ \mathrm{GMC}_K^\gamma$ given by 
	\[
	\xi_{ \mathrm{GMC}_K^\gamma}(q) = \Big(d + \frac{1}{2}\gamma^2 \Big)q - \frac{1}{2}\gamma^2 q^2, \quad q \in \mathbb{R}. 
	\]
	Therefore, by an elementary computation, we obtain that for any $\gamma\in(0,\sqrt{2d})$, almost surely,  
	\[
	\dim_2( \mathrm{GMC}_K^\gamma) = \left\{
	\begin{array}{cl}
		d-\gamma^2 & \text{if $0<\gamma \le \sqrt{2d}/2$}
		\vspace{2mm}
		\\
		2d+\gamma^2 - 2 \sqrt{2d} \gamma & \text{if $\sqrt{2d}/2\le \gamma <\sqrt{2d}$} 
	\end{array}
	\right..
	\]
	Then for any $\gamma\in(0,\sqrt{2d})$, recall the definition of $D_{\gamma,d}$ in \eqref{def-D-gamma}, almost surely,  
	\[
	\dim_2( \mathrm{GMC}_K^\gamma)   = D_{\gamma,d} .
	\]
	Finally, we complete the proof of Lemma~\ref{Uppbou-FourierDim-dDGMC} by the following standard inequality in potential theory: 
	\[
	\dim_F(\nu)\le \dim_2(\nu)
	\]
	 for any finite Borel measure $\nu$ supported on a compact subset. 
\end{proof}

\begin{proof}[Proof of Corollary~\ref{cor-low-dim}]
Since $d\in\{1,2\}$, $K$ is $\sigma(1)$-regular and $D_{\gamma,d}\in (0,2)$ for any $\gamma \in (0, \sqrt{2d})$, Corollary~\ref{cor-low-dim} follows immediately  from Proposition~\ref{thm-sigma-reg} and Lemma~\ref{Uppbou-FourierDim-dDGMC}.
\end{proof}

\subsection{Construction of $\sigma$-regular and $*$-scale invariant kernels}\label{sec-con-reg} 
In this subsection, we are going to construct  isotropic stationary $\sigma$-regular kernels  (here  isotropic means that 
$K(\mathbf{t})= K(|\mathbf{t}|)$).

\begin{lemma}\label{lem-base-fn}
There exists a non-negative, positive definite and isotropic  function $\Phi \in C^\infty(\R^d)$ with 
\[
\supp(\Phi) \subset \bar{B}(\mathbf{0},1), \quad \Phi(\mathbf{0}) = 1 \anand  \sup_{\mathbf{t}\in \R^d\setminus \{\mathbf{0}\}} \frac{|\Phi(\mathbf{t}) - \Phi(\mathbf{0})|}{|\mathbf{t}|^{2}}<\infty. 
\]
\end{lemma}

The proof of Lemma~\ref{lem-base-fn} is postponed to the end of this subsection.  

We now proceed to the construction of sharply-$\sigma$-regular kernel as follows.  Let  $\Phi \in C^\infty(\R^d)$ be  a function satisfying all conditions in Lemma~\ref{lem-base-fn}. We may define a stationary kernel
\begin{align}\label{K-dial-dec}
  K(\mathbf{t})= \int_1^\infty \Phi(u \mathbf{t}) \frac{\mathrm{d} u}{u}.
\end{align}
In the literature of multiplicative chaos, the  kernel $K$ given in the above is referred to as the $*$-scale invariant kernel,  see, e.g., \cite[Section~2]{RV14}.   Note that $\supp(\Phi)\subset \bar{B}(\mathbf{0},1)$ implies that
\[
\text{$\supp(K)\subset \bar{B}(\mathbf{0},1)$. }
\]

\begin{proposition}\label{prop-star-scale}
 The $*$-scale invariant kernel $K$ defined in \eqref{K-dial-dec} is sharply-$\sigma^b(1)$-regular for any integer $b\ge 2$ and satisfies that  for any $\mathbf{t} \in \bar{B}(\mathbf{0},1)\setminus\{\mathbf{0}\}$, 
\begin{align}\label{L-log-error}
K(\mathbf{t})= \log_{+} \Big(\frac{1}{|\mathbf{t}|}\Big) + g(|\mathbf{t}|)\quad    \text{ with $g\in C^1([0,\sqrt{d}\,])$,\, $\supp(g)\subset [0,1]$ and $\sup_{x\in (0, 1]}\frac{|g(x)-g(0)|}{x^2} <\infty$}. 
\end{align}
Moreover, $g$ can be written as 
\[
g(x) = \int_{x}^1 \frac{f(v) - 1}{v} \mathrm{d} v, 
\]
with $f$ belongs to the class $\mathscr{P}_d$ introduced in \eqref{def-Pd}. 
\end{proposition}

\begin{proof}
Fix any integer $b\ge 2$. For any integer $j\ge 0$, define 
\begin{align}\label{def-Lj}
K_j(\mathbf{t}) = \int_{b^j}^{b^{j+1}} \Phi(u \mathbf{t}) \frac{\mathrm{d} u}{u} \anand 
K(\mathbf{t}) = \sum_{j=0}^\infty K_j(\mathbf{t}). 
\end{align}
Let us now verify all the conditions (H1), (H2') and (H3) in Definition~\ref{def-sigma-reg}. 

{\it (H1) condition:} The series in \eqref{def-Lj} is a finite sum for all $\mathbf{t}\ne \mathbf{0}$ since we have 
\begin{align}\label{supp-L}
\supp(K_j) \subset \bar{B}(\mathbf{0}, b^{-j}). 
\end{align}
Indeed, take any $|\mathbf{t}|>b^{-j}$,  the assumption $\supp(\Phi) \subset \bar{B}(\mathbf{0},1)$ implies  $\Phi(u\mathbf{t})  =0$ for all $u \ge b^{j}$ and hence $K_j(\mathbf{t}) = 0$. This  proves the assertion \eqref{supp-L}.  

{\it (H2') condition:} The definition \eqref{def-Lj}  and the assumption $\Phi(\mathbf{0}) = 1$ together imply that 
\[
K_j(\mathbf{0})= \int_{b^j}^{b^{j+1}} \Phi(\mathbf{0}) \frac{\mathrm{d}u}{u} = \log b \quad \text{for all integers $j\ge 0$}. 
\]

{\it (H3) condition:} For any integer $j\ge 0$, we have 
\[
\sup_{0<|\mathbf{t}| \le b^{-j}}  \frac{ |K_j(\mathbf{t}) -K_j(\mathbf{0})|}{b^{2j} |\mathbf{t}|^{ 2}}  
\le  \sup_{0<|\mathbf{t}| \le b^{-j}}     \int_{b^j}^{b^{j+1}}   \frac{| \Phi(u \mathbf{t}) - \Phi( \mathbf{0}) |}{|b^{j} \mathbf{t}|^{2}} \frac{\mathrm{d} u}{u}
\le   \sup_{0<|\mathbf{t}| \le b^{-j}}     \int_{b^j}^{b^{j+1}}   \frac{| \Phi(u \mathbf{t}) - \Phi( \mathbf{0}) |}{ |u \mathbf{t}/b|^{ 2}} \frac{\mathrm{d} u}{u}
\]
and hence
\[
\sup_{0<|\mathbf{t}| \le b^{-j}}  \frac{ |K_j(\mathbf{t}) -K_j(\mathbf{0})|}{b^{2j} |\mathbf{t}|^{ 2}}  \le\int_{b^j}^{b^{j+1}}    \sup_{0<|\mathbf{t}| \le b^{-j}}   \frac{| \Phi(u \mathbf{t}) - \Phi( \mathbf{0}) |}{ |u \mathbf{t}/b|^{ 2}} \frac{\mathrm{d} u}{u}  \le       \int_{b^j}^{b^{j+1}}    \sup_{\mathbf{s} \in \R^d \setminus \{\mathbf{0}\}}   \frac{| \Phi(\mathbf{s}) - \Phi( \mathbf{0}) |}{ | \mathbf{s}/b|^{ 2}} \frac{\mathrm{d} u}{u}.
\]
This implies that
\[
\sup_{0<|\mathbf{t}| \le b^{-j}}  \frac{ |K_j(\mathbf{t}) -K_j(\mathbf{0})|}{b^{2j} |\mathbf{t}|^{ 2}} \le
b^{2} \log b   \sup_{\mathbf{s} \in \R^d\setminus \{\mathbf{0}\}}  \frac{ |\Phi(\mathbf{s}) -\Phi(\mathbf{0})|}{| \mathbf{s}|^{ 2}}<\infty.
\]
It follows that 
\[
\sup_{j\ge 0}\sup_{0<|\mathbf{t}| \le b^{-j}}  \frac{ |K_j(\mathbf{t}) -K_j(\mathbf{0})|}{b^{2j} |\mathbf{t}|^{ 2}}  \le  b^{2} \log b \sup_{\mathbf{s} \in \R^d\setminus \{\mathbf{0}\}}  \frac{ |\Phi(\mathbf{s}) -\Phi(\mathbf{0})|}{| \mathbf{s}|^{ 2}}<\infty.
\]
Thus, we complete the proof that $K$ is sharply-$\sigma^b(1)$-regular.

Finally, we prove that $K$ is of the form \eqref{L-log-error}.  Indeed, by \eqref{K-dial-dec},  using the change-of-variable $u\to 1/u$ and the assumption $\supp(\Phi) \subset \bar{B}(\mathbf{0}, 1)$, we obtain that for all $0<|\mathbf{t}|<1$, 
\[
K(\mathbf{t}) = \int_{0}^1 \Phi\Big( \frac{\mathbf{t}}{u}\Big) \frac{\mathrm{d}u}{u} = \int_{|\mathbf{t}|}^1 \Phi\Big( \frac{\mathbf{t}}{u}\Big) \frac{\mathrm{d}u}{u} = \log\frac{1}{|\mathbf{t}|} + \int_{|\mathbf{t}|}^1 \Big[  \Phi\Big( \frac{\mathbf{t}}{u}\Big) -1\Big] \frac{\mathrm{d}u}{u} = \log \frac{1}{|\mathbf{t}|} + r(\mathbf{t}). 
\]
Since $\Phi$ is isotropic, we may write $\Phi(\mathbf{t})= f(|\mathbf{t}|)$ and $r(\mathbf{t}) = g(|\mathbf{t}|)$. Then using the change of variable $x/u = v$,  we have $\mathrm{d}u/u = - \mathrm{d}v/v$ and
  \begin{align}\label{g-int-form}
  g(x)= \int_{x}^1 \Big[  f\Big( \frac{x}{u}\Big) -1  \Big] \frac{\mathrm{d}u}{u} =   \int_{x}^1 (f(v) -1)\frac{\mathrm{d}v}{v} \quad \text{for all $0<x\le 1$}.
  \end{align}
  Since $\Phi$ satisfies all the conditions of Lemma~\ref{lem-base-fn} and $f(x)= \Phi((x, 0, \cdots, 0))$,  we know that the function $v\mapsto (f(v)-1)/v$ is continuous on $[0,1]$. Therefore, $g\in C^1([0,1])$.  Note also that by \eqref{g-int-form} and $f(x) = \Phi((x, 0, \cdots, 0))$, for any $0<x\le 1$, we have 
\begin{align*}
|  g(x) - g(0) |  & \le \int_0^x \frac{|f(v) - 1|}{v} \mathrm{d}v = \int_0^x \frac{| \Phi((v, 0, \cdots, 0) - \Phi(\mathbf{0})|}{v^2}  v \mathrm{d}v
\\
&  \le    \sup_{\mathbf{s} \in \R^d\setminus \{\mathbf{0}\}}  \frac{ |\Phi(\mathbf{s}) -\Phi(\mathbf{0})|}{| \mathbf{s}|^{ 2}}  \cdot  \int_0^x    v \mathrm{d}v=\sup_{\mathbf{s} \in \R^d\setminus \{\mathbf{0}\}}  \frac{ |\Phi(\mathbf{s}) -\Phi(\mathbf{0})|}{| \mathbf{s}|^{ 2}}\cdot \frac{x^2}{2}.
\end{align*}
 Hence $g$ satisfies all the conditions in \eqref{L-log-error} and we complete the proof of Proposition~\ref{prop-star-scale}. 
\end{proof}

It remains to prove Lemma~\ref{lem-base-fn}. 

\begin{proof}[Proof of Lemma~\ref{lem-base-fn}]
Multiplying a constant if necessary,   we only need to prove that there indeed exists  a  non-negative, positive definite and isotropic function $\Phi \in C^\infty(\R^d)$ such that 
\[
\supp(\Phi) \subset \bar{B}(\mathbf{0},1),\quad \Phi(\mathbf{0}) >0 \anand  \sup_{\mathbf{t}\in \R^d\setminus \{\mathbf{0}\}} \frac{|\Phi(\mathbf{t}) - \Phi(\mathbf{0})|}{|\mathbf{t}|^{2}}<\infty. 
\]
Take any non-negative  smooth function  $h\in C^\infty(\R)$ defined on  the real line $\R$ such that  
\[
h\not\equiv 0 \anand
\supp(h) =  \big[0,\frac{1}{4}\big].
\]
And  set
\begin{align}\label{def-phih}
\Phi_h(\mathbf{t})= \int_{\R^d} h(|\mathbf{x}-\mathbf{t}|^2) h(|\mathbf{x}|^2)\mathrm{d}\mathbf{x}, \quad \mathbf{t}\in \R^d. 
\end{align}

Now  let us show that this function $\Phi_h$ satisfies all the desired conditions.  

{\flushleft \it 1). $\Phi_h\in C^\infty(\R^d)$, $\Phi_h\geq0$,  $\Phi_h(\mathbf{0})>0$ and $\supp(\Phi_h) \subset \bar{B}(\mathbf{0},1)$.}
The assumptions on $h$ imply that
\[
\Phi_h(\mathbf{t})= \int_{\bar{B}(\mathbf{0},1/2)} h(|\mathbf{x}-\mathbf{t}|^2) h(|\mathbf{x}|^2)\mathrm{d}\mathbf{x}\geq0.
\]
The assumption $h\in C^\infty(\R)$ then implies that $\Phi_h\in C^\infty(\R^d)$.  Moreover, since $h$ is not identically zero, real-valued and belongs to $C^\infty(\R)$, we have  $\Phi_h(\mathbf{0}) >0$.  Finally, if $|\mathbf{t}|>1$, then for any $\mathbf{x} \in \bar{B}(\mathbf{0},1/2)$, 
\[
|\mathbf{x} - \mathbf{t}|^2 \ge (|\mathbf{t}|- |\mathbf{x}|)^2 \ge  1/4 
\]
and hence $h(|\mathbf{x}-\mathbf{t}|^2) =0$. It follows that $\Phi_h(\mathbf{t})=0$ for all $|\mathbf{t}|>1$. In other words, $\supp(\Phi_h) \subset \bar{B}(\mathbf{0}, 1)$.  

{\flushleft \it 2). $\Phi_h$ is isotropic.} Let $O(d)$ denote the group of orthogonal transformations of $\R^d$. By the definition \eqref{def-phih} and the $O(d)$-invariance of the Lebesgue measure $\mathrm{d}\mathbf{x}$, as well as the $O(d)$-invariance of the Euclidean metric, for any orthogonal transformation $U\in O(d)$, we have 
\begin{align*}
	\Phi_h(U \mathbf{t})&= \int_{\R^d} h(|\mathbf{x}-U\mathbf{t}|^2) h(|\mathbf{x}|^2)\mathrm{d}\mathbf{x} =  \int_{\R^d} h(|U^{-1}\mathbf{x}-\mathbf{t}|^2) h(|U^{-1}\mathbf{x}|^2)\mathrm{d}\mathbf{x} \\
	&= \int_{\R^d} h(|\mathbf{x}-\mathbf{t}|^2) h(|\mathbf{x}|^2)\mathrm{d}\mathbf{x}=\Phi_h(\mathbf{t}).
\end{align*}
Therefore, $\Phi_h$ is isotropic. 

{\flushleft \it 3). $\Phi_h$ is positive definite.} For any $\mathbf{t}\in \R^d$, define  a real-valued function
\[
\varphi_\mathbf{t} (\mathbf{x})= h(|\mathbf{x}-\mathbf{t}|^2). 
\]
Then by the translation-invariance of the Lebesgue measure, for any $\mathbf{t}, \mathbf{s}\in \R^d$, 
\[
\Phi_h(\mathbf{t}-\mathbf{s}) = \int_{\R^d} h(|\mathbf{x}-\mathbf{t} +\mathbf{s}|^2) h(|\mathbf{x}|^2)\mathrm{d}\mathbf{x}  = \int_{\R^d} h(|\mathbf{x}-\mathbf{t}|^2) h(|\mathbf{x}-\mathbf{s}|^2)\mathrm{d}\mathbf{x}  = \langle \varphi_\mathbf{t}, \varphi_\mathbf{s}\rangle_{L^2(\R^d)}. 
\]
This implies that $\Phi_h$ is positive definite. 

{\flushleft \it 4). $\Phi_h$ has second-order regularity at origin $\mathbf{0}$.} We now show that 
\begin{align}\label{0-lip-goal}
\sup_{\mathbf{t}\in \R^d\setminus \{\mathbf{0}\}} \frac{|\Phi_h(\mathbf{t}) - \Phi_h(\mathbf{0})|}{|\mathbf{t}|^{2}}<\infty.
\end{align}
For any $\mathbf{t}\in \R^d$, using the polar-coordinate system and the assumption $\supp(h) \subset [0,1/4]$, we have 
\begin{align}\label{phi-diff}
\begin{split}
\Phi_h(\mathbf{t}) - \Phi_h(\mathbf{0})& =  \int_{\R^d} \big(h(|\mathbf{x}-\mathbf{t}|^2) - h(|\mathbf{x}|^2)\big) \cdot h(|\mathbf{x}|^2)\mathrm{d}\mathbf{x} 
\\
& = \int_0^{1/2}  h(r^2) r^{d-1} \Big[  \underbrace{  \int_{\mathbb{S}^{d-1}}   \big(     h(|r \mathbf{v} - \mathbf{t}|^2)   - h(r^2)\big) \mathrm{d}\Sigma(\mathbf{v})}_{\text{denoted $T(r,  \mathbf{t})$}}\Big] \mathrm{d}r,
\end{split}
\end{align}
where $\mathbb{S}^{d-1}$ is the unit sphere in $\R^d$ equipped with the standard surface measure $\mathrm{d}\Sigma$. 

Now the assumptions that $h\in C^\infty(\R)$ and has compact support imply that there exists a bounded continuous function $\widetilde{h}(x,y)$ such that 
\[
h(y) - h(x) = h'(x) \cdot (y-x) + \widetilde{h}(x,y) \cdot (y-x)^2 \quad \text{for all $x, y\in \R$}.  
\]
Consequently, we have
\begin{align}\label{def-T12}
\begin{split}
	T(r, \mathbf{t})  =       h'(r^2)  \underbrace{\int_{\mathbb{S}^{d-1}}     (|r \mathbf{v} - \mathbf{t}|^2   -r^2) \mathrm{d}\Sigma(\mathbf{v})}_{\text{denoted $T_1(r, \mathbf{t})$}} 
	+ \underbrace{  \int_{\mathbb{S}^{d-1}}    \widetilde{h}(r^2, |r\mathbf{v}-\mathrm{t}|^2) \cdot (|r\mathbf{v} - \mathbf{t}|^2   -r^2)^2 \mathrm{d}\Sigma(\mathbf{v}).}_{\text{denoted $T_2(r, \mathbf{t})$}}
\end{split}
\end{align}
For the first  term $T_1(r, \mathbf{t})$, we have 
\begin{align}\label{T1-eq}
T_1(r, \mathbf{t}) = \int_{\mathbb{S}^{d-1}}     (|r \mathbf{v} - \mathbf{t}|^2   -r^2) \mathrm{d}\Sigma(\mathbf{v}) = \int_{\mathbb{S}^{d-1}}     (|\mathbf{t}|^2 -  2 \mathbf{v} \cdot \mathbf{t}) \mathrm{d}\Sigma(\mathbf{v})  =  |\mathbf{t}|^2 \cdot  \Sigma(\mathbb{S}^{d-1}),
\end{align}
where we used the following elementary identity  (since the map $\mathbf{v}\mapsto -\mathbf{v}$ preserves the measure $\mathrm{d}\Sigma$): 
\[
\int_{\mathbb{S}^{d-1}}    ( \mathbf{v} \cdot \mathbf{t})  \mathrm{d}\Sigma(\mathbf{v})  = 0. 
\]
For the second term $T_2(r, \mathbf{t})$, we have 
\begin{align*}
|T_2(r, \mathbf{t})| & \le  \|\widetilde{h}\|_{\infty} \cdot \Sigma(\mathbb{S}^{d-1}) \cdot    \sup_{\mathbf{v}\in \mathbb{S}^{d-1}}   (|r\mathbf{v} - \mathbf{t}|^2   -r^2)^2 
\\
& \le \|\widetilde{h}\|_{\infty} \cdot \Sigma(\mathbb{S}^{d-1}) \cdot  \sup_{\mathbf{v}\in \mathbb{S}^{d-1}}   (|\mathbf{t}|^2 - 2 r \mathbf{v}\cdot \mathbf{t})^2 
\\
& =  \|\widetilde{h}\|_{\infty} \cdot \Sigma(\mathbb{S}^{d-1}) \cdot   (|\mathbf{t}|^2  +2 r | \mathbf{t}|)^2,
\end{align*}
where $\|\widetilde{h}\|_\infty = \|\widetilde{h}\|_{L^\infty(\R^2)}<\infty$. Hence 
\begin{align}\label{T2-ineq}
|T_2(r, \mathbf{t})|  \le  \|\widetilde{h}\|_{\infty} \cdot \Sigma(\mathbb{S}^{d-1}) \cdot   (|\mathbf{t}|^2  + | \mathbf{t}|)^2 \quad \text{for all $r\in [0,1/2]$}.
\end{align}

It follows from \eqref{phi-diff}, \eqref{def-T12}, \eqref{T1-eq} and \eqref{T2-ineq} that  (with $C(h,d), C'(h,d)$  depending only on $h, d$)
\begin{align}\label{upp-phi-diff}
\begin{split}
|\Phi_h(\mathbf{t}) - \Phi_h(\mathbf{0})| & = \Big| \int_0^{1/2}  h(r^2) r^{d-1} h'(r^2) \mathrm{d}r \cdot \Sigma(\mathbb{S}^{d-1})  \cdot |\mathbf{t}|^2 +    \int_0^{1/2}  h(r^2) r^{d-1}  T_2(r, \mathbf{t}) \mathrm{d}r\Big| 
\\
& \le C(h, d) \cdot |\mathbf{t}|^2 + C'(h, d) \cdot (|\mathbf{t}|+1)^2 |\mathbf{t}|^2.
\end{split}
\end{align}
Therefore, combining \eqref{upp-phi-diff} with  $\supp(\Phi_h)\subset \bar{B}(\mathbf{0},1)$, we obtain 
\begin{align*}
\sup_{\mathbf{t}\in \R^d\setminus \{\mathbf{0}\}} \frac{|\Phi_h(\mathbf{t}) - \Phi_h(\mathbf{0})|}{|\mathbf{t}|^{2}} & \le \sup_{0<|\mathbf{t}|\le 1} \frac{|\Phi_h(\mathbf{t}) - \Phi_h(\mathbf{0})|}{|\mathbf{t}|^{2}} + \sup_{|\mathbf{t}| \ge 1} \frac{|\Phi_h(\mathbf{t}) - \Phi_h(\mathbf{0})|}{|\mathbf{t}|^{2}} 
\\
& \le  C(h,d) +4 C'(h,d) + \Phi_h(\mathbf{0})<\infty. 
\end{align*}
Thus we obtain the desired inequality \eqref{0-lip-goal}. 

The proof of Lemma~\ref{lem-base-fn} is complete. 
\end{proof}

\subsection{Comparison results of GMC}\label{sec-compare}
In this subsection, we  are going to deal with many kernels. To distinguish them, we shall use different letters (for instance, the positive definite kernel $L$ defined in \eqref{def-L-kernel} below will  in fact play the role of the kernel $K$ in Theorem~\ref{thm-gen-12DGMC} or Theorem~\ref{thm-gen-highD}).

We now consider the centered  log-correlated Gaussian field with the following kernel:  
\begin{align}\label{def-L-kernel}
L(\mathbf{t}, \mathbf{s}) = \log_{+} \Big(\frac{1}{|\mathbf{t}-\mathbf{s}|}\Big) + G(\mathbf{t}, \mathbf{s}), \quad \mathbf{t}, \mathbf{s}\in [0,1]^d,
\end{align}
with $G$  being bounded and continuous on $[0,1]^d \times [0,1]^d$ such that the conditions \eqref{error-cond} and \eqref{error-cond-bis} are satisfied. That is, 
\begin{align}\label{G-C2}
\sup_{\mathbf{t}, \mathbf{s}\in [0,1]^d
\atop \mathbf{t}\ne \mathbf{s}} \frac{|G(\mathbf{t},\mathbf{t}) -  G(\mathbf{s},\mathbf{s})|}{|\mathbf{t}-\mathbf{s}|}<\infty, \quad  \sup_{\mathbf{t}, \mathbf{s}\in [0,1]^d
\atop \mathbf{t}\ne \mathbf{s}} \frac{|G(\mathbf{t},\mathbf{t}) + G(\mathbf{s}, \mathbf{s})- 2 G(\mathbf{t},\mathbf{s})|}{|\mathbf{t}-\mathbf{s}|^2}<\infty
\end{align}
and  there is a constant $\lambda>0$ such that the following  kernel  is positive definite on $[0,1]^d \times [0,1]^d$:
\begin{align}\label{error-cond-bis}
R_\lambda(\mathbf{t}, \mathbf{s}) : = \lambda + G(\mathbf{t}, \mathbf{s}) - g (|\mathbf{t}-\mathbf{s}|) \quad \text{with} \, \, g(|\mathbf{t}-\mathbf{s}|) = \int_{|\mathbf{t}-\mathbf{s}|}^1 \frac{f(v)-1}{v} \mathrm{d}v,
\end{align}
and $f$ belongs to the class $\mathscr{P}_d$ introduced in \eqref{def-Pd}. 

In what follows, we define the kernel $L_\lambda$ by 
\begin{align}\label{def-L-lambda}
L_\lambda(\mathbf{t}, \mathbf{s}) = L(\mathbf{t}, \mathbf{s})  + \lambda, \quad \mathbf{t}, \mathbf{s}\in [0,1]^d.
\end{align}
Clearly, the  two GMC measures $\mathrm{GMC}_{L_\lambda}^{\gamma}$ and  $\mathrm{GMC}_{L}^{\gamma}$ can be coupled such that 
\begin{align}\label{coupling-gmc}
\mathrm{GMC}_{L_\lambda}^{\gamma} (\mathrm{d} \mathbf{t}) =  e^{\gamma Z - \lambda^2 \gamma^2 /2}  \cdot  \mathrm{GMC}_{L}^{\gamma} (\mathrm{d} \mathbf{t}),
\end{align}
where $Z$ is a standard Gaussian random variable independent of $\mathrm{GMC}_{L}^{\gamma}$.  In particular, under the coupling \eqref{coupling-gmc}, for any $\lambda>0$, almost surely, we have 
\begin{align}\label{coupling-eq}
\dim_F (\mathrm{GMC}_{L_\lambda}^{\gamma}) = \dim_F(\mathrm{GMC}_{L}^{\gamma}). 
\end{align}

The kernel $L$, or more precisely $L_\lambda$ in \eqref{def-L-kernel} or \eqref{def-L-lambda}, will be compared to the $*$-scale invariant kernel $K_{\mathrm{good}}$ defined in \eqref{K-dial-dec}. We recall that $K_{\mathrm{good}}$ has been  shown to be sharply-$\sigma^b(1)$-regular for any integer $b\ge 2$ in Proposition~\ref{prop-star-scale} and is of the form 
\begin{align}\label{def-K-good}
K_{\mathrm{good}}(\mathbf{t}, \mathbf{s})= \log_{+} \Big(\frac{1}{|\mathbf{t}-\mathbf{s}|}\Big) + g(|\mathbf{t}-\mathbf{s}|) =      \log_{+} \Big(\frac{1}{|\mathbf{t}-\mathbf{s}|}\Big)  + \int_{|\mathbf{t}-\mathbf{s}|}^1 \frac{f(v) - 1}{v} \mathrm{d} v, 
\end{align}
with $f$ belongs to the class $\mathscr{P}_d$ introduced in \eqref{def-Pd} and  $g\in C^1([0,\sqrt{d}\,])$ and $\supp(g)\subset [0,1]$ such that 
\begin{align}\label{g-zero-reg}
\sup_{x\in (0, 1]}\frac{|g(x)-g(0)|}{x^2} <\infty. 
\end{align}

\begin{lemma}\label{lem-K-lambda}
The positive definite kernel $R_\lambda$ defined in  \eqref{error-cond-bis} satisfies 
 \begin{align}\label{L-K-comp}
L_\lambda  = K_{\mathrm{good}} + R_\lambda
\end{align}
and 
\begin{align}\label{R-C2}
\sup_{\mathbf{t}, \mathbf{s}\in [0,1]^d
\atop \mathbf{t}\ne \mathbf{s}} \frac{|R_\lambda(\mathbf{t},\mathbf{t}) -  R_\lambda(\mathbf{s},\mathbf{s})|}{|\mathbf{t}-\mathbf{s}|}<\infty \anand \sup_{\mathbf{t}, \mathbf{s}\in [0,1]^d
\atop \mathbf{t}\ne \mathbf{s}} \frac{|R_\lambda(\mathbf{t},\mathbf{t}) + R_\lambda(\mathbf{s}, \mathbf{s})- 2 R_\lambda(\mathbf{t},\mathbf{s})|}{|\mathbf{t}-\mathbf{s}|^2}<\infty. 
\end{align}
\end{lemma}

 Lemma~\ref{lem-K-lambda} implies that, the centered Gaussian fields $X_{L_{\lambda}}$ and  $X_{K_{\mathrm{good}}}$ corresponding to the kernels $L_\lambda$ and $K_{\mathrm{good}}$ respectively can be coupled such that 
\begin{align}\label{L-good-couple}
X_{L _{\lambda}}= X_{K_{\mathrm{good}}} + Z_\lambda \quad \text{with $X_{K_{\mathrm{good}}},\,Z_\lambda$ independent},
\end{align}
where  
\begin{align}\label{def-Z-lambda}
Z_\lambda = \{Z_\lambda (\mathbf{t}): \mathbf{t}\in [0,1]^d \}
\end{align}
 is a centered Gaussian process (which  is defined for all $\mathbf{t}$) corresponding to the covariance kernel $R_\lambda$.
 
 \begin{remark*}
 Note that the decomposition \eqref{L-good-couple} implies that the two GMC measures $\mathrm{GMC}_{L_\lambda}^\gamma$ and $\mathrm{GMC}_{K_{\mathrm{good}}}^\gamma$ can be coupled such that 
\begin{align}\label{2-GMC-comp}
 \mathrm{GMC}_{L_\lambda}^\gamma(\mathrm{d} \mathbf{t})=      \exp\Big( \gamma Z_\lambda(\mathbf{t}) -  \frac{\gamma^2}{2}  \E[Z_\lambda(\mathbf{t})^2] \Big) \cdot \mathrm{GMC}_{K_{\mathrm{good}}}^\gamma(\mathrm{d}\mathbf{t}). 
\end{align}
 In general, due to the lack of sufficient regularity (H\"older regularity is not sufficient in the general case) of the sample path of the stochastic process $Z_\lambda(\mathbf{t})$, there is no a priori relation between the polynomial Fourier decay of two GMC measures $\mathrm{GMC}_{K_1}^\gamma$ and $\mathrm{GMC}_{K_1}^\gamma$  related through
 \[
  \mathrm{GMC}_{K_1}^\gamma(\mathrm{d} \mathbf{t})=      \exp\Big( \gamma Z_\lambda(\mathbf{t}) -  \frac{\gamma^2}{2}  \E[Z_\lambda(\mathbf{t})^2] \Big) \cdot \mathrm{GMC}_{K_2}^\gamma(\mathrm{d}\mathbf{t}). 
  \]
  Therefore, instead of using the measure-theoretical relation \eqref{2-GMC-comp}, we are going to use the more fundamental relation of the kernels $L_\lambda  = K_{\mathrm{good}} + R_\lambda$. Indeed,  since $K_{\mathrm{good}}$ has a very special structure of sharp-$\sigma$-regularity in the sense of Definition~\ref{def-sigma-reg},  by  requiring a regularity condition \eqref{G-C2} on $G$ and thus on $R_\lambda$, we are able to show that the kernel $L_\lambda$  also has sharp-$\sigma$-regularity (in the sense of non-stationary kernels). Then  by applying Theorem~\ref{Fou-Dec-abs}, we show that the two GMC measures $\mathrm{GMC}_{L_\lambda}^\gamma$ and $\mathrm{GMC}_{K_{\mathrm{good}}}^\gamma$ share the same Fourier dimensions.
 \end{remark*}

For any $\gamma\in (0,\sqrt{2d})$, define a stochastic process 
$
\{P_{\gamma}^{(\lambda)}(\mathbf{t}): \mathbf{t} \in [0,1]^d\}
$ by 
\begin{align}\label{def-P-lambda}
P_\gamma^{(\lambda)}(\mathbf{t})=   \exp\Big( \gamma Z_\lambda(\mathbf{t}) -  \frac{\gamma^2}{2}  \E[Z_\lambda(\mathbf{t})^2] \Big) =   \exp\Big( \gamma Z_\lambda(\mathbf{t}) -  \frac{\gamma^2}{2} R_\lambda(\mathbf{t}, \mathbf{t}) \Big). 
\end{align}

\begin{lemma}\label{lem-gp}
Assume that the conditions in Lemma~\ref{lem-K-lambda} are satisfied.  Then for any $\gamma\in (0, \sqrt{2d})$,
 \[
	\sup_{\mathbf{t},\mathbf{s}\in [0,1]^d \atop \mathbf{t}\neq\mathbf{s}}\mathbb{E}\Big[\Big| \frac{ P_\gamma^{(\lambda)}(\mathbf{t}) -P_\gamma^{(\lambda)}(\mathbf{s})}{|\mathbf{t}-\mathbf{s}|} \Big|^{2}\Big]<\infty.
	\]
\end{lemma}

Recall again that $K_{\mathrm{good}}$ has been  shown to be sharply-$\sigma^b(1)$-regular for any integer $b\ge 2$ in Proposition~\ref{prop-star-scale}. Therefore, the decomposition \eqref{def-Lj} of the kernel $K_{\mathrm{good}}$  gives rise to a decomposition of the Gaussian field $X_{K_{\mathrm{good}}}$ into independent Gaussian processes, which is  informally written as 
\begin{align}\label{X-good-dec}
X_{K_{\mathrm{good}}}  = \sum_{j=0}^\infty \psi_j. 
\end{align}
For any $\gamma\in (0, \sqrt{2d})$,  the decomposition  \eqref{X-good-dec} gives rise to the sequence of indepedent stochastic processes  defined as in \eqref{def-P-g-j}: 
\begin{align}\label{P-good}
\Big\{\XX_{\gamma, j}(\mathbf{t})=   \exp\big( \gamma \psi_j(\mathbf{t}) - \frac{\gamma^2}{2} \E[\psi_j(\mathbf{t})^2] \big): \mathbf{t} \in [0,1)^d \Big\}_{j\ge 0}
\end{align}
Moreover, by Proposition~\ref{prop-kernel-ass}, the sequence  \eqref{P-good} of independent stochastic processes  satisfies Assumptions~\ref{assum-indep-abs}, \ref{assum-Lp0-abs} and \ref{assum-alp0-abs} for $\alpha_0 = 1$ (since here $K_{\mathrm{good}}$ is sharply-$\sigma^b(1)$-regular) and for all  $1< p_0 <  \min\{\frac{2d}{\gamma^2}, 2\}$. 

By using the decomposition \eqref{L-good-couple} of the kernel $L_\lambda$ and the decomposition \eqref{X-good-dec} of the Gaussian field $X_{K_{\mathrm{good}}}$, we obtain the decomposition  of the Gaussian field $X_{L_\lambda}$ into indepedent Gaussian processes:  
\begin{align}\label{X-L-lambda-dec}
X_{L_\lambda}  =  \sum_{j=0}^\infty \widetilde{\psi}_j \quad \text{with $\widetilde{\psi}_0 = Z_\lambda + \psi_0$ and $\widetilde{\psi}_j = \psi_j$ for all $j\ge 1$}. 
\end{align}
Here $Z_\lambda$  is the Gaussian process introduced in \eqref{def-Z-lambda}. In particular, $Z_\lambda$ is independent of the Gaussian field $X_{K_{\mathrm{good}}}$ and hence of all the stochastic processes in $\{\psi_j\}_{j\ge 0}$.  Now, for any $\gamma\in (0, \sqrt{2d})$, the above decomposition \eqref{X-L-lambda-dec} gives rise to the following sequence of indepedent stochastic processes  defined as in \eqref{def-P-g-j}: 
\begin{align}\label{P-mod}
\Big\{\widetilde{\XX}_{\gamma, j}^{(\lambda)}(\mathbf{t})=   \exp\big( \gamma \widetilde{\psi}_j(\mathbf{t}) - \frac{\gamma^2}{2} \E[\widetilde{\psi}_j(\mathbf{t})^2] \big): \mathbf{t} \in [0,1)^d \Big\}_{j\ge 0}.
\end{align}
Clearly, we have 
\begin{align}\label{zero-gp}
\widetilde{\XX}_{\gamma, 0}^{(\lambda)}(\mathbf{t})  =  P_\gamma^{(\lambda)}(\mathbf{t}) \cdot \XX_{\gamma, 0} (\mathbf{t}) \anand  \widetilde{\XX}_{\gamma, j}^{(\lambda)}(\mathbf{t}) = \XX_{\gamma, j}(\mathbf{t}) \text{ for all $j\ge 1$},
\end{align}
where $P_\gamma^{(\lambda)}(\mathbf{t})$ is given in \eqref{def-P-lambda}. 

\begin{lemma}\label{lem-3-ass-new}
Assume that the conditions in Lemma~\ref{lem-K-lambda} are satisfied. Then for any $\gamma\in (0, \sqrt{2d})$, the sequence  of independent stochastic processes given in \eqref{P-mod}   satisfies Assumptions~\ref{assum-indep-abs}, \ref{assum-Lp0-abs} and \ref{assum-alp0-abs} for $\alpha_0 = 1$ and for any  $1< p_0 <  \min\{\frac{2d}{\gamma^2}, 2\}$. 
\end{lemma}

Now we proceed to the proofs of Lemmas~\ref{lem-K-lambda}, \ref{lem-gp} and \ref{lem-3-ass-new}. 

\begin{proof}[Proof of Lemma~\ref{lem-K-lambda}]
The equality \eqref{L-K-comp} follows from the definition~\eqref{def-L-kernel} of $L$,  the definition~\eqref{error-cond-bis} of $R_\lambda$,  the definition~\eqref{def-L-lambda}  of $L_\lambda$ and the definition \eqref{def-K-good} of $K_{\mathrm{good}}$.  
By combining the equalities 
\[
\frac{R_\lambda(\mathbf{t},\mathbf{t}) -  R_\lambda(\mathbf{s},\mathbf{s})}{|\mathbf{t}-\mathbf{s}|} =  \frac{G(\mathbf{t},\mathbf{t}) -  G(\mathbf{s},\mathbf{s})}{|\mathbf{t}-\mathbf{s}|}
\]
and 
\[
\frac{R_\lambda(\mathbf{t}, \mathbf{t}) + R_\lambda(\mathbf{s}, \mathbf{s}) - 2 R_\lambda(\mathbf{t}, \mathbf{s})}{|\mathbf{t}-\mathbf{s}|^2} =  \frac{G(\mathbf{t}, \mathbf{t}) + G(\mathbf{s}, \mathbf{s}) - 2 G(\mathbf{t}, \mathbf{s})}{|\mathbf{t}-\mathbf{s}|^2}  - \frac{2g(0)  - 2 g(|\mathbf{t}-\mathbf{s}|)}{|\mathbf{t}-\mathbf{s}|^2}
\]
with the conditions \eqref{G-C2} and \eqref{g-zero-reg}, we obtain the desired inequalities \eqref{R-C2}.  
\end{proof}
 
\begin{proof}[Proof of Lemma~\ref{lem-gp}]
By the definition \eqref{def-P-lambda}, we have 
\[
\E[|  P_\gamma^{(\lambda)}(\mathbf{t}) -P_\gamma^{(\lambda)}(\mathbf{s})|^2] =  \exp\big(\gamma^2 R_\lambda(\mathbf{t}, \mathbf{t})\big) + \exp\big(\gamma^2 R_\lambda(\mathbf{s}, \mathbf{s})\big)- 2   \exp\big(\gamma^2 R_\lambda(\mathbf{t}, \mathbf{s})\big). 
\]
Thus we can write 
\begin{align*}
\E[|  P_\gamma^{(\lambda)}(\mathbf{t}) -P_\gamma^{(\lambda)}(\mathbf{s})|^2]   &  = \underbrace{\exp\big(\gamma^2 R_\lambda(\mathbf{t}, \mathbf{t})\big) + \exp\big(\gamma^2 R_\lambda(\mathbf{s}, \mathbf{s})\big)-  2 \exp\Big(\frac{\gamma^2 R_\lambda(\mathbf{t}, \mathbf{t}) + \gamma^2 R_\lambda(\mathbf{s}, \mathbf{s})}{2}\Big)}_{\text{denoted $A_{\gamma, \lambda}(\mathbf{t}, \mathbf{s})$}}
\\
& \quad\,\,+ \underbrace{2 \exp\Big(\frac{\gamma^2 R_\lambda(\mathbf{t}, \mathbf{t}) + \gamma^2 R_\lambda(\mathbf{s}, \mathbf{s})}{2}\Big)  -  2   \exp\big(\gamma^2 R_\lambda(\mathbf{t}, \mathbf{s})\big)}_{\text{denoted $B_{\gamma, \lambda}(\mathbf{t}, \mathbf{s})$}}. 
\end{align*}
Using the hyperbolic cosine function $\cosh(x) = \frac{e^{x}+e^{-x}}{2}$, we obtain 
\[
A_{\gamma, \lambda}(\mathbf{t}, \mathbf{s}) =   2 \exp\Big(\frac{\gamma^2 R_\lambda(\mathbf{t}, \mathbf{t}) + \gamma^2 R_\lambda(\mathbf{s}, \mathbf{s})}{2}\Big) \Big[  \cosh\Big(\frac{\gamma^2 R_\lambda(\mathbf{t}, \mathbf{t}) - \gamma^2 R_\lambda(\mathbf{s}, \mathbf{s})}{2}\Big)  -1\Big].
\]
Then, since $(\cosh(x) -1)/x^2$ is continuous and   $R_\lambda$ is bounded continuous on $[0,1]^d\times [0,1]^d$, there exists a constant $C = C(\gamma, \|R_\lambda\|_\infty)>0$ depending on $\gamma$ and the $L^\infty$-norm  of $R_{\lambda}$ on  $[0,1]^d\times [0,1]^d$ such that 
\[
|A_{\gamma, \lambda}(\mathbf{t}, \mathbf{s})| \le 2 \exp(\gamma^2 \|R_\lambda\|_\infty) \cdot C \cdot |R_\lambda(\mathbf{t}, \mathbf{t}) - R_\lambda(\mathbf{s}, \mathbf{s}) |^2.
\]
Therefore, by applying the first inequality in \eqref{R-C2}, we obtain 
\begin{align}\label{A-C2}
\sup_{\mathbf{t}, \mathbf{s}\in [0,1]^d
\atop \mathbf{t}\ne \mathbf{s}} \frac{|A_{\gamma, \lambda}(\mathbf{t}, \mathbf{s})|}{|\mathbf{t}-\mathbf{s}|^2}<\infty. 
\end{align}

For the term $B_{\gamma, \lambda}(\mathbf{t}, \mathbf{s})$, we have 
\[
B_{\gamma, \lambda}(\mathbf{t}, \mathbf{s}) =    2   \exp\big(\gamma^2 R_\lambda(\mathbf{t}, \mathbf{s})\big) \Big[   \exp\Big(\frac{\gamma^2}{2}\big[ R_\lambda(\mathbf{t}, \mathbf{t}) +  R_\lambda(\mathbf{s}, \mathbf{s})-2 R_\lambda(\mathbf{t},\mathbf{s})\big]\Big)  -  1\Big].
\]
Again since   $R_\lambda$ is bounded continuous on $[0,1]^d\times [0,1]^d$ and the function $(e^x -1)/x$ is continuous, there exists a constant $C' = C'(\gamma, \|R_\lambda\|_\infty)>0$ such that 
\[
|B_{\gamma, \lambda}(\mathbf{t}, \mathbf{s})| \le 2 \exp(\gamma^2 \|R_\lambda\|_\infty) \cdot C' \cdot  |R_\lambda(\mathbf{t}, \mathbf{t}) +  R_\lambda(\mathbf{s}, \mathbf{s})-2 R_\lambda(\mathbf{t},\mathbf{s})|.
\]
Therefore, by applying the second inequality in \eqref{R-C2}, we obtain 
\begin{align}\label{B-C2}
\sup_{\mathbf{t}, \mathbf{s}\in [0,1]^d
\atop \mathbf{t}\ne \mathbf{s}} \frac{|B_{\gamma, \lambda}(\mathbf{t}, \mathbf{s})|}{|\mathbf{t}-\mathbf{s}|^2}<\infty. 
\end{align}
Combining \eqref{A-C2} and \eqref{B-C2}, we complete the proof of Lemma~\ref{lem-gp}. 
\end{proof}

\begin{proof}[Proof of Lemma~\ref{lem-3-ass-new}]
By Proposition~\ref{prop-kernel-ass}, the sequence  \eqref{P-good} of independent stochastic processes  satisfies Assumptions~\ref{assum-indep-abs}, \ref{assum-Lp0-abs} and \ref{assum-alp0-abs} for $\alpha_0 = 1$ (since here $K_{\mathrm{good}}$ is sharply-$\sigma^b(1)$-regular) and for all  $1< p_0 <  \min\{\frac{2d}{\gamma^2}, 2\}\le 2$.  Therefore, it suffices to  prove 
\[
\sup_{\mathbf{t},\mathbf{s}\in [0,1]^d \atop \mathbf{t}\neq\mathbf{s}}\mathbb{E}\Big[\Big| \frac{ P_\gamma^{(\lambda)}(\mathbf{t}) \cdot \XX_{\gamma, 0} (\mathbf{t}) -P_\gamma^{(\lambda)}(\mathbf{s}) \cdot \XX_{\gamma, 0} (\mathbf{s})}{|\mathbf{t}-\mathbf{s}|} \Big|^{p_0}\Big]<\infty.
\]
Since  \eqref{zero-gp} and $1<p_0\le 2$, it suffices to prove
\begin{align}\label{desired-2-zero}
\sup_{\mathbf{t},\mathbf{s}\in [0,1]^d \atop \mathbf{t}\neq\mathbf{s}}\mathbb{E}\Big[\Big| \frac{ P_\gamma^{(\lambda)}(\mathbf{t}) \cdot \XX_{\gamma, 0} (\mathbf{t}) -P_\gamma^{(\lambda)}(\mathbf{s}) \cdot \XX_{\gamma, 0} (\mathbf{s})}{|\mathbf{t}-\mathbf{s}|} \Big|^{2}\Big]<\infty. 
\end{align}
But since $P_\gamma^{(\lambda)}$ and  $\XX_{\gamma, 0}$ are independent, using $(a+b)^2\le 2a^2 + 2b^2$,  we have 
\begin{align*}
& \quad \,\, \E\big[ \big| P_\gamma^{(\lambda)}(\mathbf{t}) \cdot \XX_{\gamma, 0} (\mathbf{t}) -P_\gamma^{(\lambda)}(\mathbf{s}) \cdot \XX_{\gamma, 0} (\mathbf{s})\big|^2\big]
\\
& \le  2 \E\big[ \big| P_\gamma^{(\lambda)}(\mathbf{t}) \cdot \XX_{\gamma, 0} (\mathbf{t}) -P_\gamma^{(\lambda)}(\mathbf{t}) \cdot \XX_{\gamma, 0} (\mathbf{s})\big|^2\big]
+ 2 \E\big[ \big| P_\gamma^{(\lambda)}(\mathbf{t}) \cdot \XX_{\gamma, 0} (\mathbf{s}) -P_\gamma^{(\lambda)}(\mathbf{s}) \cdot \XX_{\gamma, 0} (\mathbf{s})\big|^2\big] 
\\
& =  2 \E\big[ P_\gamma^{(\lambda)}(\mathbf{t})^2\big]\cdot \E \big[\big| \XX_{\gamma, 0} (\mathbf{t}) - \XX_{\gamma, 0} (\mathbf{s})\big|^2\big] + 2 \E\big[ \big| P_\gamma^{(\lambda)}(\mathbf{t})  -P_\gamma^{(\lambda)}(\mathbf{s})\big|^2\big]  \cdot \E\big[ \XX_{\gamma, 0} (\mathbf{s})^2\big]
\\
 & =     2  \exp\big(\gamma^2 R_\lambda(\mathbf{t}, \mathbf{t})\big) \cdot \E \big[\big| \XX_{\gamma, 0} (\mathbf{t}) - \XX_{\gamma, 0} (\mathbf{s})\big|^2\big]  +  2 \E\big[ \big| P_\gamma^{(\lambda)}(\mathbf{t})  -P_\gamma^{(\lambda)}(\mathbf{s})\big|^2\big]  \cdot \E\big[ \XX_{\gamma, 0} (\mathbf{0})^2\big],
\end{align*}
where we used the translation-invariance of the stochastic process $P_{\gamma, 0}$.  Now the desired inequality \eqref{desired-2-zero} follows from  the inequality \eqref{hol-goal} and   Lemma~\ref{lem-gp} for $P_{\gamma}^{(\lambda)}$, as well as the boundedness of the continuous function $R_\lambda$ on $[0,1]^d\times [0,1]^d$. 
\end{proof}

\subsection{Proofs of Theorem~\ref{thm-gen-12DGMC} and  Theorem~\ref{thm-gen-highD}}
Here we use the notation in \S\ref{sec-compare}.    In particular, the kernel $L$ below in fact play the role of the kernel $K$ in Theorem~\ref{thm-gen-12DGMC} or Theorem~\ref{thm-gen-highD}.  

For any integer  $d\ge 1$ and  fix any $\gamma\in (0, \sqrt{2d})$.  Consider a positive definite kernel  $L$ of the form \eqref{def-L-kernel} such that the assumption \eqref{G-C2} is satisfied.  Recall the definition  \eqref{def-L-lambda} of the kernel $L_\lambda$.  Then by the equality \eqref{coupling-eq}, it suffices to show that for large enough $\lambda>0$, almost surely, 
\begin{align}\label{low-upp-bdd}
\min\{2,D_{\gamma,d}\}\leq  \mathrm{dim}_{F}(\mathrm{GMC}_{L_\lambda}^\gamma) = \mathrm{dim}_{F}(\mathrm{GMC}_{L}^\gamma)  \leq D_{\gamma,d}.
\end{align}
The upper bound of the inequality \eqref{low-upp-bdd} follows from  Lemma~\ref{Uppbou-FourierDim-dDGMC}.   The lower bound follows immediately from Lemma~\ref{lem-3-ass-new} and our unified Theorem~\ref{Fou-Dec-abs}, together with exactly the same elementary compuations as in \eqref{def-zeta-p-1} and \eqref{def-zeta-p-2}.

\section{Fourier dimensions of Mandelbrot random coverings}\label{S-poi-abs}
This section is devoted to the derivation of Theorem~\ref{thm-manrancov-poi} and Theorem~\ref{salemcanonicalcase-poi} from our unified Theorem~\ref{Fou-Dec-abs}.  The main step is to prove   that the  MRC measure is a multiplicative chaos measure associated to a sequence of independent stochastic processes satisfying Assumptions~\ref{assum-indep-abs}, \ref{assum-Lp0-abs} and \ref{assum-alp0-abs}.

\subsection{Multiplicative structure of MRC measures}
Recall the definitions of $E_{\Lambda}(\epsilon)$ in \eqref{def-E-set} and $\mathrm{MRC}_\Lambda^\epsilon$ in \eqref{def-muPCepsilon-poi}. 
For simplifying notation,  given any $\epsilon>0$,  denote 
\[
G_\Lambda(\epsilon):= \mathcal{U}[\mathrm{PPP}(\omega_{\Lambda})\cap S_\epsilon] =    \bigcup_{z\in \mathrm{PPP}(\omega_{\Lambda})\cap S_\epsilon}I_z.
\]
Then the random measure $\mathrm{MRC}_\Lambda^\epsilon$ on the unit interval $[0,1]$ can be rewritten as
\begin{align}\label{def-muPCm-poi}
	\mathrm{MRC}_\Lambda^\epsilon(\mathrm{d}t)=\frac{1-\mathds{1}_{G_\Lambda(\epsilon)}(t)}{1-\mathbb{E}[\mathds{1}_{G_\Lambda(\epsilon)}(t)]}\mathrm{d}t.
\end{align}

Fix any integer $b\geq2$. For any integer $m\geq1$ and any $t\in[0,1]$, define 
\[
D^{b}_{m}(t):=\Big\{(x,y)\in\mathbb{R}\times(0,1)\,\Big|\,b^{-m}\leq y<1\an t-y<x<t\Big\}.
\]
Clearly,  $t\notin G_\Lambda(b^{-m})$ if and only if  $\mathrm{PPP}(\omega_{\Lambda})\cap D^{b}_{m}(t) =\varnothing$. Hence by \eqref{def-muPCm-poi},
\begin{align}\label{muPCm-num-poi}
	\mathrm{MRC}_\Lambda^{b^{-m}}(\mathrm{d}t)=\frac{\mathds{1} (\mathrm{PPP}(\omega_{\Lambda})\cap D^{b}_{m}(t) =\varnothing )}{\PP( \mathrm{PPP}(\omega_{\Lambda})\cap D^{b}_{m}(t) =\varnothing )}\mathrm{d}t.
\end{align}

For any $t\in[0,1]$ and any integer $j\geq1$, consider the set
\begin{align}\label{def-Delm-poi}
	\Delta^{b}_j(t):=\Big\{(x,y)\in\mathbb{R}\times(0,1)\,\Big|\,b^{-j}\leq y<b^{-(j-1)}\an t-y<x<t\Big\}
\end{align}
and define
\begin{align}\label{muPCm-numDelta-poi}
	\mathcal{X}_{b,0}(t):\equiv1\anand\mathcal{X}_{b,j}(t):=\frac{\mathds{1}(\mathrm{PPP}(\omega_{\Lambda})\cap\Delta^{b}_j(t) = \varnothing)}{\mathbb{P}(\mathrm{PPP}(\omega_{\Lambda})\cap\Delta^{b}_j(t) = \varnothing)}\text{ for each $j\geq1$}.
\end{align}
Then we obtain a sequence of independent stochastic processes:
\[
\{\mathcal{X}_{b,j}(t): t\in [0,1]\}_{j\ge 0}  \quad \text{with $\mathcal{X}_{b,j}(t)\geq0$ and $\E[\mathcal{X}_{b,j}(t)]\equiv 1$}.
\]
Note that for any $t\in[0,1]$,
\[
D^{b}_{m}(t)=\bigsqcup_{j=0}^{m}\Delta^{b}_j(t)
\anand
\mathds{1} (\mathrm{PPP}(\omega_{\Lambda})\cap D^{b}_{m}(t) =\varnothing) =  \prod_{j=0}^m  \mathds{1}(\mathrm{PPP}(\omega_{\Lambda})\cap\Delta^{b}_j(t) = \varnothing).
\]
Therefore, by  \eqref{muPCm-num-poi} and \eqref{muPCm-numDelta-poi}, we have 
\begin{align}\label{def-muPCb-m-poi}
	\mathrm{MRC}_\Lambda^{b^{-m}}(\mathrm{d}t)=\Big[\prod_{j=0}^{m}\mathcal{X}_{b,j}(t)\Big] \mathrm{d}t.
\end{align}
Hence, by comparing \eqref{def-muPCb-m-poi} and \eqref{weacon-muPC-poi} , we have the following limit in the sense of weak convergence:
\begin{align}\label{weacon-muPCm-poi}
	\lim_{m\to\infty}\mathrm{MRC}_\Lambda^{b^{-m}}=\mathrm{MRC}_\Lambda.
\end{align}

For the fixed integer $b\geq2$, recall the definition of $\chi(b,\Lambda)$ in \eqref{def-chi-b}.

\begin{proposition}\label{prop-manrancov-ass}
Suppose that $\chi(b,\Lambda)<1$, then the independent stochastic processes defined as \eqref{muPCm-numDelta-poi}: 
\[
\{\mathcal{X}_{b,j}(t): t\in [0,1]\}_{j\ge 0}
\]
satisfies Assumptions~\ref{assum-indep-abs}, \ref{assum-Lp0-abs} and \ref{assum-alp0-abs} for $\alpha_0=1/2$ and $p_0=2$. 
\end{proposition}

\subsection{Proof of Proposition~\ref{prop-manrancov-ass}}
One can verify Assumption~\ref{assum-indep-abs} by the following Lemma~\ref{elem-prop-poi}. And it follows from Lemmas~\ref{p-moment-X_PCm-poi}  and \ref{Holder-p-moment-X_PCm-poi} that  Assumptions~\ref{assum-Lp0-abs} and \ref{assum-alp0-abs} hold for $\alpha_0=1/2$ and $p_0=2$. 

\begin{lemma}\label{elem-prop-poi}
	The stochastic processes $\mathcal{X}_{b,j}$ satisfy the following properties:
	\begin{itemize}
		\item[(P1)] The stochastic processes $\{\mathcal{X}_{b,j}\}_{j\geq1}$ are  independent$;$  
		\item[(P2)] For any sub-intervals $T,S\subset[0,1]$ satisfying
		\[
		\mathrm{dist}(T,S) = \inf\{|t_1-t_2|: t_1\in T\,\,\mathrm{and\,\,}t_2 \in S\} \geq  b^{-(j-1)},
		\]
		the stochastic processes $\{\mathcal{X}_{b,j}(t):t\in T\}$ and $\{\mathcal{X}_{b,j}(t): t\in S\}$ are independent.
	\end{itemize}
\end{lemma}

\begin{lemma}\label{p-moment-X_PCm-poi}
	Suppose that $\chi(b,\Lambda)<1$, then for any $p>1$, 
	\[
	\limsup_{j\to\infty}\sup_{t\in[0,1]}\mathbb{E}[\mathcal{X}_{b,j}^{p}(t)]=\exp\big((p-1)\chi(b,\Lambda)\log b\big)<b^{p-1}.
	\]
\end{lemma}

\begin{lemma}\label{Holder-p-moment-X_PCm-poi}
	Suppose that $\chi(b,\Lambda)<1$, then
	\[
	\sup_{j\in\N}\sup_{0<|t-s|\leq b^{-j} }\mathbb{E}\Big[\Big| \frac{\mathcal{X}_{b,j}(t)-\mathcal{X}_{b,j}(s)}{\sqrt{b^{j}  |t-s|}} \Big|^2\Big]<\infty.
	\]
\end{lemma}

\begin{proof}[Proof of Lemma~\ref{elem-prop-poi}]
Clearly, the family of sets $\{\Delta^{b}_j(t):t\in[0,1],\,j\geq1\}$ defined in \eqref{def-Delm-poi} satisfies:
\begin{itemize}
	\item   $\Delta^{b}_j(t)$ and $\Delta^{b}_k(s)$ are disjoint for any $j\neq k$ and any $t,s\in[0,1]$;
	\item  $\Delta^{b}_j(t_1)$ and $\Delta^{b}_j(t_2)$ are disjoint for any $t_1, t_2\in[0,1]$ satisfying $|t_1-t_2|\geq b^{-(j-1)}$. Consequently, for any sub-intervals $T,S\subset[0,1]$ satisfying
	\[
	\mathrm{dist}(T,S) = \inf\{|t_1-t_2|: t_1\in T\,\,\mathrm{and}\,\,t_2 \in S\} \geq  b^{-(j-1)},
	\]
	the two subsets $\bigcup\limits_{t_1 \in T} \Delta^{b}_j(t_1)$ and $\bigcup\limits_{t_2\in S} \Delta^{b}_j(t_2)$ are disjoint.
\end{itemize}
Then  Lemma~\ref{elem-prop-poi} follows immediately by the elementary property of Poisson point process. 
\end{proof}

\begin{proof}[Proof of Lemma~\ref{p-moment-X_PCm-poi}]
Fix any $t\in[0,1]$ and $p>1$. By the definitions \eqref{def-Delm-poi} and \eqref{muPCm-numDelta-poi} for $\Delta^{b}_j(t)$ and $\mathcal{X}_{b,j}(t)$ respectively, for any  $j\ge 1$,  we have
\[
\mathbb{E}[\mathcal{X}_{b,j}^{p}(t)]=\mathbb{E}\Big[\Big|\frac{\mathds{1}(\mathrm{PPP}(\omega_{\Lambda})\cap\Delta^{b}_j(t) = \varnothing)}{\mathbb{P}(\mathrm{PPP}(\omega_{\Lambda})\cap\Delta^{b}_j(t) = \varnothing)}\Big|^p\Big]=[\mathbb{P}(\mathrm{PPP}(\omega_{\Lambda})\cap\Delta^{b}_j(t) = \varnothing)]^{1-p}.
\]
Recall that $\omega_\Lambda(\mathrm{d}x\mathrm{d}y)= \mathrm{d}x  \otimes \Lambda(\mathrm{d}y)$. Then by the elementary property of Poisson point process, we have 
\begin{align}\label{cal-numdel=0-poi}
		\mathbb{P} (\mathrm{PPP}(\omega_{\Lambda})\cap\Delta^{b}_j(t) = \varnothing)=\exp\Big(-\int_{[b^{-j},b^{-(j-1)})}y\Lambda(\mathrm{d}y)\Big)
\end{align}
and hence
\begin{align}\label{cal-pmomentofXPCm-poi}
	\mathbb{E}[\mathcal{X}_{b,j}^{p}(t)]=\exp\Big((p-1)\int_{[b^{-j},b^{-(j-1)})}y\Lambda(\mathrm{d}y)\Big).
\end{align}
Then by the condition $\chi(b,\Lambda)<1$,
we obtain
\[
\limsup_{j\to\infty}\sup_{t\in[0,1]}\mathbb{E}[\mathcal{X}_{b,j}^{p}(t)]=\exp\big((p-1)\chi(b,\Lambda)\log b\big)<b^{p-1}.
\]
This completes the proof of Lemma~\ref{p-moment-X_PCm-poi}.
\end{proof}

\begin{proof}[Proof of Lemma~\ref{Holder-p-moment-X_PCm-poi}]
We only need to consider $j\geq1$ since $\mathcal{X}_{b,0}(t)\equiv1$. Fix $j\geq1$ and $t,s\in[0,1]$ satisfying $|t-s|\leq b^{-j}$, without loss of generality, we may assume that $t>s$. Write 
\begin{align}\label{cal1-2momentdiffer-poi}
	\mathbb{E}[|\mathcal{X}_{b,j}(t)-\mathcal{X}_{b,j}(s)|^2]=\mathbb{E}[\mathcal{X}^2_{b,j}(t)]+\mathbb{E}[\mathcal{X}^2_{b,j}(s)]-2\mathbb{E}[\mathcal{X}_{b,j}(t)\mathcal{X}_{b,j}(s)].
\end{align}
It follows from \eqref{cal-pmomentofXPCm-poi} that
\begin{align}\label{cal2-2momentdiffer-poi}
	\mathbb{E}[\mathcal{X}^2_{b,j}(t)]=\mathbb{E}[\mathcal{X}^2_{b,j}(s)]=\exp\Big(\int_{[b^{-j},b^{-(j-1)})}y\Lambda(\mathrm{d}y)\Big).
\end{align}

By the definition of $\mathcal{X}_{b,j}$ in \eqref{muPCm-numDelta-poi}, we have
\begin{align*}
\mathbb{E}[\mathcal{X}_{b,j}(t)\mathcal{X}_{b,j}(s)] & =\frac{\mathbb{E}[\mathds{1}(\mathrm{PPP}(\omega_{\Lambda})\cap\Delta^{b}_j(t) = \varnothing)\cdot\mathds{1}(\mathrm{PPP}(\omega_{\Lambda})\cap\Delta^{b}_j(s) = \varnothing)]}{\mathbb{P}(\mathrm{PPP}(\omega_{\Lambda})\cap\Delta^{b}_j(t) = \varnothing)\cdot\mathbb{P}(\mathrm{PPP}(\omega_{\Lambda})\cap\Delta^{b}_j(s) = \varnothing)}
\\
& = \frac{\mathbb{P}(\mathrm{PPP}(\omega_{\Lambda})\cap(\Delta^{b}_j(t) \cup \Delta_j^b(s)) = \varnothing)}{\mathbb{P}(\mathrm{PPP}(\omega_{\Lambda})\cap\Delta^{b}_j(t) = \varnothing)\cdot\mathbb{P}(\mathrm{PPP}(\omega_{\Lambda})\cap\Delta^{b}_j(s) = \varnothing)}.
\end{align*}
Then by using the partition  $\Delta^{b}_j(t) \cup \Delta_j^b(s)= \Delta_j^b(t) \sqcup    (\Delta_j^b(s)\setminus \Delta_j^b(t))$, we obtain 
\begin{align*}
 & \quad \,\, \PP(\mathrm{PPP}(\omega_{\Lambda})\cap(\Delta^{b}_j(t) \cup \Delta_j^b(s)) = \varnothing)
\\
&= 
\PP(\mathrm{PPP}(\omega_{\Lambda})\cap\Delta^{b}_j(t) = \varnothing)\cdot \PP(\mathrm{PPP}(\omega_{\Lambda})\cap(\Delta^{b}_j(s)  \setminus \Delta_j^b(t))
= \varnothing).
\end{align*}
Hence 
\[
\mathbb{E}[\mathcal{X}_{b,j}(t)\mathcal{X}_{b,j}(s)]=\frac{\PP(\mathrm{PPP}(\omega_{\Lambda})\cap(\Delta^{b}_j(s)  \setminus \Delta_j^b(t))
= \varnothing)}{\mathbb{P}(\mathrm{PPP}(\omega_{\Lambda})\cap\Delta^{b}_j(s) = \varnothing)}.
\]
 Recall that $\omega_\Lambda(\mathrm{d}x\mathrm{d}y)= \mathrm{d}x  \otimes \Lambda(\mathrm{d}y)$. By the definition of $\mathrm{PPP}(\omega_\Lambda)$, we have 
\[
\PP(\mathrm{PPP}(\omega_{\Lambda})\cap(\Delta^{b}_j(s)  \setminus \Delta_j^b(t))
= \varnothing)=\exp\Big(-(t-s)\int_{[b^{-j},b^{-(j-1)})}\Lambda(\mathrm{d}y)\Big)
\]
and by \eqref{cal-numdel=0-poi},
\[
\mathbb{P}(\mathrm{PPP}(\omega_{\Lambda})\cap\Delta^{b}_j(s) = \varnothing)=\exp\Big(-\int_{[b^{-j},b^{-(j-1)})}y\Lambda(\mathrm{d}y)\Big).
\]
Therefore, 
\begin{align}\label{cal3-2momentdiffer-poi}
	\mathbb{E}[\mathcal{X}_{b,j}(t)\mathcal{X}_{b,j}(s)]=\exp\Big(\int_{[b^{-j},b^{-(j-1)})}y\Lambda(\mathrm{d}y)\Big)\cdot\exp\Big(-(t-s)\int_{[b^{-j},b^{-(j-1)})}\Lambda(\mathrm{d}y)\Big).
\end{align}

Now by \eqref{cal1-2momentdiffer-poi}, \eqref{cal2-2momentdiffer-poi} and \eqref{cal3-2momentdiffer-poi}, we have
\[
\mathbb{E}[|\mathcal{X}_{b,j}(t)-\mathcal{X}_{b,j}(s)|^2]=2\exp\Big(\int_{[b^{-j},b^{-(j-1)})}y\Lambda(\mathrm{d}y)\Big)\cdot\Big[1-\exp\Big(-(t-s)\int_{[b^{-j},b^{-(j-1)})}\Lambda(\mathrm{d}y)\Big)\Big].
\]
By the condition $\chi(b,\Lambda)<1$ and the elementary inequality
\[
\int_{[b^{-j},b^{-(j-1)})}\Lambda(\mathrm{d}y)\leq b^{j}\int_{[b^{-j},b^{-(j-1)})}y\Lambda(\mathrm{d}y),
\]
there exist constants $T,T'>0$ such that
\[
\mathbb{E}[|\mathcal{X}_{b,j}(t)-\mathcal{X}_{b,j}(s)|^2]\leq T\cdot[e^{T\cdot b^j(t-s)}-1]\leq T'\cdot b^j(t-s).
\]
This completes the proof  of Lemma~\ref{Holder-p-moment-X_PCm-poi}. 
\end{proof}

\begin{proof}[Proof of Proposition~\ref{prop-manrancov-ass}]
Proposition~\ref{prop-manrancov-ass} follows immediately from Lemmas~\ref{elem-prop-poi}, \ref{p-moment-X_PCm-poi} and \ref{Holder-p-moment-X_PCm-poi}. 
\end{proof}

\subsection{Proofs of Theorem~\ref{thm-manrancov-poi} and Theorem~\ref{salemcanonicalcase-poi}}
\begin{proof}[Proof of Theorem~\ref{thm-manrancov-poi}]
For any integer $b\geq2$ satisfying $\chi(b,\Lambda)<1$, by Proposition~\ref{prop-manrancov-ass}, we can apply Theorem~\ref{Fou-Dec-abs} to  $\mathrm{MRC}_\Lambda$ under the constructions  \eqref{def-muPCb-m-poi} and \eqref{weacon-muPCm-poi}. By \eqref{def-Theta-abs} and Lemma~\ref{p-moment-X_PCm-poi}, we have
\[
\Theta_b(p):=(p-1)\log b-\log\Big(\limsup_{j\to\infty}\sup_{t\in[0,1]}\mathbb{E}[\mathcal{X}_{b,j}^{p}(t)]\Big)=(p-1)[1-\chi(b,\Lambda)]\log b
\]
and hence 
\[
\sup_{1<p\leq2}\frac{\Theta_b(p)}{p\log b}=\frac{\Theta_b(2)}{2\log b}=\frac{1}{2}[1-\chi(b, \Lambda)].
\]
Hence by Proposition~\ref{prop-manrancov-ass}  and Theorem~\ref{Fou-Dec-abs}, we obtain that almost surely,
\[
\dim_{F}(\mathrm{MRC}_{\Lambda})\geq\min\Big\{2 \alpha_0,\sup_{1<p\leq 2}\frac{2  \Theta_b(p)}{p\log b}\Big\}=\min\Big\{1,1-\chi(b, \Lambda)\Big\}=1-\chi(b, \Lambda).
\]
Consequently, by the definition \eqref{def-dimF-set} of the Fourier dimension of a set,  almost surely,  we have  
\[
\dim_F(E_\Lambda) \ge \dim_{F}(\mathrm{MRC}_{\Lambda}) \ge 1-\chi(b, \Lambda).
\]

By the definitions \eqref{def-chi-b}  of $\chi(b, \Lambda)$ and $\chi(\Lambda)$, recalling the condition of Theorem~\ref{thm-manrancov-poi} that
\[
\chi(\Lambda)=\inf_{b\in \N_{\ge 2}}\chi(b, \Lambda)<1,
\]
there exists a sequence $\{b_n\}\subset\N_{\geq2}$ such that
\[
\lim_{n\to\infty}\chi(b_n, \Lambda)=\chi(\Lambda)\anand\chi(b_n,\Lambda)<1\text{ for all $b_n$}.
\]
Therefore, almost surely, we have
$
\dim_F(E_\Lambda) \ge \dim_{F}(\mathrm{MRC}_{\Lambda})\geq1-\chi(b_n, \Lambda)
$
and hence
\[
\dim_F(E_\Lambda) \ge \dim_{F}(\mathrm{MRC}_{\Lambda})\geq1-\chi(\Lambda).
\]
This completes the whole proof of Theorem~\ref{thm-manrancov-poi}.
\end{proof}

\begin{proof}[Proof of Theorem~\ref{salemcanonicalcase-poi}]
By the definition  \eqref{canonicalcase-poi} of the measure $\Lambda_\alpha$, for any integer $b\ge 2$, we have 
\[
\int_{[b^{-j}, b^{-(j-1)})}y\Lambda_\alpha(\mathrm{d}y) =\sum_{b^{-j}\leq\frac{\alpha}{n}<b^{-(j-1)}}\frac{\alpha}{n}=\alpha\sum_{\alpha b^{j-1}<n\leq\alpha b^j}\frac{1}{n},
\]
which implies
\[
\limsup_{j\to\infty}\int_{[b^{-j}, b^{-(j-1)})}y\Lambda_\alpha(\mathrm{d}y) =\alpha\log b.
\]
Then by \eqref{def-chi-b},
we have  $\chi(b, \Lambda_\alpha)= \alpha$ and hence  $\chi(\Lambda_\alpha)= \alpha\in(0,1)$. It follows from Theorem~\ref{thm-manrancov-poi} that almost surely,
\[
\dim_{F}(E_{\Lambda_\alpha})\geq\dim_{F}(\mathrm{MRC}_{\Lambda_\alpha})\geq1-\alpha.
\]

On the other hand, note that the Hausdorff dimension of the uncovered set in $[0,1]$ is smaller than the Hausdorff dimension of the uncovered set in $[0,\infty)$ and thus by  a result of Fitzsimmons, Fristedt and Shepp \cite[Corollary~3]{FFS85},  almost surely,  we have 
\[
\dim_{H}(E_{\Lambda_\alpha})\leq1-\alpha.
\]
Consequently, by  the definition \eqref{def-Hdim-m} of Hausdorff dimension of a measure,   almost surely, we have 
\[
\dim_{H}(\mathrm{MRC}_{\Lambda_\alpha})\leq\dim_{H}(E_{\Lambda_\alpha})\leq1-\alpha.
\]

Finally,  using  the following natural inequalities (see, e.g., \cite{BSS23}):
\[
	\dim_{F}(\mathrm{MRC}_{\Lambda_\alpha})\leq\dim_{H}(\mathrm{MRC}_{\Lambda_\alpha})\anand\dim_{F}(E_{\Lambda_\alpha})\leq\dim_{H}(E_{\Lambda_\alpha}),
\]
we complete the proof of Theorem~\ref{salemcanonicalcase-poi}. 
\end{proof}

\section{Fourier decay of Poisson multiplicative chaos}\label{S-pmc-abs}
This section is devoted to the derivation of Theorem~\ref{thm-poimulcha-pmc} from our unified Theorem~\ref{Fou-Dec-abs}. We are going to show that  the  PMC measure  is a multiplicative chaos measure associated to a sequence of indepedent stochastic processes satisfying Assumptions~\ref{assum-indep-abs}, \ref{assum-Lp0-abs} and \ref{assum-alp0-abs}.

\subsection{Multiplicative structure of PMC measures}
Fix  $a\in(0,1)$. 
Recall the definitions of $\mathscr{D}_{\epsilon}$ in \eqref{def-Depsil-pmc} and $\mathrm{PMC}_{\Lambda}^{a,\epsilon}$ in \eqref{def-mu-a-pmc}. Fix any integer $b\geq2$. For any integer $m\geq1$ and any $t\in[0,1]$, denote
\[
D^{b}_{m}(t):=\Big\{(x,y)\in\mathbb{R}\times(0,1)\,\Big|\,b^{-m}\leq y<1\an t-y<x<t\Big\}.
\]
Recall that $\omega_\Lambda(\mathrm{d}x\mathrm{d}y)= \mathrm{d}x  \otimes \Lambda(\mathrm{d}y)$. Then by the elementary property of Poisson point process, 
\[
\E[a^{\#(\mathrm{PPP} (\omega_\Lambda)\cap D^{b}_{m}(t))}] = \exp\Big(-(1-a)\omega_\Lambda(D_m^b(t))\Big).
\]
Hence, by \eqref{def-mu-a-pmc}, we can write
\begin{align}\label{muam-num-pmc}
	\mathrm{PMC}_{\Lambda}^{a,b^{-m}}(\mathrm{d}t)=a^{\#(\mathrm{PPP} (\omega_\Lambda)\cap D^{b}_{m}(t))} \exp\Big((1-a) \omega_\Lambda(D_m^b(t)) \Big)\mathrm{d}t.
\end{align}

For any $t\in[0,1]$ and any integer $j\geq1$, consider the set
\begin{align}\label{def-Delm-pmc}
	\Delta^{b}_j(t):=\Big\{(x,y)\in\mathbb{R}\times(0,1)\,\Big|\,b^{-j}\leq y< b^{-(j-1)}\an t-y<x<t\Big\}
\end{align}
and define 
\begin{align}\label{muPCm-numDelta-pmc}
	\mathcal{X}_{a,b,0}(t):\equiv1\anand\mathcal{X}_{a,b,j}(t):=a^{\#(\mathrm{PPP} (\omega_\Lambda)\cap\Delta^{b}_j(t))}\exp\Big((1-a)\omega_\Lambda(\Delta^{b}_j(t))\Big)\text{ for each $j\geq1$}.
\end{align}
Hence  we obtain a sequence of independent stochastic processes:
\[
\{\mathcal{X}_{a,b,j}(t): t\in [0,1]\}_{j\ge 0}  \quad \text{with $\mathcal{X}_{a,b,j}(t)\geq0$ and $\E[\mathcal{X}_{a,b,j}(t)]\equiv 1$}.
\]
Note that for any $t\in[0,1]$,
\[
D^{b}_{m}(t)=\bigsqcup_{j=0}^{m}\Delta^{b}_j(t)
\anand
 \#(\mathrm{PPP}(\omega_{\Lambda})\cap D^{b}_{m}(t)) =  \sum_{j=0}^m  \#(\mathrm{PPP} (\omega_\Lambda)\cap\Delta^{b}_j(t)). 
\]
Therefore, by \eqref{muam-num-pmc} and \eqref{muPCm-numDelta-pmc}, we have
\begin{align}\label{def-muPCb-m-pmc}
	\mathrm{PMC}_{\Lambda}^{a,b^{-m}}(\mathrm{d}t)=\Big[\prod_{j=0}^{m}\mathcal{X}_{a,b,j}(t)\Big] \mathrm{d}t.
\end{align}
Hence, by comparing \eqref{def-muPCb-m-pmc} and  \eqref{weacon-mu-a-pmc}, we have the following limit in the sense of weak convergence:
\begin{align}\label{weacon-mu-am-pmc}
	\lim_{m\to\infty}\mathrm{PMC}_{\Lambda}^{a,b^{-m}}=\mathrm{PMC}_{\Lambda}^{a}.
\end{align}

For the fixed integer $b\geq2$, recall the definition of $\chi(b,\Lambda)$ in \eqref{def-chi-b}.

\begin{proposition}\label{prop-manrancov-ass-pmc}
	Suppose that $\chi(b,\Lambda)<1$, then the independent stochastic processes defined as \eqref{muPCm-numDelta-pmc}: 
	\[
	\{\mathcal{X}_{a,b,j}(t): t\in [0,1]\}_{j\ge 0}
	\]
	satisfies Assumptions~\ref{assum-indep-abs}, \ref{assum-Lp0-abs} and \ref{assum-alp0-abs} for $\alpha_0=1/2$ and $p_0=2$.
\end{proposition}

\subsection{Proof of Proposition~\ref{prop-manrancov-ass-pmc}}
We shall  verify Assumptions~\ref{assum-indep-abs}, \ref{assum-Lp0-abs} and \ref{assum-alp0-abs} for $\alpha_0=1/2$ and $p_0=2$  by the following Lemmas~\ref{elem-prop-pmc}, \ref{p-moment-X_PCm-pmc} and \ref{Holder-p-moment-X_PCm-pmc}. 

\begin{lemma}\label{elem-prop-pmc}
	The stochastic processes $\mathcal{X}_{a,b,j}$ satisfy the following properties:
	\begin{itemize}
		\item[(P1)] The stochastic processes $\{\mathcal{X}_{a,b,j}\}_{j\geq1}$ are  independent$;$  
		\item[(P2)] For any sub-intervals $T,S\subset[0,1]$ satisfying
		\[
		\mathrm{dist}(T,S) = \inf\{|t_1-t_2|: t_1\in T\,\,\mathrm{and\,\,}t_2 \in S\} \geq  b^{-(j-1)},
		\]
		the stochastic processes $\{\mathcal{X}_{a,b,j}(t):t\in T\}$ and $\{\mathcal{X}_{a,b,j}(t): t\in S\}$ are independent.
	\end{itemize}
\end{lemma}

\begin{lemma}\label{p-moment-X_PCm-pmc}
	Suppose that $\chi(b,\Lambda)<1$, then  for any $p>1$, 
	\[
	\limsup_{j\to\infty}\sup_{t\in[0,1]}\mathbb{E}[\mathcal{X}_{a,b,j}^{p}(t)]=\exp\big((a^p-ap+p-1)\chi(b,\Lambda)\log b\big)<b^{p-1}.
	\]
\end{lemma}

\begin{lemma}\label{Holder-p-moment-X_PCm-pmc}
	Suppose that $\chi(b,\Lambda)<1$, then
	\[
	\sup_{j\in\N}\sup_{0<|t-s|\leq b^{-j} }\mathbb{E}\Big[\Big| \frac{\mathcal{X}_{a,b,j}(t)-\mathcal{X}_{a,b,j}(s)}{\sqrt{b^{j}  |t-s|}} \Big|^2\Big]<\infty.
	\]
\end{lemma}

\begin{proof}[Proof of Lemma~\ref{elem-prop-pmc}]
The proof of Lemma~\ref{elem-prop-pmc} is the same as  that of  Lemma~\ref{elem-prop-poi}.
\end{proof}

\begin{proof}[Proof of Lemma~\ref{p-moment-X_PCm-pmc}]
Fix any $t\in[0,1]$ and $p>1$. By the definitions \eqref{def-Delm-pmc} and \eqref{muPCm-numDelta-pmc} for $\Delta^{b}_j(t)$ and $\mathcal{X}_{a,b,j}(t)$ respectively, for any $j\ge 1$,  we have
\begin{align}\label{p-2-use}
\begin{split}
\mathbb{E}[\mathcal{X}_{a,b,j}^{p}(t)]&=\mathbb{E}\big[a^{p \#(\mathrm{PPP} (\omega_\Lambda)\cap\Delta^{b}_j(t))}\big]  \cdot \exp\big((1-a)p \omega_\Lambda(\Delta^{b}_j(t))\big)
\\
& = \exp\big((a^p-1)\omega_\Lambda(\Delta^{b}_j(t)) \big)\cdot \exp\big((1-a)p \omega_\Lambda(\Delta^{b}_j(t))\big)
\\
& = \exp\big( (a^p - ap +p-1) \omega_\Lambda(\Delta_j^b(t))\big).
\end{split} 
\end{align}
Since $\omega_\Lambda(\mathrm{d}x\mathrm{d}y)= \mathrm{d}x  \otimes \Lambda(\mathrm{d}y)$, we obtain 
\begin{align}\label{ind-t-omega}
\omega_\Lambda(\Delta^{b}_j(t))=\int_{[b^{-j},b^{-(j-1)})}y\Lambda(\mathrm{d}y).
\end{align}
and hence 
\[
		\mathbb{E}[\mathcal{X}_{a,b,j}^{p}(t)] =\exp\Big((a^p-ap+p-1)\int_{[b^{-j},b^{-(j-1)})}y\Lambda(\mathrm{d}y)\Big).
\]
Since $a^p-ap=a(a^{p-1}-p)<0$,  by the condition $\chi(b,\Lambda)<1$,
\[
	\limsup_{j\to\infty}\sup_{t\in[0,1]}\mathbb{E}[\mathcal{X}_{a,b,j}^{p}(t)]=\exp\big((a^p-ap+p-1)\chi(b,\Lambda)\log b\big)<b^{p-1}.
\]
This completes the proof of Lemma~\ref{p-moment-X_PCm-pmc}.
\end{proof}

\begin{proof}[Proof of Lemma~\ref{Holder-p-moment-X_PCm-pmc}]
	We only need to consider $j\geq1$ since $\mathcal{X}_{a,b,0}(t)\equiv1$. Fix $j\geq1$ and $t,s\in[0,1]$ satisfying $|t-s|\leq b^{-j}$, without loss of generality, we may assume that $t>s$. Write 
	\begin{align}\label{cal1-2momentdiffer-pmc}
		\mathbb{E}[|\mathcal{X}_{a,b,j}(t)-\mathcal{X}_{a,b,j}(s)|^2]=\mathbb{E}[\mathcal{X}^2_{a,b,j}(t)]+\mathbb{E}[\mathcal{X}^2_{a,b,j}(s)]-2\mathbb{E}[\mathcal{X}_{a,b,j}(t)\mathcal{X}_{a,b,j}(s)].
	\end{align}
	Take $p=2$ in  \eqref{p-2-use} and use the independence of $t$ in \eqref{ind-t-omega},  we get 
	\begin{align}\label{cal2-2momentdiffer-pmc}
		\mathbb{E}[\mathcal{X}^2_{a,b,j}(t)]=\mathbb{E}[\mathcal{X}^2_{a,b,j}(s)]=\exp\Big((1-a)^2\omega_\Lambda(\Delta^b_j(t))\Big).
	\end{align}
	
	By the definition of $\mathcal{X}_{a,b,j}$ in \eqref{muPCm-numDelta-pmc}, we have
	\begin{align}\label{cal1-num-pmc}
	\mathbb{E}[\mathcal{X}_{a,b,j}(t)\mathcal{X}_{a,b,j}(s)]=\mathbb{E}\big[a^{  \#(\mathrm{PPP} (\omega_\Lambda)\cap\Delta^{b}_j(t)) + \#(\mathrm{PPP} (\omega_\Lambda)\cap\Delta^{b}_j(s))}\big]  \exp\big((1-a)\big[\omega_\Lambda(\Delta^{b}_{j}(t))+\omega_\Lambda(\Delta^{b}_{j}(s))\big]\big).
	\end{align}
	Define three independent Poisson variables as
	\begin{align*}
	\#_{1}:=&\#\big(\mathrm{PPP} (\omega_\Lambda)\cap(\Delta^{b}_j(t)\setminus\Delta^{b}_j(s))\big),
	\\
	\#_{2}:=&\#\big(\mathrm{PPP} (\omega_\Lambda)\cap(\Delta^{b}_j(s)\setminus\Delta^{b}_j(t))\big),
	\\
	\#_{3}:=&\#\big(\mathrm{PPP} (\omega_\Lambda)\cap(\Delta^{b}_j(t)\cap\Delta^{b}_j(s))\big).
	\end{align*}
	Then we have
	\[
	\mathbb{E}\big[a^{  \#(\mathrm{PPP} (\omega_\Lambda)\cap\Delta^{b}_j(t)) + \#(\mathrm{PPP} (\omega_\Lambda)\cap\Delta^{b}_j(s))}\big]=\mathbb{E}\big[a^{\#_1+\#_1+2\#_3}\big]=\mathbb{E}\big[a^{\#_1}\big]\cdot\mathbb{E}\big[a^{\#_2}\big]\cdot\mathbb{E}\big[a^{2\#_3}\big].
	\]
	By the elementary equalities
	\begin{align*}
	\mathbb{E}\big[a^{\#_1}\big]&=\exp\big((a-1)\omega_\Lambda(\Delta^{b}_j(t)\setminus\Delta^{b}_j(s))\big),
	\\
	\mathbb{E}\big[a^{\#_2}\big]&=\exp\big((a-1)\omega_\Lambda(\Delta^{b}_j(s)\setminus\Delta^{b}_j(t))\big),
\\
	\mathbb{E}\big[a^{2\#_3}\big]&=\exp\big((a^2-1)\omega_\Lambda(\Delta^{b}_j(t)\cap\Delta^{b}_j(s))\big),
	\end{align*}
	and the elementary computation
	\begin{align*}
&  	(a-1)\omega_\Lambda(\Delta^{b}_j(t)\setminus\Delta^{b}_j(s)) +(a-1)\omega_\Lambda(\Delta^{b}_j(s)\setminus\Delta^{b}_j(t))  +  (a^2-1)\omega_\Lambda(\Delta^{b}_j(t)\cap\Delta^{b}_j(s))
\\
&=  (a-1)\big[\omega_\Lambda(\Delta^{b}_{j}(t))+\omega_\Lambda(\Delta^{b}_{j}(s))\big] + (1-a)^2\omega_\Lambda(\Delta^{b}_j(t)\cap\Delta^{b}_j(s)), 
\end{align*}
we obtain 
	\begin{align}\label{cal3-num-pmc}
		\begin{split}
			&\quad\,\ \mathbb{E}\big[a^{  \#(\mathrm{PPP} (\omega_\Lambda)\cap\Delta^{b}_j(t)) + \#(\mathrm{PPP} (\omega_\Lambda)\cap\Delta^{b}_j(s))}\big]
			\\
			&=\exp\Big((a-1)\big[\omega_\Lambda(\Delta^{b}_{j}(t))+\omega_\Lambda(\Delta^{b}_{j}(s))\big]\Big)\cdot\exp\Big((1-a)^2\omega_\Lambda(\Delta^{b}_j(t)\cap\Delta^{b}_j(s))\Big).
		\end{split}
	\end{align}
	Combining  \eqref{cal1-num-pmc} and \eqref{cal3-num-pmc}, we get
	\begin{align}\label{cal3-2momentdiffer-pmc}
		\mathbb{E}[\mathcal{X}_{a,b,j}(t)\mathcal{X}_{a,b,j}(s)]=\exp\Big((1-a)^2\omega_\Lambda(\Delta^{b}_j(t)\cap\Delta^{b}_j(s))\Big).
	\end{align}
	
	Now by \eqref{cal1-2momentdiffer-pmc}, \eqref{cal2-2momentdiffer-pmc} and \eqref{cal3-2momentdiffer-pmc}, we have
	\begin{align*}
		&\quad\,\,\mathbb{E}[|\mathcal{X}_{a,b,j}(t)-\mathcal{X}_{a,b,j}(s)|^2]\\
		&=2\exp\Big((1-a)^2\omega_\Lambda(\Delta^b_j(t))\Big)-2\exp\Big((1-a)^2\omega_\Lambda(\Delta^{b}_j(t)\cap\Delta^{b}_j(s))\Big)\\
		&=2\exp\Big((1-a)^2\omega_\Lambda(\Delta^{b}_j(t)\cap\Delta^{b}_j(s))\Big)\cdot\Big[\exp\Big((1-a)^2\omega_\Lambda(\Delta^b_j(t)\setminus\Delta^b_j(s))\Big)-1\Big].
	\end{align*}
Recall that $\omega_\Lambda(\mathrm{d}x\mathrm{d}y)= \mathrm{d}x  \otimes \Lambda(\mathrm{d}y)$. 	By the condition $\chi(b,\Lambda)<1$ and the elementary inequalities 
	\[
	\omega_\Lambda(\Delta^{b}_j(t)\cap\Delta^{b}_j(s))\leq\omega_\Lambda(\Delta^{b}_j(t))=\int_{[b^{-j},b^{-(j-1)})}y\Lambda(\mathrm{d}y),
	\]
	\[
	\omega_\Lambda(\Delta^b_j(t)\setminus\Delta^b_j(s))=(t-s)\int_{[b^{-j},b^{-(j-1)})}\Lambda(\mathrm{d}y)\leq(t-s)b^{j}\int_{[b^{-j},b^{-(j-1)})}y\Lambda(\mathrm{d}y),
	\]
	there exist constants $H,H'>0$ such that
	\[
	\mathbb{E}[|\mathcal{X}_{a,b,j}(t)-\mathcal{X}_{a,b,j}(s)|^2]\leq H\cdot[e^{H\cdot b^j(t-s)}-1]\leq H'\cdot b^j(t-s).
	\]
	This completes the proof  of Lemma~\ref{Holder-p-moment-X_PCm-pmc}.
\end{proof}

\begin{proof}[Proof of Proposition~\ref{prop-manrancov-ass-pmc}]
	Proposition~\ref{prop-manrancov-ass-pmc} follows immediately from Lemmas~\ref{elem-prop-pmc}, \ref{p-moment-X_PCm-pmc} and \ref{Holder-p-moment-X_PCm-pmc}.
\end{proof}

\subsection{Proof of Theorem~\ref{thm-poimulcha-pmc}}
	For any integer $b\geq2$ satisfying $\chi(b,\Lambda)<1$, by Proposition~\ref{prop-manrancov-ass-pmc}, we can apply Theorem~\ref{Fou-Dec-abs} to  $\mathrm{PMC}_\Lambda^a$ under the constructions  \eqref{def-muPCb-m-pmc} and \eqref{weacon-mu-am-pmc}. By \eqref{def-Theta-abs} and Lemma~\ref{p-moment-X_PCm-pmc}, we have
	\begin{align*}
		\Theta_{a,b}(p):&=(p-1)\log b-\log\Big(\limsup_{j\to\infty}\sup_{t\in[0,1]}\mathbb{E}[\mathcal{X}_{a,b,j}^{p}(t)]\Big)\\
		&=[(p-1)-(a^p-ap+p-1)\chi(b,\Lambda)]\log b
	\end{align*}
	and then one can show that
	\[
	\sup_{1<p\leq2}\frac{\Theta_{a,b}(p)}{p\log b}=\frac{\Theta_{a,b}(2)}{2\log b}=\frac{1}{2}[1-(1-a)^2\chi(b, \Lambda)].
	\]
	Hence by Proposition~\ref{prop-manrancov-ass-pmc}  and Theorem~\ref{Fou-Dec-abs}, we obtain that almost surely,
	\[
	 \dim_{F}(\mathrm{PMC}_{\Lambda}^a)\geq\min\Big\{2 \alpha_0,\sup_{1<p\leq 2}\frac{2  \Theta_{a,b}(p)}{p\log b}\Big\}=\min\Big\{1,1-(1-a)^2\chi(b, \Lambda)\Big\}=1-(1-a)^2\chi(b, \Lambda).
	\]
	
	By the definitions \eqref{def-chi-b} of $\chi(b, \Lambda)$ and $\chi(\Lambda)$, recalling the condition of Theorem~\ref{thm-poimulcha-pmc} that
	\[
	\chi(\Lambda)=\inf_{b\in \N_{\ge 2}}\chi(b, \Lambda)<1,
	\]
 there exists a sequence $\{b_n\}\subset\N_{\geq2}$ such that
	\[
	\lim_{n\to\infty}\chi(b_n, \Lambda)=\chi(\Lambda)\anand\chi(b_n,\Lambda)<1\text{ for all $b_n$}.
	\]
	Therefore, almost surely, we have
	$
	 \dim_{F}(\mathrm{PMC}_{\Lambda}^a)\geq1-(1-a)^2\chi(b_n, \Lambda)
	$
	and hence
	\[
	 \dim_{F}(\mathrm{PMC}_{\Lambda}^a)\geq1-(1-a)^2\chi(\Lambda).
	\]
	This completes the whole proof of Theorem~\ref{thm-poimulcha-pmc}.

\section{Fourier decay of Generalized Mandelbrot cascades}\label{S-genercas-abs}
This section is devoted to the derivation of Theorem~\ref{thm-selsim-cas} from our unified  Theorem~\ref{Fou-Dec-abs}, as well as the proofs of Corollary~\ref{MC-FouDec-cas} and Corollary~\ref{GBM-FouDec-cas} from Theorem~\ref{thm-selsim-cas}.

\subsection{Proof of Theorem~\ref{thm-selsim-cas}}
We shall show that the sequence of independent stochastic processes \eqref{def-Xm-cas} in the construction \eqref{def-geneMC-cas} of the generalized Mandelbrot cascades satisfies Assumptions~\ref{assum-indep-abs}, \ref{assum-Lp0-abs} and \ref{assum-alp0-abs}, and then Theorem~\ref{thm-selsim-cas} follows from Proposition~\ref{non-degene-abs} and Theorem~\ref{Fou-Dec-abs} immediately.

First of all,  by the construction of $\mathrm{MC}_\mathscr{W}^b$ in \S\ref{S-cas-abs},  Assumption~\ref{assum-indep-abs} is satisfied automatically.

We now turn to verify  Assumption~\ref{assum-Lp0-abs} and Assumption~\ref{assum-alp0-abs}.   Recall the definition of the random function $\mathscr{W}(\mathbf{t})$ in \eqref{def-Wt-cas} and the independent random functions $\{\mathscr{W}_{\mathbf{I}}(\mathbf{t}):\mathbf{t}\in\mathbf{I}\}$ indexed by $\mathbf{I}\in\bigsqcup\limits_{m=1}^{\infty}\mathscr{D}_{m}^{b}$ in \eqref{def-WIt-cas}.  By the assumption  \eqref{condi-thm-selsim-cas} of Theorem~\ref{thm-selsim-cas}: there exist $p_0\in (1,2]$ and $\alpha_0\in(0,1]$ such that
\[
\sup_{\mathbf{t}\in[0,1)^d}\E[\mathscr{W}(\mathbf{t})^{p_0}]<b^{d(p_0-1)}\anand\sup_{\mathbf{t},\mathbf{s}\in[0,1)^d,\mathbf{t}\neq\mathbf{s}}\mathbb{E}\Big[\Big|\frac{\mathscr{W}(\mathbf{t})-\mathscr{W}(\mathbf{s})}{|\mathbf{t}-\mathbf{s}|^{\alpha_0}} \Big|^{p_0}\Big]<\infty.
\] 
For any integer $j>0$, we have the decomposition
\[
[0,1)^d=\bigsqcup_{\mathbf{I}\in\mathscr{D}_j^b}\mathbf{I},
\]
then by the expression \eqref{def-XWm-cas}, we can write
\[
\sup_{\mathbf{t}\in[0,1)^d}\mathbb{E}[\XX_{\mathscr{W},j}^{p_0}(\mathbf{t})]=\sup_{\mathbf{I}\in\mathscr{D}_j^b}\sup_{\mathbf{t}\in\mathbf{I}}\mathbb{E}\Big[\Big(\sum_{\mathbf{I}\in\mathscr{D}_{j}^b}\mathscr{W}_{\mathbf{I}}(\mathbf{t})\mathds{1}_{\mathbf{I}}(\mathbf{t})\Big)^{p_0}\Big]=\sup_{\mathbf{I}\in\mathscr{D}_j^b}\sup_{\mathbf{t}\in\mathbf{I}}\mathbb{E}[\mathscr{W}_{\mathbf{I}}^{p_0}(\mathbf{t})].
\]
By the self-similar property \eqref{selsimproper-cas} and the definition of $\ell_{\mathbf{I}}$ in \eqref{ell-I-cas},
we know
\[
\sup_{\mathbf{t}\in\mathbf{I}}\mathbb{E}[\mathscr{W}_{\mathbf{I}}^{p_0}(\mathbf{t})]=\sup_{\mathbf{t}\in\mathbf{I}}\mathbb{E}\big[\mathscr{W}^{p_0}\big(b^j(\mathbf{t}-\ell_{\mathbf{I}})\big)\big]=\sup_{\mathbf{t}\in[0,1)^d}\E[\mathscr{W}^{p_0}(\mathbf{t})]
\]
and hence for any integer $j>0$,
\[
\sup_{\mathbf{t}\in[0,1)^d}\mathbb{E}[\XX_{\mathscr{W},j}^{p_0}(\mathbf{t})]=\sup_{\mathbf{I}\in\mathscr{D}_j^b}\sup_{\mathbf{t}\in[0,1)^d}\E[\mathscr{W}^{p_0}(\mathbf{t})]=\sup_{\mathbf{t}\in[0,1)^d}\E[\mathscr{W}^{p_0}(\mathbf{t})].
\]
Therefore, we obtain
\begin{align}\label{limsupEPwjp-cas}
	\limsup_{j\to\infty}\sup_{\mathbf{t}\in[0,1)^d}\mathbb{E}[\XX_{\mathscr{W},j}^{p_0}(\mathbf{t})]=\sup_{\mathbf{t}\in[0,1)^d}\E[\mathscr{W}^{p_0}(\mathbf{t})]<b^{d(p_0-1)}.
\end{align}
Hence,   Assumption~\ref{assum-Lp0-abs} is satisfied.

Similarly, we have
\[
\sup_{j>0}\sup_{\mathbf{I}\in\mathscr{D}_j^b}\sup_{\mathbf{t},\mathbf{s}\in\mathbf{I} \atop\mathbf{t}\neq\mathbf{s}}\mathbb{E}\Big[\Big| \frac{\XX_{\mathscr{W},j}(\mathbf{t})-\XX_{\mathscr{W},j}(\mathbf{s})}{b^{j\alpha_0}|\mathbf{t}-\mathbf{s}|^{\alpha_0}} \Big|^{p_0}\Big]=\sup_{j>0}\sup_{\mathbf{I}\in\mathscr{D}_j^b}\sup_{\mathbf{t},\mathbf{s}\in\mathbf{I} \atop\mathbf{t}\neq\mathbf{s}}\mathbb{E}\Big[\Big| \frac{\mathscr{W}_{\mathbf{I}}(\mathbf{t})-\mathscr{W}_{\mathbf{I}}(\mathbf{s})}{b^{j\alpha_0}|\mathbf{t}-\mathbf{s}|^{\alpha_0}}\Big|^{p_0}\Big]
\]
and then
\begin{align*}
	&\sup_{j>0}\sup_{\mathbf{I}\in\mathscr{D}_j^b}\sup_{\mathbf{t},\mathbf{s}\in\mathbf{I} \atop\mathbf{t}\neq\mathbf{s}}\mathbb{E}\Big[\Big| \frac{\XX_{\mathscr{W},j}(\mathbf{t})-\XX_{\mathscr{W},j}(\mathbf{s})}{b^{j\alpha_0}|\mathbf{t}-\mathbf{s}|^{\alpha_0}} \Big|^{p_0}\Big]=\sup_{j>0}\sup_{\mathbf{I}\in\mathscr{D}_j^b}\sup_{\mathbf{t},\mathbf{s}\in\mathbf{I} \atop\mathbf{t}\neq\mathbf{s}}\mathbb{E}\Big[\Big| \frac{\mathscr{W}\big(b^j(\mathbf{t}-\ell_{\mathbf{I}})\big)-\mathscr{W}\big(b^j(\mathbf{s}-\ell_{\mathbf{I}})\big)}{b^{j\alpha_0}|\mathbf{t}-\mathbf{s}|^{\alpha_0}}\Big|^{p_0}\Big]\\
	&=\sup_{j>0}\sup_{\mathbf{I}\in\mathscr{D}_j^b}\sup_{\mathbf{t},\mathbf{s}\in[0,1)^d \atop\mathbf{t}\neq\mathbf{s}}\mathbb{E}\Big[\Big|\frac{\mathscr{W}(\mathbf{t})-\mathscr{W}(\mathbf{s})}{|\mathbf{t}-\mathbf{s}|^{\alpha_0}} \Big|^{p_0}\Big]=\sup_{\mathbf{t},\mathbf{s}\in[0,1)^d,\mathbf{t}\neq\mathbf{s}}\mathbb{E}\Big[\Big|\frac{\mathscr{W}(\mathbf{t})-\mathscr{W}(\mathbf{s})}{|\mathbf{t}-\mathbf{s}|^{\alpha_0}}\Big|^{p_0}\Big].
\end{align*}
Hence by noting that $\XX_{\mathscr{W},0}(\mathbf{t})\equiv1$, we get
\[
\sup_{j\in\N}\sup_{\mathbf{I}\in\mathscr{D}_j^b}\sup_{\mathbf{t},\mathbf{s}\in\mathbf{I} \atop\mathbf{t}\neq\mathbf{s}}\mathbb{E}\Big[\Big| \frac{\XX_{\mathscr{W},j}(\mathbf{t})-\XX_{\mathscr{W},j}(\mathbf{s})}{b^{j\alpha_0}|\mathbf{t}-\mathbf{s}|^{\alpha_0}} \Big|^{p_0}\Big]=\sup_{\mathbf{t},\mathbf{s}\in[0,1)^d,\mathbf{t}\neq\mathbf{s}}\mathbb{E}\Big[\Big|\frac{\mathscr{W}(\mathbf{t})-\mathscr{W}(\mathbf{s})}{|\mathbf{t}-\mathbf{s}|^{\alpha_0}}\Big|^{p_0}\Big]<\infty,
\]
that is, Assumption~\ref{assum-alp0-abs} is satisfied.  

Finally, by Proposition~\ref{non-degene-abs}, Theorem~\ref{Fou-Dec-abs} and \eqref{limsupEPwjp-cas}, the generalized Mandelbrot cascade measure $\mathrm{MC}_\mathscr{W}^b$ is non-degenerate and almost surely on $\{ \mathrm{MC}_\mathscr{W}^b \neq0\}$,
\[
\dim_{F}(\mathrm{MC}_\mathscr{W}^b)\geq \min\Big\{2 \alpha_0,\sup_{1<p\leq p_0 }  \Big[  2d \big(1- \frac{1}{p}) - 2\log_b \Big( \sup_{\mathbf{t}\in[0,1)^d}  \big(\mathbb{E}[\mathscr{W}^{p}(\mathbf{t})]\big)^{\frac{1}{p}}\Big)   \Big]\Big\}.
\]
This completes the proof of Theorem~\ref{thm-selsim-cas}.

\subsection{Proofs of Corollary~\ref{MC-FouDec-cas} and Corollary~\ref{GBM-FouDec-cas}}
\begin{proof}[Proof of Corollary~\ref{MC-FouDec-cas}]
Note that $\mathscr{W}(\mathbf{t})=\mathscr{W}(\mathbf{s})=W$ for any $\mathbf{t},\mathbf{s}\in[0,1]^d$, we have $\alpha_0=1$ for the second condition in \eqref{condi-thm-selsim-cas} of Theorem~\ref{thm-selsim-cas}. The first condition in \eqref{condi-thm-selsim-cas} of Theorem~\ref{thm-selsim-cas} follows from the assumptions that $\mathbb{E}[W\log W]<d\log b$ and $\E[W^{1+\varepsilon}]<\infty$ for some $\varepsilon>0$, which can be seen, e.g., \cite{KP76, CHQW24}. Then Corollary~\ref{MC-FouDec-cas} follows from Theorem~\ref{thm-selsim-cas} immediately.
\end{proof}

\begin{proof}[Proof of Corollary~\ref{GBM-FouDec-cas}]
Let $\mathscr{W}_\sigma$ be  given as in \eqref{def-GBM}. Then for any $p>1$, we have
\[
\E[\mathscr{W}_\sigma^{p}(t)]=\E\big[\exp\big(p\sigma\mathrm{B}(t)-\frac{p\sigma^2}{2}t\big)\big]=\exp\big(\frac{p(p-1) \sigma^2}{2}t\big)
\]
and hence
\begin{align}\label{sup-p-norm}
\sup_{t\in[0,1)}\E[\mathscr{W}_\sigma^{p}(t)]=\exp\big(\frac{p(p-1)\sigma^2}{2}\big).
\end{align}
Therefore, for sufficiently small $\varepsilon>0$ and  $p_0 = p_0(\varepsilon)=\min\{\frac{2\log b}{\sigma^2}-\varepsilon,2\}$,  the first condition in \eqref{condi-thm-selsim-cas} of Theorem~\ref{thm-selsim-cas} is satisfied:
\[
\sup_{t\in[0,1)}\E[\mathscr{W}_\sigma^{p_0}(t)]=\exp\big(\frac{p_0(p_0-1) \sigma^2}{2}\big)\leq\exp\Big((p_0-1)\log b-\frac{\varepsilon(p_0-1) \sigma^2}{2}\Big)<b^{p_0-1}.
\]

On the other hand,  for any $t,s\in[0,1]$, by a routine calculation for Brownian motion, we obtain 
\[
\E[|\mathscr{W}_\sigma(t)-\mathscr{W}_\sigma(s)|^2]=\big|e^{\sigma^2t}- e^{\sigma^2s}\big|.
\]
It follows that there exists a constant $C = C_\sigma>0$ such that for all $t, s\in [0,1]$,
\[
\E[|\mathscr{W}_\sigma(t)-\mathscr{W}_\sigma(s)|^2]\leq C|t-s| .
\]
Hence,  the second condition in \eqref{condi-thm-selsim-cas} of Theorem~\ref{thm-selsim-cas} holds for $\alpha_0=1/2$:
\[
\sup_{t,s\in[0,1),t\neq s}\mathbb{E}\Big[\Big| \frac{\mathscr{W}_\sigma(t)-\mathscr{W}_\sigma(s)}{\sqrt{|t-s|}}\Big|^{p_0}\Big]\leq\Big(\sup_{t,s\in[0,1),t\neq s}\mathbb{E}\Big[\Big| \frac{\mathscr{W}_\sigma(t)-\mathscr{W}_\sigma(s)}{\sqrt{|t-s|}}\Big|^2\Big]\Big)^{p_0/2}<\infty.
\]

Therefore, by Theorem~\ref{thm-selsim-cas} and \eqref{sup-p-norm}, we get that almost surely,
\[
\dim_{F}(\mathrm{MC}_{\mathscr{W}_\sigma}^b) \geq   \min\Big\{1,\sup_{1<p\leq p_0 }   \Big[  2\big(1- \frac{1}{p}) - \frac{(p-1) \sigma^2}{\log b}  \Big]\Big\}.
\]
Note that when $\varepsilon$ decreases to $0$, the exponent $p_0=p_0(\varepsilon)$ increases to $\min\{\frac{2\log b}{\sigma^2}, 2\}$, hence almost surely,
\[
\dim_{F}(\mathrm{MC}_{\mathscr{W}_\sigma}^b) \geq   \min\Big\{1,\sup_{1<p\leq \min\{\frac{2\log b}{\sigma^2}, 2\}}   \Big[  2 \big(1- \frac{1}{p}) - \frac{(p-1) \sigma^2}{\log b}  \Big]\Big\}.
\]
Recall the definition \eqref{def-D-sigma} of $D_\sigma$. A simple computation yields that
\[
\sup_{1<p\leq \min\{\frac{2\log b}{\sigma^2}, 2\}}   \Big[  2 \big(1- \frac{1}{p}) - \frac{(p-1) \sigma^2}{\log b}  \Big]=D_\sigma.
\]
This completes the proof of Corollary~\ref{GBM-FouDec-cas}.
\end{proof}

\section{The non-degeneracy of the multiplicative chaos}\label{S-Pf-Thm-Nondegen-abs}
This section is devoted to the proof of Proposition~\ref{non-degene-abs}.  

Throughout this section, we assume that Assumptions~\ref{assum-indep-abs} and \ref{assum-Lp0-abs} are satisfied by a given  sequence of independent stochastic processes with the natural conditions in \eqref{def-Xm-abs}: 
\[
\{\XX_m(\mathbf{t}): \mathbf{t}\in [0,1]^d\}_{m\geq0}  \quad \text{with $\XX_m(\mathbf{t})\geq0$ and $\E[\XX_m(\mathbf{t})]\equiv 1$}.
\]
For any $m\ge 0$, recall the definition \eqref{def-num-abs} of the random measure $\mu_m$:
\[
\mu_{m}(\mathrm{d}\mathbf{t})=\Big[\prod_{j=0}^{m}\XX_{j}(\mathbf{t}) \Big] \mathrm{d}\mathbf{t}.
\]
Let $(U_m)_{m\geq0}$ be  the martingale with respect to the natural filtration \eqref{def-filtra-abs} defined by 
\[
U_m := \mu_m([0,1]^d)=\int_{[0,1]^d}\Big[\prod_{j=0}^{m}\XX_{j}(\mathbf{t}) \Big] \mathrm{d}\mathbf{t}.
\]
Our goal is to prove that for any $1<p\leq p_0\le 2$ (here $p_0$ is given as in Assumption~\ref{assum-Lp0-abs}),
\begin{align}\label{propA-goal}
	\sup_{m\geq0}\mathbb{E}[U_m^p]<\infty.
\end{align}

\subsection{Elementary lemmas}
We start with the following elementary lemmas. For further reference, we include their simple proofs.

Recall the definitions of $\mathcal{L}_F$ and $\Theta(p)$ in \eqref{def-LF-abs} and \eqref{def-Theta-abs} respectively:
\[ 
\mathcal{L}_F=\min\Big\{2 \alpha_0,\sup_{1<p\leq p_0}\frac{2  \Theta(p)}{p\log b}\Big\}
\anand
\Theta(p)=d(p-1)\log b-\log\Big(\limsup_{j\to\infty}\sup_{\mathbf{t}\in[0,1)^d}\mathbb{E}[\XX_j^p(\mathbf{t})]\Big).
\]
And, as in \eqref{def-phi-abs},  we fix a large enough integer $j_0>0$ and denote
\begin{align}\label{def-phi-abs-again}
	\zeta(p)=\zeta_{j_0}(p):=d(p-1)\log b-\log\Big(\sup_{j>j_0}\sup_{\mathbf{t}\in[0,1)^d}\mathbb{E}[\XX_j^p(\mathbf{t})]\Big).
\end{align}

\begin{lemma}\label{lem-LF}
	Assumption~\ref{assum-Lp0-abs} implies $0<\mathcal{L}_F\leq d$.
\end{lemma}

\begin{proof}
	Fix any $j\ge 0$ and $\mathbf{t}\in [0,1)^d$. Since $\XX_j(\mathbf{t})\ge 0$ and $\E[\XX_j(\mathbf{t})]\equiv 1$, we may define a new probability measure with density $\XX_j(\mathbf{t})$:  
	\[
	\mathbb{Q}_{j, \mathbf{t}} [A]: = \E[\mathds{1}_A\cdot \XX_j(\mathbf{t})]
	\] and write the corresponding expectation by $\E_{j, \mathbf{t}}$.   Then 
	\[
	(\E[\XX_j^p(\mathbf{t})])^{\frac{1}{p-1}} = (\E[\XX_j^{p-1}(\mathbf{t}) \cdot \XX_j(\mathbf{t})])^{\frac{1}{p-1}}  =   (\E_{j, \mathbf{t}}[\XX_j^{p-1}(\mathbf{t})])^{\frac{1}{p-1}}. 
	\]
	Now since the map  $\alpha \mapsto 
	( \E_{j, \mathbf{t}}[X^\alpha])^{1/\alpha}$ is non-decreasing on $\alpha>0$,  for any $1<p\le p_0\le 2$, we have 
	\begin{align}\label{ine-hol-jen}
		(\E[\XX_j^p(\mathbf{t})])^{\frac{1}{p-1}}\le (\E[\XX_j^{p_0}(\mathbf{t})])^{\frac{1}{p_0-1}}.
	\end{align}
	Then by condition \eqref{def-Xm-abs} and Assumption~\ref{assum-Lp0-abs}, inequality \eqref{ine-hol-jen} implies that for any $1<p\leq p_0\leq2$, 
	\[
	1\le \limsup_{j\to\infty}\sup_{\mathbf{t}\in[0,1)^d}\mathbb{E}[\XX_j^{p}(\mathbf{t})]<b^{d(p-1)}
	\]
	and hence
	\[
	0<\sup_{1<p\leq p_0 }\frac{2\Theta(p)}{p\log b} \le d.
	\]
	Therefore, by noting that $\alpha_0\in(0,1]$,  we get $\mathcal{L}_F\in(0,d\,]$.
\end{proof}

 \begin{lemma}\label{lem-small-p}
 Assumption~\ref{assum-Lp0-abs} implies that for sufficiently large integer $j_0>0$, if $1<p\leq p_0\le 2$, then
\begin{align}\label{ass2-sma-p-abs}
	C_p = C_{j_0, p}: = \sup_{0\leq j\leq j_0}\sup_{\mathbf{t}\in[0,1)^d}\mathbb{E}[\XX_j^{p}(\mathbf{t})]<\infty\anand  S_p = S_{j_0, p}: = \sup_{j>j_0}\sup_{\mathbf{t}\in[0,1)^d}\mathbb{E}[\XX_j^{p}(\mathbf{t})]<b^{d(p-1)}.
\end{align}
 \end{lemma}
 
 \begin{proof}
 For $p=p_0$, the assertion \eqref{ass2-sma-p-abs} follows from Assumption~\ref{assum-Lp0-abs} directly.   
By the inequality \eqref{ine-hol-jen}, it is clear that Assumption~\ref{assum-Lp0-abs} implies the assertion \eqref{ass2-sma-p-abs} for any $p\in(1,  p_0]$. 
\end{proof}
 
\begin{lemma}\label{lem-finite-m}
Assumption~\ref{assum-Lp0-abs} implies tha for any fixed integer $m\ge 0$, if $1<p\le p_0\le 2$,  then $\E[U_m^p]<\infty$.
\end{lemma}
 
\begin{proof}
By the definition \eqref{def-num-abs} of $\mu_m$ and using the triangle inequality, for each fixed integer $m\ge 0$,
\[
(\E[U_m^p])^{1/p} = (\mathbb{E}[\mu^p_{m}([0,1]^d)])^{1/p}=\Big\{\mathbb{E}\Big[\Big|\int_{[0,1]^d}\Big[\prod_{j=0}^{m}\XX_{j}(\mathbf{t})\Big]\mathrm{d}\mathbf{t}\Big|^p\Big]\Big\}^{1/p}\leq\int_{[0,1]^d}\Big\|\prod_{j=0}^{m}\XX_{j}(\mathbf{t})\Big\|_{L^p(\mathbb{P})}\mathrm{d}\mathbf{t}.
\]
By the independence of the stochastic processes $\{\XX_j\}_{0\le j \le m}$, 
\[
\Big\|\prod_{j=0}^{m}\XX_{j}(\mathbf{t})\Big\|_{L^p(\mathbb{P})} =  \prod_{j=0}^{m}  \big\| \XX_{j}(\mathbf{t})\big\|_{L^p(\mathbb{P})} = \prod_{j=0}^{m} \big(\E[\XX_j^p(\mathbf{t})]\big)^{1/p}. 
\]
Therefore, by  Assumption~\ref{assum-Lp0-abs} and Lemma~\ref{lem-small-p}, for any $1<p\le p_0\le 2$,   we have
\[
	(\mathbb{E}[U^p_{m}])^{1/p} \leq    \prod_{j=0}^{m}   \Big[ \sup_{\mathbf{t}\in [0,1)^d}\big (\E[\XX_j^p(\mathbf{t})]\big)^{1/p}\Big]<\infty.
\]
This completes the proof of the lemma. 
\end{proof}

\subsection{Proof of Proposition~\ref{non-degene-abs}}
We now turn to the proof of Proposition~\ref{non-degene-abs}. By Lemma~\ref{lem-finite-m}, to prove  the desired inequality \eqref{propA-goal}, it suffices to prove 
\begin{align}\label{large-m-sup}
\sup_{m>k_0} \E[U_m^p]<\infty,
\end{align}
where $k_0$ is the integer given as in Assumption~\ref{assum-indep-abs}. 

We divide the proof of the inequality \eqref{large-m-sup} into the following six steps. 

\medskip
{\flushleft \bf Step 1.  Use the martingale structure and  Burkholder inequality.}
\medskip

For any $m>k_0$, write
\[
U_m=U_{k_0}+\sum_{k=k_0+1}^{m} (U_{k}-U_{k-1}).
\]
Since $(U_m)_{m\ge 0}$ is a martingale and $1<p\leq p_0\leq2$, the classical Burkholder inequality \eqref{def-Mtype-scalar} implies 
\begin{align}\label{fir-Burk-abs}
\mathbb{E}[U_m^p]\lesssim_p \mathbb{E}[U^p_{k_0}]+ \sum_{k=k_0+1}^{m} \mathbb{E}\big[|U_{k}-U_{k-1}|^p\big],
\end{align}
where $\lesssim_p$ means that the above inequality holds up to a multiplicative constant depending only on $p$.

\medskip
{\flushleft \bf Step 2. Use the sub-exponential partition in  Assumption~\ref{assum-indep-abs}.}
\medskip

Recall the notation \eqref{def-b-dya-abs} for the $b$-adic structure.  For each $b$-adic sub-cube $\mathbf{I}\in\mathscr{D}_{k-1}^b$, set
\[
\mathcal{R}_{\mathbf{I}}:=\int_{\mathbf{I}}\Big[\prod_{j=0}^{k-1}\XX_{j}(\mathbf{t})\Big]\mathring{\XX}_k(\mathbf{t})\mathrm{d}\mathbf{t}  \quad \text{with $\mathring{\XX}_{k}(\mathbf{t}):=\XX_k(\mathbf{t})-1$}.
\] 
For each $k>k_0$, by using the decomposition 
\[
[0,1)^d = \bigsqcup_{\mathbf{I}\in\mathscr{D}_{k-1}^b}\mathbf{I},
\]
then we get
\[
U_k - U_{k-1} = \mu_{k}([0,1]^d)-\mu_{k-1}([0,1]^d)=\int_{[0,1]^d}\Big[\prod_{j=0}^{k-1}\XX_{j}(\mathbf{t})\Big]\mathring{\XX}_k(\mathbf{t})\mathrm{d}\mathbf{t}=\sum_{\mathbf{I}\in\mathscr{D}_{k-1}^b}\mathcal{R}_{\mathbf{I}}.
\]
By applying the sub-exponential partition  \eqref{dec-subexp-dec-abs} of  the family $\mathscr{D}_{k-1}^b$ in Assumption~\ref{assum-indep-abs},  we obtain 
\[
U_{k}-U_{k-1}=\sum_{i=1}^{N_{k-1}}\sum_{\mathbf{I}\in\mathscr{D}_{k-1,i}^{b}}\mathcal{R}_{\mathbf{I}}   \quad (\text{where  $\mathscr{D}_{k-1}^b=\bigsqcup_{i=1}^{N_{k-1}}\mathscr{D}_{k-1,i}^{b}$})
\]
and hence 
\begin{align}\label{diff-ineq-p}
\E[|U_{k}-U_{k-1}|^p]\le N_{k-1}^{p-1}  \sum_{i=1}^{N_{k-1}}  \E\Big[ \Big|\sum_{\mathbf{I}\in\mathscr{D}_{k-1,i}^{b}}\mathcal{R}_{\mathbf{I}}   \Big|^p\Big], 
\end{align}
where we used the following Jensen's inequality for the function $x\mapsto x^p$: 
\begin{align}\label{ele-ine-jen-abs}
	\Big( \frac{1}{N_{k-1}} \sum_{i=1}^{N_{k-1}}|x_i|\Big)^p\leq \frac{1}{N_{k-1}}\sum_{i=1}^{N_{k-1}}|x_i|^p.
\end{align}

\medskip
{\flushleft \bf Step 3. Use  the independence  condition in Assumption~\ref{assum-indep-abs} and Burkholder inequality.}
\medskip

Fix any integer $1\le i \le N_{k-1}$.     By Assumption~\ref{assum-indep-abs}, conditioned on the natural filtration $\mathscr{G}_{k-1}$ defined in \eqref{def-filtra-abs}, the random variables $\{\mathcal{R}_{\mathbf{I}}\}_{\mathbf{I}\in\mathscr{D}_{k-1,i}^{b}}$ are jointly independent and $\E[\mathcal{R}_\mathbf{I}|\mathscr{G}_{k-1}] = 0$.  Therefore, with respect to the conditional expectation
\[
\E_{k-1}[\cdot]=\E[\cdot|\mathscr{G}_{k-1}],
\]
we may apply the Burkholder inequality  \eqref{def-ind-Mtype-scalar} (since $1<p<p_0\le 2$) and obtain 
\[
\mathbb{E}_{k-1}\Big[\Big|\sum_{\mathbf{I}\in\mathscr{D}_{k-1,i}^{b}}\mathcal{R}_{\mathbf{I}}\Big|^p\Big]\lesssim_p \sum_{\mathbf{I}\in\mathscr{D}_{k-1,i}^{b}}\mathbb{E}_{k-1}[|\mathcal{R}_{\mathbf{I}}|^p].
\]
Hence, by taking expectation on both sides, we get
\begin{align}\label{cond-Burk}
\mathbb{E}\Big[\Big|\sum_{\mathbf{I}\in\mathscr{D}_{k-1,i}^{b}}\mathcal{R}_{\mathbf{I}}\Big|^p\Big]\lesssim_p \sum_{\mathbf{I}\in\mathscr{D}_{k-1,i}^{b}}\mathbb{E}[|\mathcal{R}_{\mathbf{I}}|^p].
\end{align}

\medskip
{\flushleft \bf Step 4. Reduce to local estimate.}
\medskip

Combining \eqref{fir-Burk-abs},  \eqref{diff-ineq-p} and \eqref{cond-Burk}, we obtain the upper-estimate of $\E[U_m^p]$ for $m>k_0$ by local estimate: 
\begin{align}\label{sesec-Burk-abs}
\begin{split}
	\mathbb{E}[U_m^p]&\lesssim_p \mathbb{E}[U^p_{k_0}]+\sum_{k=k_0+1}^{m} N_{k-1}^{p-1}\sum_{i=1}^{N_{k-1}}\sum_{\mathbf{I}\in\mathscr{D}_{k-1,i}^{b}}\mathbb{E}[|\mathcal{R}_{\mathbf{I}}|^p]  
	\\
&=\mathbb{E}[U^p_{k_0}]+\sum_{k=k_0+1}^{m}N_{k-1}^{p-1}\sum_{\mathbf{I}\in\mathscr{D}_{k-1}^b}\mathbb{E}[|\mathcal{R}_{\mathbf{I}}|^p] \quad (\text{here use again $\mathscr{D}_{k-1}^b=\bigsqcup_{i=1}^{N_{k-1}}\mathscr{D}_{k-1,i}^{b}$}).
\end{split}
\end{align}

\medskip
{\flushleft \bf Step 5. Use triangle inequality and Assumption~\ref{assum-Lp0-abs}.}
\medskip

In this step, we use Assumption~\ref{assum-Lp0-abs} to control the local estimate $\E[|\mathcal{R}_\mathbf{I}|^p]$.  For each $k>k_0$ and $\mathbf{I}\in\mathscr{D}_{k-1}^b$, by the triangle inequality, we have
\[
(\mathbb{E}[|\mathcal{R}_{\mathbf{I}}|^p])^{1/p}= \Big\|\int_{\mathbf{I}}\Big[\prod_{j=0}^{k-1}\XX_{j}(\mathbf{t})\Big]\mathring{\XX}_k(\mathbf{t})\mathrm{d}\mathbf{t}\Big\|_{L^p(\PP)}\leq\int_{\mathbf{I}}\Big\|\Big[\prod_{j=0}^{k-1}\XX_{j}(\mathbf{t})\Big]|\mathring{\XX}_k(\mathbf{t})|\Big\|_{L^p(\mathbb{P})}\mathrm{d}\mathbf{t}.  
\]
Since the stochastic processes $\{\XX_j\}_{j\ge 0}$ are independent, for any $\mathbf{t}\in \mathbf{I}$,  we have
\[
\Big\|\Big[\prod_{j=0}^{k-1}\XX_{j}(\mathbf{\mathbf{t}})\Big]|\mathring{\XX}_k(\mathbf{\mathbf{t}})|\Big\|_{L^p(\mathbb{P})}=\Big(\Big[\prod_{j=0}^{k-1}\E[\XX_j^p(\mathbf{t})]\Big]\cdot\E[|\mathring{\XX}_k(\mathbf{t})|^p]\Big)^{1/p}
\]
and hence 
\begin{align}\label{Y-local}
 \mathbb{E}[|\mathcal{R}_{\mathbf{I}}|^p] \le \sup_{\mathbf{t}\in [0,1)^d}  \Big[\prod_{j=0}^{k-1}\E[\XX_j^p(\mathbf{t})]\Big]\cdot\E[|\mathring{\XX}_k(\mathbf{t})|^p]   \cdot  |\mathbf{I}|^p,
\end{align}
where $|\mathbf{I}| = b^{d(k-1)}$ is the volume of $\mathbf{I}$.  Therefore, combining \eqref{Y-local} with  \eqref{ass2-sma-p-abs}, we get 
\begin{align}\label{desired-Y-local}
 \mathbb{E}[|\mathcal{R}_{\mathbf{I}}|^p]   \le C_p \cdot S_p^{k-j_0} \cdot \E[|\mathring{\XX}_k(\mathbf{t})|^p]  \cdot b^{-d(k-1)p}\lesssim S_p^k\cdot b^{-dkp},
\end{align}
where $\lesssim$ means that  the above inequality holds up to a multiplicative constant $C'$ depending only on the constants  $C_p, d, b, p, j_0$. In particular, here we used the fact that 
\[
\E[|\mathring{\XX}_k(\mathbf{t})|^p]  \le 2^p \cdot \E[\XX^p_k(\mathbf{t})].
\]

\medskip
{\flushleft \bf Step 6. Conclusion.}
\medskip

By \eqref{sesec-Burk-abs} and \eqref{desired-Y-local}, as well as the cardinality  equality $\#\mathscr{D}_{k-1}^b=b^{d(k-1)}$,  for any $m>k_0$, 
\[
\mathbb{E}[U^p_{m}]\lesssim\mathbb{E}[U^p_{k_0}]+\sum_{k=k_0+1}^{m} N_{k-1}^{p-1}\cdot S_p^k\cdot b^{-dkp}\cdot\#\mathscr{D}_{k-1}^b \lesssim\mathbb{E}[U^p_{k_0}]+\sum_{k=k_0+1}^{m} N_{k-1}^{p-1}\cdot S_p^k\cdot b^{-dk(p-1)}.
\]
By the definition \eqref{def-phi-abs-again} of $\zeta(p)$ and the definition \eqref{ass2-sma-p-abs} of $S_p$, we have
\[
\zeta(p)=d(p-1)\log b-\log\Big(\sup_{j>j_0}\sup_{\mathbf{t}\in[0,1)^d}\mathbb{E}[\XX_j^{p}(\mathbf{t})]\Big) = d (p-1) \log b - \log S_p
\]
and hence 
\begin{align}\label{Sp-phi-abs}
	S_p=b^{d(p-1)-\frac{\zeta(p)}{\log b}}. 
\end{align}
Moreover, by Lemma~\ref{lem-small-p}, we have
$\zeta(p)>0$ for any $1<p\le p_0$.
It follows that 
\[
\sup_{m>k_0}\mathbb{E}[U^p_{m} ]\lesssim\mathbb{E}[U^p_{k_0}]+\sum_{k=k_0+1}^{\infty} N_{k-1}^{p-1}\cdot b^{-k\frac{\zeta(p)}{\log b}}<\infty,
\]
where in the last step, the inequality $\mathbb{E}[U^p_{k_0}]<\infty$ follows from  Lemma~\ref{lem-finite-m}  and the convergence of the  series  follows from   the sub-exponential growth condition  \eqref{def-subexp-abs} of $N_{k-1}$.
This completes the proof of the desired inequality \eqref{large-m-sup} and thus the whole proof of Proposition~\ref{non-degene-abs}.

\section{The unified theorem for polynomial Fourier decay}\label{S-Pf-Thm-PolyFoudec-abs}
This section is devoted to the proof of the unified Theorem~\ref{Fou-Dec-abs}.   We are going to prove Theorem~\ref{Fou-Dec-abs} following the strategy outlined in \S\ref{S-metvecvalmar-abs}. The procedure is presented as in  Figure~\ref{figure-procedure-proving}. 

The main idea  of the proof of  Theorem~\ref{Fou-Dec-abs} is similar to the proof of its counterpart of  in our previous paper \cite[Theorem 1.1]{LQT24} on 1D GMC (note that the proof in \cite{LQT24} only works for  1D GMC since it relies on the specific whitenoise decomposition,  which can not be directly generalized to higher dimension).  For the reader's convenience, we write the proof in a parallel way to that of \cite[Theorem 1.1]{LQT24}.  The reader is invited to take a look at the proof of  \cite[Theorem 1.1]{LQT24} to get the main idea. 

Throughout this section, to simplify notation and to ease the comparison between the proof of Theorem~\ref{Fou-Dec-abs} and that of \cite[Theorem 1.1]{LQT24},   by writing $\mathbb{Z}^{d}$ as the union of $2^d$ orthants and  by slightly modifying the notation,  we may  and will always restrict the indices $\mathbf{n}\in\mathbb{Z}^{d}$ to $\mathbf{n}\in\mathbb{N}^{d}\setminus\{\mathbf{0}\}$ in all the quantities appeared in \S\ref{subS-metvecvalmar-abs}.

\subsection{Reduction to a localization estimate}
In this subsection, we shall prove Lemma~\ref{dec-vm-abs} to obtain the reduction to a localization estimate. 
The proof of Lemma~\ref{dec-vm-abs} relies on  twice applications of Pisier's martingale type $p$ inequalities  \eqref{def-Mtype} and \eqref{def-ind-Mtype} for the Banach space $\ell^q$ with $1<p\leq2\leq q<\infty$. 

Recall the definitions of $\mathcal{L}_F$ in \eqref{def-LF-abs} (note that in particular, $\mathcal{L}_F \le 2 \alpha_0$).   From now on, let us  fix positive numbers $\tau$, $p$, $q$ such that 
\begin{align}\label{fix-exponents}
 \tau\in (0,    \mathcal{L}_F)  \subset (0, 2\alpha_0) \anand  1<p\leq p_0\leq\max\Big\{2,\frac{2d}{2\alpha_0-\tau}\Big\}<q<\infty,
\end{align}
where $\alpha_0$ and $p_0$ are the parameters given in Assumptions \ref{assum-Lp0-abs} and \ref{assum-alp0-abs}. 

Consider the $\ell^q$-valued martingale $(\mathcal{M}_m)_{m\geq 0}$ defined in \eqref{vec-val-mar-abs} with respect to the increasing filtration $(\mathscr{G}_m)_{m\ge 0}$ defined in \eqref{def-filtra-abs}.  Recall the integer $k_0>0$ given in Assumption \ref{assum-indep-abs}.  For any $m>k_0$, write $\mathcal{M}_m$ as the sum of martingale differences: 
\[
\mathcal{M}_m =\mathcal{M}_{k_0} + \sum_{k=k_0+1}^{m}(\mathcal{M}_{k}-\mathcal{M}_{k-1}).
\]
By the martingale type $p$ inequality \eqref{def-Mtype} for the Banach space $\ell^q$, we have
\begin{align}\label{fir-mar-ty-ine-abs}
	\mathbb{E}[\|\mathcal{M}_m\|_{\ell^q}^p]\lesssim_{p,q}  \E[\|\mathcal{M}_{k_0}\|_{\ell^q}^p] + \sum_{k=k_0+1}^{m}\mathbb{E}[\|\mathcal{M}_{k}-\mathcal{M}_{k-1}\|_{\ell^q}^p].
\end{align}

Now, for each $k>k_0$ and $\mathbf{n}=(n_1,\cdots,n_d)\in\mathbb{N}^{d}\setminus\{\mathbf{0}\}$,  by the definitions \eqref{def-num-abs} and  \eqref{vec-val-mar-abs}, we have  
\[
\mathcal{M}_k(\mathbf{n})=|\mathbf{n}|^{\frac{\tau}{2}}\widehat{\mu_{k}}(\mathbf{n})=|\mathbf{n}|^{\frac{\tau}{2}}\int_{[0,1]^d}e^{-2\pi i\mathbf{n}\cdot\mathbf{t}}\mu_{k}(\mathrm{d}\mathbf{t})=|\mathbf{n}|^{\frac{\tau}{2}}\int_{[0,1]^d}\Big[\prod_{j=0}^{k}\XX_j(\mathbf{t})\Big]e^{-2\pi i\mathbf{n}\cdot\mathbf{t}}\mathrm{d}\mathbf{t}.
\]
Hence we get
\[
\mathcal{M}_{k}(\mathbf{n})-\mathcal{M}_{k-1}(\mathbf{n})=|\mathbf{n}|^{\frac{\tau}{2}}\int_{[0,1]^d}\Big[\prod_{j=0}^{k-1}\XX_j(\mathbf{t})\Big]\mathring{\XX}_k(\mathbf{t})e^{-2\pi i\mathbf{n}\cdot\mathbf{t}}\mathrm{d}\mathbf{t} \quad \text{with $\mathring{\XX}_k(\mathbf{t})=\XX_k(\mathbf{t})-1$}.
\]
Recall the notation $\mathscr{D}_{k-1}^b$ introduced in \eqref{def-b-dya-abs} for the family of $b$-adic sub-cubes and the definition \eqref{def-Y-abs} for the random vector $\mathscr{Y}_\mathbf{I}$ defined for any $b$-adic sub-cube $\mathbf{I}\in \mathscr{D}_{k-1}^b$: 
\[
\mathscr{Y}_\mathbf{I}(\mathbf{n})= |\mathbf{n}|^{\frac{\tau}{2}}\int_{\mathbf{I}}\Big[\prod_{j=0}^{k-1}\XX_j(\mathbf{t})\Big]  \mathring{\XX}_{k} (\mathbf{t}) e^{-2\pi i\mathbf{n}\cdot\mathbf{t}}\mathrm{d}\mathbf{t}.
\] 
Therefore, we obtain
\[
\mathcal{M}_{k}(\mathbf{n})-\mathcal{M}_{k-1}(\mathbf{n})= \sum_{\mathbf{I}\in \mathscr{D}_{k-1}^b}\mathscr{Y}_\mathbf{I}(\mathbf{n}) \quad \text{for all $\mathbf{n}\in\mathbb{N}^{d}\setminus\{\mathbf{0}\}$}. 
\]
That is, as random vectors, we have the equality 
\[
\mathcal{M}_k-\mathcal{M}_{k-1}= \sum_{\mathbf{I}\in \mathscr{D}_{k-1}^b}\mathscr{Y}_\mathbf{I}. 
\]

For using again the martingale type $p$ inequality, we  consider the sub-exponential decomposition \eqref{dec-subexp-dec-abs}:
\[
\mathscr{D}_{k-1}^b=\bigsqcup_{i=1}^{N_{k-1}}\mathscr{D}_{k-1,i}^{b}\quad\text{with all $\mathscr{D}_{k-1,i}^{b}$ being the sub-families of $\mathscr{D}_{k-1}^b$}.
\]
Then we have 
\[
\mathcal{M}_{k}-\mathcal{M}_{k-1}=\sum_{i=1}^{N_{k-1}}\sum_{\mathbf{I}\in\mathscr{D}_{k-1,i}^{b}}\mathscr{Y}_\mathbf{I}. 
\]
Therefore, by the triangle inequality of $\ell^q$, we have 
\[
\|\mathcal{M}_{k}-\mathcal{M}_{k-1}\|_{\ell^q}^p\leq\Big(\sum_{i=1}^{N_{k-1}}\Big\|\sum_{\mathbf{I}\in\mathscr{D}_{k-1,i}^{b}}\mathscr{Y}_\mathbf{I}\Big\|_{\ell^q}\Big)^p
\]
and hence by the Jensen's inequality \eqref{ele-ine-jen-abs}, we obtain
\[
\|\mathcal{M}_{k}-\mathcal{M}_{k-1}\|_{\ell^q}^p\leq N_{k-1}^{p-1}\sum_{i=1}^{N_{k-1}}\Big\|\sum_{\mathbf{I}\in\mathscr{D}_{k-1,i}^{b}}\mathscr{Y}_\mathbf{I}\Big\|_{\ell^q}^p.
\]
It follows that
\begin{align}\label{D-D-dec-abs}
	\E[\|\mathcal{M}_{k}-\mathcal{M}_{k-1}\|_{\ell^q}^p] \leq  N_{k-1}^{p-1}\sum_{i=1}^{N_{k-1}}\E\Big[ \Big\|  \sum_{\mathbf{I}\in\mathscr{D}_{k-1,i}^{b}}\mathscr{Y}_\mathbf{I} \Big\|_{\ell^q}^p\Big]. 
\end{align}

Now a crucial observation is that, by Assumption~\ref{assum-indep-abs}, for each $i=1,2,\cdots,N_{k-1}$, conditioned on the filtration $\mathscr{G}_{k-1}$,   the random vectors $\mathscr{Y}_\mathbf{I}$ indexed by $\mathbf{I}\in\mathscr{D}_{k-1,i}^{b}$ are independent (and also conditionally centered by \eqref{Y-I-cond-abs}) and hence with respect to the conditional expectation
\[
\E_{k-1}[\cdot]=\E[\cdot|\mathscr{G}_{k-1}],
\]
we may apply the martingale type $p$ inequality  \eqref{def-ind-Mtype} for the Banach space $\ell^q$ and obtain 
\[
\E_{k-1}\Big[\Big\|\sum_{\mathbf{I}\in\mathscr{D}_{k-1,i}^{b}}\mathscr{Y}_\mathbf{I}\Big\|_{\ell^q}^p\Big]\lesssim_{p,q}\sum_{\mathbf{I}\in\mathscr{D}_{k-1,i}^{b}}\E_{k-1}[\|\mathscr{Y}_\mathbf{I}\|_{\ell^q}^p]. 
\]
Hence, by taking expectation on both sides, we get 
\begin{align}\label{aft-conexp-abs}
	\E\Big[\Big\|\sum_{\mathbf{I}\in\mathscr{D}_{k-1,i}^{b}}\mathscr{Y}_\mathbf{I}\Big\|_{\ell^q}^p\Big]\lesssim_{p,q}\sum_{\mathbf{I}\in\mathscr{D}_{k-1,i}^{b}}\E[\|\mathscr{Y}_\mathbf{I}\|_{\ell^q}^p].
\end{align}

Finally,  by  combining the inequalities \eqref{fir-mar-ty-ine-abs}, \eqref{D-D-dec-abs} and \eqref{aft-conexp-abs},  we get  the desired inequality 
\begin{align*}
	\mathbb{E}[\|\mathcal{M}_{m}\|_{\ell^q}^p]&\lesssim_{p,q}  \E[\|\mathcal{M}_{k_0}\|_{\ell^q}^p]  +   \sum_{k=k_0+1}^{m}N_{k-1}^{p-1}\sum_{i=1}^{N_{k-1}}\sum_{\mathbf{I}\in\mathscr{D}_{k-1,i}^{b}}\mathbb{E}[\|\mathscr{Y}_\mathbf{I} \|_{\ell^q}^p]\\
	&=  \E[\|\mathcal{M}_{k_0}\|_{\ell^q}^p]  +   \sum_{k=k_0+1}^{m}N_{k-1}^{p-1}\sum_{\mathbf{I} \in \mathscr{D}_{k-1}^b}\mathbb{E}[\|\mathscr{Y}_\mathbf{I} \|_{\ell^q}^p].
\end{align*}
This completes the proof of Lemma~\ref{dec-vm-abs}.

\subsection{Establishment of the localization estimate} 
This subsection is devoted to the proof of Lemma~\ref{UB-YZW-abs} on the localization estimate of $\mathbb{E}[\|\mathscr{Z}_\mathbf{I} \|_{\ell^q}^p]$. 

Recall the definition of $\zeta(p)=\zeta_{j_0}(p)$ in \eqref{def-phi-abs} (and repeated in \eqref{def-phi-abs-again}, in particular, here $j_0$ is a fixed integer $j_0\ge 1$).    Recall also that $\tau$, $p$, $q$ are fixed as in \eqref{fix-exponents}.   In what follows, by writing $A\lesssim B$, we mean that there exists a finite constant $C=C(d,b,\tau,p,q,j_0)>0$ depending on the parameters $d$, $b$, $\tau$, $p$, $q$, $j_0$ such that $A\leq CB$.

\subsubsection{An outline of the proof of Lemma \ref{UB-YZW-abs}}
The proof of Lemma~\ref{UB-YZW-abs} is divided into the following twelve steps and the main steps are outlined as follows: take  any integer $m\geq0$ and any $b$-adic sub-cube $\mathbf{I}\in \mathscr{D}_{m}^b$ defined in \eqref{def-b-dya-abs}, recall the $\mathscr{G}_m$-measurable random vector $\mathscr{Z}_\mathbf{I}=\big(\mathscr{Z}_\mathbf{I}(\mathbf{n})\big)_{\mathbf{n}\in\mathbb{N}^{d}\setminus\{\mathbf{0}\}}$ defined in \eqref{def-Z-abs}:
\[
	\mathscr{Z}_\mathbf{I}(\mathbf{n})=|\mathbf{n}|^{\frac{\tau}{2}}\int_{\mathbf{I}}\Big[\prod_{j=0}^{m}\XX_j(\mathbf{t})\Big]e^{-2\pi i\mathbf{n}\cdot\mathbf{t}}\mathrm{d}\mathbf{t} \anand 
\|\mathscr{Z}_\mathbf{I}\|_{\ell^q}^p = \Big\{\sum_{\mathbf{n}\in\mathbb{N}^{d}\setminus\{\mathbf{0}\}}|\mathscr{Z}_\mathbf{I}(\mathbf{n})|^q\Big\}^{p/q}.
\]
The key in obtaining the desired upper estimate of $\E[\|\mathscr{Z}_\mathbf{I}\|_{\ell^q}^p]$ is to establish  the following separation-of-variable pointwise upper estimate for $|\mathscr{Z}_\mathbf{I}(\mathbf{n})|$ (which is inspired by the standard Littlewood-Paley decomposition in harmonic analysis): 
\[
|\mathscr{Z}_\mathbf{I}(\mathbf{n})| \le    v_0(\mathbf{n}) R_0 +  \sum_{L=1}^\infty v_{L}(\mathbf{n})  R_{L} + \sum_{L=1}^\infty\sum_{\beta=1}^{d} w_{L,\beta}(\mathbf{n})  Q_{L,\beta} \quad \text{for all $\mathbf{n}\in\mathbb{N}^{d}\setminus\{\mathbf{0}\}$},
\]
where $R_0$, $R_{L}$, $Q_{L,\beta}$ are non-negative random variables and  $v_0$, $v_{L}$, $w_{L,\beta}$ are deterministic (without randomness)  sequences of non-negative numbers with supports 
\[
\supp(v_0)=\{1\leq|\mathbf{n}|\leq b^m\},\quad\supp(v_{L})=\{b^{m+L-1}<|\mathbf{n}|\leq b^{m+L}\}
\]
and
\[
\supp(w_{L,\beta})=\Big\{b^{m+L-1}<|\mathbf{n}|\leq b^{m+L},n_{\beta}>\frac{b^{m+L-1}}{\sqrt{d}},n_{\beta'}\leq\frac{b^{m+L-1}}{\sqrt{d}},1\leq\beta'\leq\beta-1\Big\}.
\]

In particular, for any $L\ge 1$ and $1\leq\beta\leq d$, the constructions of $v_{L}$, $w_{L,\beta}$ and $R_{L}$, $Q_{L,\beta}$  rely on a $b$-adic-discrete-time approximation of the stochastic process  
\[
\prod_{j=0}^{m}\XX_j(\mathbf{t}),\quad\mathbf{t}\in\mathbf{I}. 
\]
The level of the discrete-time approximation  depends on each $b$-adic sub-cube  $\mathbf{I}$ and on $\mathbf{n}$ via its $b$-adic level $b^{m+L-1}<|\mathbf{n}|\leq b^{m+L}$.  It should be mentioned that, such approximation is reasonable (meaning that the difference can be controlled) by Assumption~\ref{assum-Lp0-abs} and Assumption~\ref{assum-alp0-abs}.

\subsubsection{The proof of Lemma \ref{UB-YZW-abs}}
We now proceed to the proof of Lemma \ref{UB-YZW-abs}. 

\medskip
{\flushleft \bf Step 1. The lower-frequency part $1\leq|\mathbf{n}|\leq b^m$.}
\medskip

For the lower-frequency part $1\leq|\mathbf{n}|\leq b^m$,  the quantities $\mathscr{Z}_{\mathbf{I}}(\mathbf{n})$ are controlled by the total mass of $\mu_{m}$ on the $b$-adic sub-cube $\mathbf{I}$. More precisely, here we use very rough upper estimate of $\mathscr{Z}_\mathbf{I}(\mathbf{n})$: 
\[
|\mathscr{Z}_\mathbf{I}(\mathbf{n})|=  \Big|  |\mathbf{n}|^{\frac{\tau}{2}}\int_{\mathbf{I}}\Big[\prod_{j=0}^{m}\XX_j(\mathbf{t})\Big] e^{-2\pi i\mathbf{n}\cdot\mathbf{t}}\mathrm{d}\mathbf{t}  \Big| \le      |\mathbf{n}|^{\frac{\tau}{2}}\int_{\mathbf{I}}\Big[\prod_{j=0}^{m}\XX_j(\mathbf{t})\Big] \mathrm{d}\mathbf{t}. 
\]
Hence by defining
\begin{align}\label{def-v-0-abs}
	v_0(\mathbf{n}): = |\mathbf{n}|^{\frac{\tau}{2}}\cdot \mathds{1}(1\le |\mathbf{n}| \le b^m)
\end{align}
and 
\begin{align}\label{def-R-0-abs}
	R_0: = \int_{\mathbf{I}}\Big[\prod_{j=0}^{m}\XX_j(\mathbf{t})\Big] \mathrm{d}\mathbf{t},
\end{align}
we obtain 
\begin{align}\label{low-Y-abs}
	|\mathscr{Z}_\mathbf{I}(\mathbf{n})| \le v_0(\mathbf{n})  R_0 \quad \text{for all $1\le |\mathbf{n}| \le b^m$}.
\end{align}

\medskip
{\flushleft \bf Step 2. $b$-adic-discrete-time approximation for the higher-frequency part.}
\medskip

For the higher-frequency part $\mathscr{Z}_\mathbf{I}(\mathbf{n})$ with $|\mathbf{n}|>b^m$,  we shall use a  finer estimate by applying a $b$-adic-discrete-time approximation of the stochastic process   
\begin{align}\label{def-Dk-abs}
	\mathcal{D}_m(\mathbf{t}): = \Big[\prod_{j=0}^{m} \XX_j(\mathbf{t})\Big], \quad \mathbf{t}\in \mathbf{I}. 
\end{align}
Namely, we shall approximate $\mathcal{D}_m(\mathbf{t})$ by the value of $\mathcal{D}_m$ at some $b$-adic $\mathbf{t}$. It is important for our purpose to use  finer and finer approximation of $\mathcal{D}_m(\mathbf{t})$ when $|\mathbf{n}|$ grows. That is, to control $\mathscr{Z}_\mathbf{I}(\mathbf{n})$, the level of the $b$-adic-discrete-time approximation depends on each $b^{m+L-1}<|\mathbf{n}|\leq b^{m+L}$.

More precisely,  given any integer $L\ge 1$,  by using the same $b$-adic decomposition of $\mathbf{I}$, we shall decompose $\mathscr{Z}_\mathbf{I}(\mathbf{n})$  of the same  manner for all $b^{m+L-1}<|\mathbf{n}|\leq b^{m+L}$. That is, we divide the $b$-adic sub-cube  $\mathbf{I}\in\mathscr{D}_{m}^b$  into $b^{dL}$ sub-pieces (hence each sub-cube has side length $b^{-(m+L)}$). In other words, denote by  $\mathscr{D}_{m+L}^b(\mathbf{I})$ the family of sub-cubes $\mathbf{J}\subset\mathbf{I}$ in $\mathscr{D}_{m+L}^b$:
\begin{align}\label{DLk-1I-abs}
	\mathscr{D}_{m+L}^b(\mathbf{I}): = \Big\{\mathbf{J}\subset\mathbf{I}:   \mathbf{J}\in \mathscr{D}_{m+L}^b \Big\}. 
\end{align}
By using  the decomposition 
\[
\mathbf{I}=\bigsqcup_{\mathbf{J}\in\mathscr{D}_{m+L}^b(\mathbf{I})}\mathbf{J}, 
\]
we can  decompose $\mathscr{Z}_\mathbf{I}(\mathbf{n})$ as 
\[
\mathscr{Z}_\mathbf{I}(\mathbf{n})=|\mathbf{n}|^{\frac{\tau}{2}}\int_{\mathbf{I}}\mathcal{D}_{m}(\mathbf{t})e^{-2\pi i\mathbf{n}\cdot\mathbf{t}}\mathrm{d}\mathbf{t}=|\mathbf{n}|^{\frac{\tau}{2}}\sum_{\mathbf{J}\in\mathscr{D}_{m+L}^b(\mathbf{I})}\int_{\mathbf{J}}\mathcal{D}_{m}(\mathbf{t})e^{-2\pi i\mathbf{n}\cdot\mathbf{t}}\mathrm{d}\mathbf{t}. 
\]

Then on each sub-cube $\mathbf{J}\in\mathscr{D}_{m+L}^b(\mathbf{I})$, we approximate $\mathcal{D}_m(\mathbf{t})$ with $\mathcal{D}_m$ evaluated on the minimum vertex of $\mathbf{J}$. Here the minimum vertex of $\mathbf{J}$ is defined by 
\[
\ell_\mathbf{J}:=\Big(\frac{h_1-1}{b^{m+L}},\frac{h_2-1}{b^{m+L}},\cdots,\frac{h_d-1}{b^{m+L}}\Big)\quad\text{if}\quad\mathbf{J} = \prod_{\beta=1}^{d}\Big[\frac{h_\beta-1}{b^{m+L}}, \frac{h_\beta}{b^{m+L}}\Big)\in\mathscr{D}_{m+L}^b(\mathbf{I})\subset\mathscr{D}_{m+L}^b.
\]
Using the decomposition 
\[
\mathcal{D}_m(\mathbf{t})=[\mathcal{D}_m(\mathbf{t})-\mathcal{D}_m(\ell_\mathbf{J})]+\mathcal{D}_m(\ell_\mathbf{J}),  
\]
we obtain that for any $b^{m+L-1}<|\mathbf{n}|\leq b^{m+L}$,
\begin{align}\label{Y-UV-abs}
	\begin{split}
		\mathscr{Z}_\mathbf{I}(\mathbf{n}) =   \underbrace{ |\mathbf{n}|^{\frac{\tau}{2}}\sum_{\mathbf{J}\in \mathscr{D}_{m+L}^b(\mathbf{I})}  \int_{\mathbf{J}} [\mathcal{D}_{m} (\mathbf{t}) -\mathcal{D}_m(\ell_\mathbf{J})]e^{-2\pi i\mathbf{n}\cdot\mathbf{t}}\mathrm{d}\mathbf{t}}_{\text{denoted $\mathscr{V}_\mathbf{I}(\mathbf{n})$}}+ \underbrace{|\mathbf{n}|^{\frac{\tau}{2}}\sum_{\mathbf{J}\in \mathscr{D}_{m+L}^b(\mathbf{I})}  \int_{\mathbf{J}}  \mathcal{D}_m(\ell_\mathbf{J}) e^{-2\pi i\mathbf{n}\cdot\mathbf{t}}\mathrm{d}\mathbf{t}}_{\text{denoted $\mathscr{U}_\mathbf{I}(\mathbf{n})$}}. 
	\end{split}
\end{align}
The two terms $\mathscr{V}_\mathbf{I}(\mathbf{n})$ and $\mathscr{U}_\mathbf{I}(\mathbf{n})$ will be controlled by different methods. 

\medskip
{\flushleft \bf Step 3. The simple control of $\mathscr{V}_\mathbf{I}(\mathbf{n})$.}
\medskip

The term $\mathscr{V}_\mathbf{I}(\mathbf{n})$ defined in \eqref{Y-UV-abs} is controlled directly by using the triangle inequality: 
\[
|\mathscr{V}_\mathbf{I}(\mathbf{n})| \le |\mathbf{n}|^{\frac{\tau}{2}} \sum_{\mathbf{J}\in \mathscr{D}_{m+L}^b(\mathbf{I})}  \int_{\mathbf{J}} |\mathcal{D}_{m} (\mathbf{t}) -\mathcal{D}_n(\ell_\mathbf{J}) | \mathrm{d}\mathbf{t}. 
\]
Hence for all $L\leq1$, by defining 
\begin{align}\label{def-v-l-abs}
	v_L(\mathbf{n}) : = |\mathbf{n}|^{\frac{\tau}{2}}\cdot  \mathds{1}(b^{m+L-1} <|\mathbf{n}| \le b^{m+L})
\end{align}
and
\begin{align}\label{def-Rl-abs}
	R_L: =   \sum_{\mathbf{J}\in \mathscr{D}_{m+L}^b(\mathbf{I})}  \int_{\mathbf{J}} |\mathcal{D}_{m} (\mathbf{t}) -\mathcal{D}_m(\ell_\mathbf{J}) | \mathrm{d}\mathbf{t}, 
\end{align}
we obtain 
\begin{align}\label{V-wQ-abs}
	|\mathscr{V}_\mathbf{I}(\mathbf{n})| \le v_L(\mathbf{n}) R_L \quad \text{for all $b^{m+L-1} <|\mathbf{n}| \le b^{m+L}$}. 
\end{align}
It should be emphasized that  the random variable  $R_L$ defined as above  depends on $L$ (and of course it depends on $m$, which is determined by $\mathbf{I}$), but does not  depend on $\mathbf{n}$.   In other words, all  $b^{m+L-1} <|\mathbf{n}|\le b^{m+L}$ share the same $R_L$. 

\medskip
{\flushleft \bf Step 4. The Abel's summation method for controlling $\mathscr{U}_\mathbf{I}(\mathbf{n})$.}
\medskip

We shall apply the Abel's summation method to control the term $\mathscr{U}_\mathbf{I}(\mathbf{n})$ defined in \eqref{Y-UV-abs}.  It should be emphasized that, since the standard Abel summation works only  for one dimensional index set (say some interval in $\N$), in what follows, we need to divide the set $\{b^{m+L-1}<|\mathbf{n}|\leq b^{m+L}\}$ into $d$-parts and the Abel summation is applied on each part with respect to only one dimension. 

More precisely,  for each $L\geq1$ and $1\leq\beta\leq d$, denote
\[
\mathbf{S}_{L,\beta}:=\Big\{b^{m+L-1}<|\mathbf{n}|\leq b^{m+L},n_{\beta}>\frac{b^{m+L-1}}{\sqrt{d}},n_{\beta'}\leq\frac{b^{m+L-1}}{\sqrt{d}},1\leq\beta'\leq\beta-1\Big\}.
\]
That is, $\beta$ is the smallest integer in $\{1, 2, \cdots, d\}$ such that $n_\beta > b^{m+L-1}/\sqrt{d}$.  
Clearly,  we have
\[
\big\{b^{m+L-1}<|\mathbf{n}|\leq b^{m+L}\big\}=\bigsqcup_{\beta=1}^d\mathbf{S}_{L,\beta}.
\]
Indeed,  otherwise, there exists $\mathbf{n}\in \N^d\setminus \{0\}$ with  $b^{m+L-1}<|\mathbf{n}|\leq b^{m+L}$ and $n_\beta\le b^{m+L-1}/\sqrt{d}$ for all $\beta = 1, \cdots, d$. But this would imply that $|\mathbf{n}|\le b^{m+L-1}$ and we get a contradiction. 

Now given any cube 
\[
\mathbf{J}=\prod_{\beta'=1}^{d}\Big[\frac{h_{\beta'}-1}{b^{m+L}}, \frac{h_{\beta'}}{b^{m+L}}\Big)\in\mathscr{D}_{m+L}^b(\mathbf{I})\subset\mathscr{D}_{m+L}^b, 
\]
we define for each $1\leq\beta\leq d$, 
\[
\widehat{\mathbf{J}}_{\beta}:=\prod_{1\leq\beta'\leq d,\beta'\neq\beta}\Big[\frac{h_{\beta'}-1}{b^{m+L}}, \frac{h_{\beta'}}{b^{m+L}}\Big)\anand\mathbf{J}_\beta:=\Big[\frac{h_{\beta}-1}{b^{m+L}}, \frac{h_{\beta}}{b^{m+L}}\Big). 
\]
In other words, after a  permutation of coordinates, the cube $\mathbf{J}$ can be transformed to  the the cube
\[
\widehat{\mathbf{J}_\beta} \times \mathbf{J}_\beta. 
\]
Hence for any $\mathbf{n}\in\mathbf{S}_{L,\beta}$, we can write
\begin{align*}
\mathscr{U}_{\mathbf{I}}(\mathbf{n})&=|\mathbf{n}|^{\frac{\tau}{2}}\sum_{\mathbf{J}\in \mathscr{D}_{m+L}^b(\mathbf{I})}  \int_{\widehat{\mathbf{J}}_{\beta}}\Big[\int_{\mathbf{J}_\beta} \mathcal{D}_m(\ell_\mathbf{J})e^{-2\pi in_{\beta}t_{\beta}}\mathrm{d}t_{\beta}\Big] e^{-2\pi i\widehat{\mathbf{n}}_{\beta}\cdot\widehat{\mathbf{t}}_{\beta}}\mathrm{d}\widehat{\mathbf{t}}_{\beta}\\
&=|\mathbf{n}|^{\frac{\tau}{2}}\sum_{\substack{\text{all $\widehat{\mathbf{J}}_{\beta}$ induced by}\\\mathbf{J}\in \mathscr{D}_{m+L}^b(\mathbf{I})}}\int_{\widehat{\mathbf{J}}_{\beta}}\Big[\sum_{\substack{\text{all $\mathbf{J}_{\beta}$ induced by}\\\mathbf{J}\in \mathscr{D}_{m+L}^b(\mathbf{I})}}\int_{\mathbf{J}_\beta} \mathcal{D}_m(\ell_{\widehat{\mathbf{J}}_{\beta}},\ell_{\mathbf{J}_{\beta}})e^{-2\pi in_{\beta}t_{\beta}}\mathrm{d}t_{\beta}\Big]e^{-2\pi i\widehat{\mathbf{n}}_{\beta}\cdot\widehat{\mathbf{t}}_{\beta}}\mathrm{d}\widehat{\mathbf{t}}_{\beta},
\end{align*}
where   
\[\widehat{\mathbf{n}}_{\beta}=(n_1,\cdots,n_{\beta-1},n_{\beta+1},\cdots,n_d), \quad  \widehat{\mathbf{t}}_{\beta}=(t_1,\cdots,t_{\beta-1},t_{\beta+1},\cdots,t_d), \quad \mathrm{d}\widehat{\mathbf{t}}_{\beta}=\prod_{1\leq\beta'\leq d,\beta'\neq\beta}\mathrm{d}t_{\beta'}.
\]
Here we rewrite $\ell_\mathbf{J}=(\ell_{\widehat{\mathbf{J}}_{\beta}},\ell_{\mathbf{J}_{\beta}})$ by ignoring the coordinate order with the minimum vertex $\ell_{\widehat{\mathbf{J}}_{\beta}}$ of $\widehat{\mathbf{J}}_{\beta}$ and the left end-point $\ell_{\mathbf{J}_{\beta}}$ of $\mathbf{J}_{\beta}$. 

Ordering all the sub-intervals $\mathbf{J}_{\beta}$ induced by $\mathbf{J}\in \mathscr{D}_{m+L}^b(\mathbf{I})$ from left to right according to their natural ordering on the real line,  we get 
\[
\mathbf{J}_{\beta}(l) = [\ell_{\mathbf{J}_{\beta}(l)}, \ell_{\mathbf{J}_{\beta}(l+1)}), \quad 1\le l \le b^{L},
\]
and 
\[
\ell_{\mathbf{J}_{\beta}(l+1)}  - \ell_{\mathbf{J}_{\beta}(l)} = |\mathbf{J}_{\beta}(l)|= b^{-(m+L)}, \quad \text{i.e.}\quad 
\ell_{\mathbf{J}_{\beta}(l)} = \ell_{\mathbf{J}_{\beta}(1)}+  (l-1)\cdot b^{-(m+L)}. 
\]
Under the above notation, by using the elementary equality 
\[
\int_x^y e^{-2\pi i n t}\mathrm{d}t =  \frac{e^{-2\pi i n y} - e^{-2\pi i n x}}{- 2\pi in }, 
\]
we obtain
\[
\sum_{\substack{\text{all $\mathbf{J}_{\beta}$ induced by}\\\mathbf{J}\in \mathscr{D}_{m+L}^b(\mathbf{I})}}\int_{\mathbf{J}_\beta} \mathcal{D}_m(\ell_{\widehat{\mathbf{J}}_{\beta}},\ell_{\mathbf{J}_{\beta}})e^{-2\pi in_{\beta}t_{\beta}}\mathrm{d}t_{\beta}=\frac{1}{-2\pi in_\beta}\sum_{l=1}^{b^L}\mathcal{D}_m(\ell_{\widehat{\mathbf{J}}_{\beta}},\ell_{\mathbf{J}_{\beta}(l)})\big[e^{-2\pi in_{\beta}\ell_{\mathbf{J}_{\beta}(l+1)}}-e^{-2\pi in_{\beta}\ell_{\mathbf{J}_{\beta}(l)}}\big].
\]
An application of Abel's summation method then yields
\begin{align*}
	&\quad\,\,\sum_{\substack{\text{all $\mathbf{J}_{\beta}$ induced by}\\\mathbf{J}\in \mathscr{D}_{m+L}^b(\mathbf{I})}}\int_{\mathbf{J}_\beta} \mathcal{D}_m(\ell_{\widehat{\mathbf{J}}_{\beta}},\ell_{\mathbf{J}_{\beta}})e^{-2\pi in_{\beta}t_{\beta}}\mathrm{d}t_{\beta}\\
	&=\frac{1}{-2\pi in_\beta}\Big(\mathcal{D}_m(\ell_{\widehat{\mathbf{J}}_{\beta}},\ell_{\mathbf{J}_{\beta}(b^L)})e^{-2\pi in_{\beta}\ell_{\mathbf{J}_{\beta}(b^L+1)}}-\mathcal{D}_m(\ell_{\widehat{\mathbf{J}}_{\beta}},\ell_{\mathbf{J}_{\beta}(1)})e^{-2\pi in_{\beta}\ell_{\mathbf{J}_{\beta}(1)}}\\
	&\qquad\qquad\qquad\quad+\sum_{l=1}^{b^L-1}[\mathcal{D}_m(\ell_{\widehat{\mathbf{J}}_{\beta}},\ell_{\mathbf{J}_{\beta}(l)})-\mathcal{D}_m(\ell_{\widehat{\mathbf{J}}_{\beta}},\ell_{\mathbf{J}_{\beta}(l+1)})]e^{-2\pi in_{\beta}\ell_{\mathbf{J}_{\beta}(l+1)}}\Big).
\end{align*}
It follows that 
\begin{align*}
	&\quad\,\,\Big|\sum_{\substack{\text{all $\mathbf{J}_{\beta}$ induced by}\\\mathbf{J}\in \mathscr{D}_{m+L}^b(\mathbf{I})}}\int_{\mathbf{J}_\beta} \mathcal{D}_m(\ell_{\widehat{\mathbf{J}}_{\beta}},\ell_{\mathbf{J}_{\beta}})e^{-2\pi in_{\beta}t_{\beta}}\mathrm{d}t_{\beta}\Big|\\
	&\leq \frac{1}{2\pi n_\beta}\Big(|\mathcal{D}_m(\ell_{\widehat{\mathbf{J}}_{\beta}},\ell_{\mathbf{J}_{\beta}(b^L)})|+|\mathcal{D}_m(\ell_{\widehat{\mathbf{J}}_{\beta}},\ell_{\mathbf{J}_{\beta}(1)})|+\sum_{l=1}^{b^L-1}|\mathcal{D}_m(\ell_{\widehat{\mathbf{J}}_{\beta}},\ell_{\mathbf{J}_{\beta}(l)})-\mathcal{D}_m(\ell_{\widehat{\mathbf{J}}_{\beta}},\ell_{\mathbf{J}_{\beta}(l+1)})|\Big). 
\end{align*}
Therefore, note that $|\widehat{\mathbf{J}}_{\beta}|=b^{-(d-1)(m+L)}$,  then  for any $\mathbf{n}\in\mathbf{S}_{L,\beta}$, we have
\begin{align*}
	|\mathscr{U}_{\mathbf{I}}(\mathbf{n})|&\leq|\mathbf{n}|^{\frac{\tau}{2}}\frac{b^{-(d-1)(m+L)}}{2\pi n_{\beta}}\sum_{\substack{\text{all $\widehat{\mathbf{J}}_{\beta}$ induced by}\\\mathbf{J}\in \mathscr{D}_{m+L}^b(\mathbf{I})}}\Big(|\mathcal{D}_m(\ell_{\widehat{\mathbf{J}}_{\beta}},\ell_{\mathbf{J}_{\beta}(b^L)})|+|\mathcal{D}_m(\ell_{\widehat{\mathbf{J}}_{\beta}},\ell_{\mathbf{J}_{\beta}(1)})|\\
	&\qquad\qquad\qquad\qquad\qquad\qquad\quad+\sum_{l=1}^{b^L-1}|\mathcal{D}_m(\ell_{\widehat{\mathbf{J}}_{\beta}},\ell_{\mathbf{J}_{\beta}(l)})-\mathcal{D}_m(\ell_{\widehat{\mathbf{J}}_{\beta}},\ell_{\mathbf{J}_{\beta}(l+1)})|\Big).
\end{align*}

Hence for each $L\geq1$ and $1\leq\beta\leq d$, by defining 
\begin{align}\label{def-wl-abs}
	w_{L,\beta}(\mathbf{n}) : = \frac{|\mathbf{n}|^{\frac{\tau}{2}}}{n_{\beta}} \cdot \mathds{1}(\mathbf{n}\in\mathbf{S}_{L,\beta})
\end{align}
and
\begin{align}\label{def-Ql-abs}
	\begin{split}
		Q_{L,\beta}: &= \frac{b^{-(d-1)(m+L)}}{2\pi}\sum_{\substack{\text{all $\widehat{\mathbf{J}}_{\beta}$ induced by}\\\mathbf{J}\in \mathscr{D}_{m+L}^b(\mathbf{I})}}\Big(|\mathcal{D}_m(\ell_{\widehat{\mathbf{J}}_{\beta}},\ell_{\mathbf{J}_{\beta}(b^L)})|+|\mathcal{D}_m(\ell_{\widehat{\mathbf{J}}_{\beta}},\ell_{\mathbf{J}_{\beta}(1)})|\\
		&\qquad\qquad\qquad\qquad\qquad\quad\,\,+\sum_{l=1}^{b^L-1}|\mathcal{D}_m(\ell_{\widehat{\mathbf{J}}_{\beta}},\ell_{\mathbf{J}_{\beta}(l)})-\mathcal{D}_m(\ell_{\widehat{\mathbf{J}}_{\beta}},\ell_{\mathbf{J}_{\beta}(l+1)})|\Big), 
	\end{split}
\end{align}
we obtain 
\begin{align}\label{U-vR-abs}
	|\mathscr{U}_\mathbf{I}(\mathbf{n})| \le \sum_{\beta=1}^{d}w_{L,\beta}(\mathbf{n}) Q_{L,\beta}\quad \text{for all $b^{m+L-1} <|\mathbf{n}| \le b^{m+L}$}. 
\end{align}

\medskip
{\flushleft \bf Step 5. Separation-of-variable estimate of $\mathscr{Z}_\mathbf{I}(\mathbf{n})$.}
\medskip

Combining  \eqref{low-Y-abs}, \eqref{Y-UV-abs},  \eqref{V-wQ-abs} and \eqref{U-vR-abs}, we obtain  the desired separation-of-variable estimate
\begin{align}\label{sep-Y-abs}
	|\mathscr{Z}_\mathbf{I}(\mathbf{n})| \le    v_0(\mathbf{n}) R_0 +  \sum_{L=1}^\infty v_{L}(\mathbf{n})  R_{L} + \sum_{L=1}^\infty\sum_{\beta=1}^{d} w_{L,\beta}(\mathbf{n})  Q_{L,\beta} \quad \text{for all $\mathbf{n}\in\mathbb{N}^{d}\setminus\{\mathbf{0}\}$},
\end{align}
where $R_0$, $R_L$, $Q_{L,\beta}$ are non-negative random variables and  $v_0$, $v_L$, $w_{L,\beta}$ are deterministic (without randomness) sequences of non-negative numbers with supports 
\[
\supp(v_0)=\{1\leq|\mathbf{n}|\leq b^m\},\quad\supp(v_{L})=\{b^{m+L-1}<|\mathbf{n}|\leq b^{m+L}\}
\]
and
\[
\supp(w_{L,\beta})=\Big\{b^{m+L-1}<|\mathbf{n}|\leq b^{m+L},n_{\beta}>\frac{b^{m+L-1}}{\sqrt{d}},n_{\beta'}\leq\frac{b^{m+L-1}}{\sqrt{d}},1\leq\beta'\leq\beta-1\Big\}.
\]

\medskip
{\flushleft\bf Step 6. Upper estimate of $(\E[\|\mathscr{Z}_\mathbf{I}\|_{\ell^q}^p])^{1/p}$ via separation-of-variable.}
\medskip

By applying the triangle inequality of $\ell^q$ to \eqref{sep-Y-abs}, we have
\begin{align}\label{sep-Y-lq-abs}
	\|\mathscr{Z}_\mathbf{I}\|_{\ell^q}\leq\sum_{L=0}^\infty \|v_L\|_{\ell^q} \cdot R_L  +  \sum_{L=1}^\infty\sum_{\beta=1}^{d} \|w_{L,\beta}\|_{\ell^q} \cdot Q_{L,\beta}.
\end{align}
Then by the triangle inequality of $L^p(\PP)$ to \eqref{sep-Y-lq-abs}, we get
\begin{align}\label{E-Y-I-abs}
	(\E[\|\mathscr{Z}_\mathbf{I}\|_{\ell^q}^p])^{1/p}  = \big\| \|\mathscr{Z}_\mathbf{I}\|_{\ell^q} \big \|_{L^p(\PP)}  \le   \sum_{L=0}^\infty \|v_L\|_{\ell^q} \cdot (\E[R_L^p])^{1/p}  + \sum_{L=1}^\infty\sum_{\beta=1}^{d} \|w_{L,\beta}\|_{\ell^q}\cdot (\E[ Q_{L,\beta}^p])^{1/p}.
\end{align}

\medskip
{\flushleft \bf Step 7. Simple estimates of the quantities $\|v_L\|_{\ell^q}$ and $\|w_{L,\beta}\|_{\ell^q}$.}
\medskip

By the definitions \eqref{def-v-0-abs}, \eqref{def-v-l-abs} and \eqref{def-wl-abs} for $v_L$ and $w_{L,\beta}$, we have 
\begin{align}\label{cal-v-0-abs}
	\|v_0\|_{\ell^q}   = \Big( \sum_{1\leq|\mathbf{n}|\leq b^m} |\mathbf{n}|^{\frac{\tau q}{2}}\Big)^{1/q}  \lesssim \big(b^{\frac{\tau q}{2}m}\cdot b^{dm}\big)^{1/q} = b^{(\frac{\tau}{2} +\frac{d}{q})m}
\end{align}
and for all $L\ge 1$,
\begin{align}\label{vL-wL1-abs}
	\|v_L\|_{\ell^q}  = \Big(  \sum_{b^{m+L-1}<|\mathbf{n}|\le b^{m+L}} |\mathbf{n}|^{\frac{\tau q}{2}}  \Big)^{1/q}  \lesssim \big(b^{\frac{\tau q}{2}(m+L)}\cdot b^{d(m+L)}\big)^{1/q}=b^{(\frac{\tau}{2}+\frac{d}{q})(m+L)}. 
\end{align}
While for each $L\ge 1$ and $1\leq\beta\leq d$, since
\[
\mathbf{S}_{L,\beta}:=\Big\{b^{m+L-1}<|\mathbf{n}|\leq b^{m+L},n_{\beta}>\frac{b^{m+L-1}}{\sqrt{d}},n_{\beta'}\leq\frac{b^{m+L-1}}{\sqrt{d}},1\leq\beta'\leq\beta-1\Big\},
\]
we have
\begin{align}\label{vL-wL-abs}
	\|w_{L,\beta}\|_{\ell^q}    = \Big(  \sum_{\mathbf{n}\in\mathbf{S}_{L,\beta}} \frac{|\mathbf{n}|^{\frac{\tau q}{2}}}{n_{\beta}^q}  \Big)^{1/q}  \lesssim \Big(\frac{b^{\frac{\tau q}{2}(m+L)}}{b^{q(m+L)}}\cdot b^{d(m+L)}\Big)^{1/q}=  b^{ (\frac{\tau }{2} - 1+\frac{d}{q})(m+L)}.
\end{align}

\medskip
{\flushleft \bf Step 8. Estimate of $(\E[R_0^p])^{1/p}$.}
\medskip

Recall the definition \eqref{def-R-0-abs} for $R_0$. By the triangle inequality, we have 
\[
(\E[R_0^p])^{1/p}  = \Big\| 
\int_{\mathbf{I}}\Big[\prod_{j=0}^{m}\XX_j(\mathbf{t})\Big]    \mathrm{d}\mathbf{t}\Big\|_{L^p(\PP)}\le 
\int_{\mathbf{I}}  \Big\|  \prod_{j=0}^{m}\XX_j(\mathbf{t})  \Big\|_{L^p(\PP)} \mathrm{d}\mathbf{t}.
\]
Since the stochastic processes $\{\XX_j\}_{0\le j \le m}$ are independent, by Assumption~\ref{assum-Lp0-abs} and Lemma~\ref{lem-small-p}, as well as the definition of $S_p$ in \eqref{ass2-sma-p-abs}:
\[
S_p=\sup_{j>j_0}\sup_{\mathbf{t}\in[0,1)^d}\mathbb{E}[\XX_j^{p}(\mathbf{t})],
\]
we have 
\begin{align}\label{Dt-Lp-abs}
	\Big\|  \Big[\prod_{j=0}^{m}\XX_j(\mathbf{t})\Big] \Big\|_{L^p(\PP)} =   \Big(\prod_{j=0}^{m} \E[\XX_j^p(\mathbf{t})]   \Big)^{1/p}\lesssim S_p^{m/p}.
\end{align}
Therefore,  by recalling $|\mathbf{I}|=b^{-dm}$, we obtain 
\begin{align}\label{R0-es-abs}
	(\E[R_0^p])^{1/p} \lesssim S_p^{m/p}\cdot b^{-dm}. 
\end{align}

\medskip
{\flushleft \bf Step 9. Control of the difference $\mathcal{D}_m(\mathbf{t})-\mathcal{D}_m(\mathbf{s})$.}
\medskip

From the expressions \eqref{def-Rl-abs} and \eqref{def-Ql-abs}, we are led to study the difference $\mathcal{D}_m(\mathbf{t})-\mathcal{D}_m(\mathbf{s})$. Recall the definition \eqref{def-Dk-abs} of $\mathcal{D}_m(\mathbf{t})$, by the elementary identity
\[
\prod_{j=0}^{m}a_j-\prod_{j=0}^{m}b_j=\sum_{r=0}^{m}\Big(\prod_{j=0}^{r-1}b_j\Big) \big(a_r-b_r\big) \Big(\prod_{j=r+1}^{m}a_j\Big),
\]
we obtain 
\[
|\mathcal{D}_m(\mathbf{t})-\mathcal{D}_m(\mathbf{s})|  \le   \sum_{r=0}^{m}\Big[\prod_{j=0}^{r-1} \XX_j(\mathbf{s})\Big] \big|\XX_r(\mathbf{t})-\XX_r(\mathbf{s})\big| \Big[\prod_{j=r+1}^{m}\XX_j(\mathbf{t})\Big].
\]
By using the H\"older inequality, it follows from Assumption~\ref{assum-alp0-abs} that for any $1<p\leq p_0$, we have
\begin{align}\label{ass3-sma-p-abs}
	\sup_{r\in \N}\sup_{\mathbf{I}\in\mathscr{D}_r^b}\sup_{\mathbf{t},\mathbf{s}\in\mathbf{I} \atop\mathbf{t}\neq\mathbf{s}}\mathbb{E}\Big[\Big| \frac{\XX_r(\mathbf{t})-\XX_r(\mathbf{s})}{b^{r\alpha_0}|\mathbf{t}-\mathbf{s}|^{\alpha_0}} \Big|^{p}\Big]<\infty.
\end{align}
Therefore, for any $\mathbf{t},\mathbf{s}\in\mathbf{I}\in\mathscr{D}_{m}^b$, by applying  the triangle inequality, the independence of $\{\XX_j\}_{0\le j \le k}$ and then Assumption~\ref{assum-Lp0-abs}, Assumption~\ref{assum-alp0-abs}, as well as \eqref{ass2-sma-p-abs}, \eqref{ass3-sma-p-abs}, we obtain 
\begin{align}\label{Dt-Ds-E-abs}
	\begin{split}
		\| \mathcal{D}_m(\mathbf{t})-\mathcal{D}_m(\mathbf{s})\|_{L^p(\PP)}   &\le    \sum_{r=0}^{m}\Big\|\Big[\prod_{j=0}^{r-1} \XX_j(\mathbf{s}) \Big] \big|\XX_r(\mathbf{t})-\XX_r(\mathbf{s})\big| \Big[\prod_{j=r+1}^{m}|\XX_j(\mathbf{t})|\Big]\Big\|_{L^p(\PP)}
		\\
		& \lesssim \sum_{r=0}^{m}S_p^{m/p} \cdot  b^{r\alpha_0}\cdot|\mathbf{t}-\mathbf{s}|^{\alpha_0}
		 \lesssim S_p^{m/p}\cdot b^{m\alpha_0}\cdot|\mathbf{t}-\mathbf{s}|^{\alpha_0}.
	\end{split}
\end{align}

\medskip
{\flushleft \bf Step 10. Estimate of $(\E[R_L^p])^{1/p}$ for $L\ge 1$.}
\medskip

Recall the definition \eqref{def-Rl-abs} of $R_L$ for $L\ge 1$. We have 
\[
(\E[R_L^p])^{1/p}=  \Big\|  \sum_{\mathbf{J}\in \mathscr{D}_{m+L}^b(\mathbf{I})}  \int_{\mathbf{J}} |\mathcal{D}_{m} (\mathbf{t}) -\mathcal{D}_k(\ell_\mathbf{J}) | \mathrm{d}\mathbf{t}\Big\|_{L^p(\PP)}
\le  \sum_{\mathbf{J}\in \mathscr{D}_{m+L}^b(\mathbf{I})}  \int_{\mathbf{J}}   \|  \mathcal{D}_{m} (\mathbf{t}) -\mathcal{D}_k(\ell_\mathbf{J})  \|_{L^p(\PP)} \mathrm{d}\mathbf{t}.
\]
Now by \eqref{Dt-Ds-E-abs} and the fact that 
\[
|\mathbf{t}-\ell_\mathbf{J}|\leq\sqrt{d}\cdot b^{-(m+L)}
\]
with $\mathbf{t},\ell_\mathbf{J}\in\mathbf{J}\subset\mathbf{I}\in\mathscr{D}_{m}^b$, and then by $|\mathbf{J}|=b^{-d(m+L)}$, we obtain 
\[
(\E[R_L^p])^{1/p}\lesssim\sum_{\mathbf{J}\in\mathscr{D}_{m+L}^b(\mathbf{I})}S_p^{m/p}\cdot b^{m\alpha_0}\cdot b^{-(m+L)\alpha_0}\cdot b^{-d(m+L)}= S_p^{m/p}\cdot b^{-dm}\cdot b^{-(d+\alpha_0)L}\cdot\# \mathscr{D}_{m+L}^b(\mathbf{I}).
\]
Note that by the definition \eqref{DLk-1I-abs} of $\mathscr{D}_{m+L}^b(\mathbf{I})$ (recall that the $b$-adic sub-cube $\mathbf{I}\in\mathscr{D}_{m}^b$ is divided into $b^{dL}$ equal pieces), we have 
\[
\# \mathscr{D}_{m+L}^b(\mathbf{I})  = b^{dL}. 
\]
Hence we get
\begin{align}\label{RL-es-abs}
	(\E[R_L^p])^{1/p}   \lesssim   S_p^{m/p}\cdot b^{-dm} \cdot b^{-\alpha_0L}. 
\end{align}

\medskip
{\flushleft \bf Step 11. Estimate of $(\E[Q_{L,\beta}^p])^{1/p}$ for $L\ge 1$ and $1\leq\beta\leq d$.}
\medskip

Recall the definition \eqref{def-Ql-abs} of $Q_{L,\beta}$ for $L\ge 1$. We have  
\begin{align*}
	(\E[Q_{L,\beta}^p])^{1/p}&  = \frac{b^{-(d-1)(m+L)}}{2\pi}\Big\|
	\sum_{\substack{\text{all $\widehat{\mathbf{J}}_{\beta}$ induced by}\\\mathbf{J}\in \mathscr{D}_{m+L}^b(\mathbf{I})}}\Big(|\mathcal{D}_m(\ell_{\widehat{\mathbf{J}}_{\beta}},\ell_{\mathbf{J}_{\beta}(b^L)})|\\
	&\quad\,\,+|\mathcal{D}_m(\ell_{\widehat{\mathbf{J}}_{\beta}},\ell_{\mathbf{J}_{\beta}(1)})|+\sum_{l=1}^{b^L-1}|\mathcal{D}_m(\ell_{\widehat{\mathbf{J}}_{\beta}},\ell_{\mathbf{J}_{\beta}(l)})-\mathcal{D}_m(\ell_{\widehat{\mathbf{J}}_{\beta}},\ell_{\mathbf{J}_{\beta}(l+1)})|\Big)\Big\|_{L^p(\PP)}
\end{align*}
and hence
\begin{align*}
	(\E[Q_{L,\beta}^p])^{1/p}&\leq\frac{b^{-(d-1)(m+L)}}{2\pi}\sum_{\substack{\text{all $\widehat{\mathbf{J}}_{\beta}$ induced by}\\\mathbf{J}\in \mathscr{D}_{m+L}^b(\mathbf{I})}} \Big( \| \mathcal{D}_m(\ell_{\widehat{\mathbf{J}}_{\beta}},\ell_{\mathbf{J}_{\beta}(b^L)}) \|_{L^p(\PP)}   \\
	&\quad\,\,+ \|\mathcal{D}_m(\ell_{\widehat{\mathbf{J}}_{\beta}},\ell_{\mathbf{J}_{\beta}(1)})\|_{L^p(\PP)}+ \sum_{l=1}^{b^L-1} \| \mathcal{D}_m(\ell_{\widehat{\mathbf{J}}_{\beta}},\ell_{\mathbf{J}_{\beta}(l)})-\mathcal{D}_m(\ell_{\widehat{\mathbf{J}}_{\beta}},\ell_{\mathbf{J}_{\beta}(l+1)})\|_{L^p(\PP)}\Big). 
\end{align*}
By the same calculation as in \eqref{Dt-Lp-abs}, we have  
\[
\| \mathcal{D}_m(\ell_{\widehat{\mathbf{J}}_{\beta}},\ell_{\mathbf{J}_{\beta}(b^L)}) \|_{L^p(\PP)}\lesssim S_p^{m/p}  \anand \|\mathcal{D}_m(\ell_{\widehat{\mathbf{J}}_{\beta}},\ell_{\mathbf{J}_{\beta}(1)})\|_{L^p(\PP)}  \lesssim S_p^{m/p}. 
\]
On the other hand, by \eqref{Dt-Ds-E-abs} and by using
\[
|(\ell_{\widehat{\mathbf{J}}_{\beta}},\ell_{\mathbf{J}_{\beta}(l)})-(\ell_{\widehat{\mathbf{J}}_{\beta}},\ell_{\mathbf{J}_{\beta}(l+1)})| =|\ell_{\mathbf{J}_{\beta}(l)}-\ell_{\mathbf{J}_{\beta}(l+1)}|= b^{-(m+L)}
\]
with $(\ell_{\widehat{\mathbf{J}}_{\beta}},\ell_{\mathbf{J}_{\beta}(l)}),(\ell_{\widehat{\mathbf{J}}_{\beta}},\ell_{\mathbf{J}_{\beta}(l+1)})\in\mathbf{J}\subset\mathbf{I}\in\mathscr{D}_{m}^b$,
we obtain 
\[
\| \mathcal{D}_m(\ell_{\widehat{\mathbf{J}}_{\beta}},\ell_{\mathbf{J}_{\beta}(l)})-\mathcal{D}_m(\ell_{\widehat{\mathbf{J}}_{\beta}},\ell_{\mathbf{J}_{\beta}(l+1)})\|_{L^p(\PP)} \lesssim S_p^{m/p}\cdot b^{m\alpha_0}\cdot b^{-(m+L)\alpha_0}=S_p^{m/p}\cdot b^{-\alpha_0L}.
\]
It follows that
\begin{align*}
	&\|\mathcal{D}_m(\ell_{\widehat{\mathbf{J}}_{\beta}},\ell_{\mathbf{J}_{\beta}(b^L)}) \|_{L^p(\PP)}   + \|\mathcal{D}_m(\ell_{\widehat{\mathbf{J}}_{\beta}},\ell_{\mathbf{J}_{\beta}(1)})\|_{L^p(\PP)}+\sum_{l=1}^{b^L-1} \| \mathcal{D}_m(\ell_{\widehat{\mathbf{J}}_{\beta}},\ell_{\mathbf{J}_{\beta}(l)})-\mathcal{D}_m(\ell_{\widehat{\mathbf{J}}_{\beta}},\ell_{\mathbf{J}_{\beta}(l+1)})\|_{L^p(\PP)}\\
	&\lesssim S_p^{m/p}+S_p^{m/p}+b^L\cdot S_p^{m/p}\cdot b^{-\alpha_0L}=S_p^{m/p}\cdot b^{(1-\alpha_0)L}.
\end{align*}
Since
\[
\#\Big\{\text{all $\widehat{\mathbf{J}}_{\beta}$ induced by $\mathbf{J}\in \mathscr{D}_{m+L}^b(\mathbf{I})$}\Big\} = b^{(d-1)L},
\]
we get
\begin{align}\label{QL-es-abs}
	(\E[Q_{L,\beta}^p])^{1/p} \lesssim b^{-(d-1)(m+L)}\cdot b^{(d-1)L}\cdot S_p^{m/p}\cdot b^{(1-\alpha_0)L}=S_p^{m/p}\cdot b^{-(d-1)m}\cdot b^{(1-\alpha_0)L} .
\end{align}

\medskip
{\flushleft \bf Step 12. Conclusion of the estimate of $\E[\|\mathscr{Z}_\mathbf{I}\|_{\ell^q}^p]$.}
\medskip

Combining the inequalities \eqref{cal-v-0-abs} and \eqref{R0-es-abs}, we obtain 
\[
\|v_0\|_{\ell^q} \cdot (\E[R_0^p])^{1/p}       \lesssim     b^{(\frac{\tau}{2} +\frac{d}{q})m} \cdot S_p^{m/p}\cdot b^{-dm} =S_p^{m/p}\cdot b^{- (d - \frac{\tau }{2} - \frac{d}{q})m}.
\]
For all $L\ge 1$, by \eqref{vL-wL1-abs}, \eqref{RL-es-abs}, 
\begin{align*}
	\|v_L\|_{\ell^q} \cdot (\E[R_L^p])^{1/p} &\lesssim  b^{(\frac{\tau}{2}+\frac{d}{q})(m+L)}\cdot  S_p^{m/p}\cdot b^{-dm} \cdot b^{-\alpha_0L} \\
	&=S_p^{m/p}\cdot b^{- (d - \frac{\tau }{2} - \frac{d}{q})m} \cdot b^{-(\alpha_0-\frac{\tau}{2} -\frac{d}{q})L}.
\end{align*}
For each $L\geq1$ and $1\leq\beta\leq d$, by   \eqref{vL-wL-abs},  \eqref{QL-es-abs}, 
\begin{align*}
	\|w_{L,\beta}\|_{\ell^q} \cdot (\E[Q_{L,\beta}^p])^{1/p} &\lesssim  b^{ (\frac{\tau }{2} - 1+\frac{d}{q})(m+L)}\cdot S_p^{m/p}\cdot b^{-(d-1)m}\cdot b^{(1-\alpha_0)L} \\
	&= S_p^{m/p}\cdot b^{-(d -  \frac{\tau }{2} - \frac{d}{q})m} \cdot b^{-(\alpha_0-\frac{\tau}{2} -\frac{d}{q})L}. 
\end{align*}
Therefore, by \eqref{E-Y-I-abs}, we obtain 
\begin{align*}
	(\E[\|\mathscr{Z}_\mathbf{I}\|_{\ell^q}^p])^{1/p}  &\lesssim  S_p^{m/p}\cdot b^{- (d - \frac{\tau }{2} - \frac{d}{q})m}\Big[1 + \sum_{L=1}^\infty  b^{-(\alpha_0-\frac{\tau}{2} -\frac{d}{q})L} +  \sum_{L=1}^\infty \sum_{\beta=1}^{d} b^{-(\alpha_0-\frac{\tau}{2} -\frac{d}{q})L} \Big]\\
	&=S_p^{m/p}\cdot b^{- (d - \frac{\tau }{2} - \frac{d}{q})m}\Big[1 + (1+d)\sum_{L=1}^\infty  b^{-(\alpha_0-\frac{\tau}{2} -\frac{d}{q})L} \Big].
\end{align*}
Since $q>\frac{2d}{2\alpha_0-\tau}$, we have 
\[
\alpha_0-\frac{\tau}{2} -\frac{d}{q}>\alpha_0-\frac{\tau}{2}-\frac{2\alpha_0-\tau}{2} =   0.
\]
Hence we get
\[
\sum_{L=1}^\infty  b^{-(\alpha_0-\frac{\tau}{2} -\frac{d}{q})L}<\infty 
\]
and then
\[
(\E[\|\mathscr{Z}_\mathbf{I}\|_{\ell^q}^p])^{1/p}\lesssim S_p^{m/p}\cdot b^{- (d - \frac{\tau }{2} - \frac{d}{q})m},
\]
that is,
\[
\E[\|\mathscr{Z}_\mathbf{I}\|_{\ell^q}^p]  \lesssim  S_p^{m}\cdot b^{- (dp - \frac{\tau p}{2} - \frac{dp}{q})m}. 
\]
By the expression \eqref{Sp-phi-abs} of $S_p$:
\[
S_p=b^{d(p-1)-\frac{\zeta(p)}{\log b}},
\]
we get the desired inequality 
\[
\E[\|\mathscr{Z}_\mathbf{I}\|_{\ell^q}^p]=\mathbb{E}\Big[\Big\{\sum_{\mathbf{n}\in\mathbb{N}^{d}\setminus\{\mathbf{0}\}}|\mathscr{Z}_\mathbf{I}(\mathbf{n})|^q\Big\}^{p/q}\Big]\lesssim b^{-m[d+\frac{\zeta(p)}{\log b}-\frac{\tau p}{2}-\frac{dp}{q}]}.
\]

This completes the whole proof of Lemma~\ref{UB-YZW-abs}.

\subsection{Single weighted $L^p(\ell^q)$-boundedness}
We now prove Proposition~\ref{lem-lq-abs} by using Lemma~\ref{UB-YZW-abs}.  Recall again  that $\tau$, $p$, $q$ are fixed as in \eqref{fix-exponents}.   

Recall the definition \eqref{vec-val-mar-abs} of the vector-valued martingale $(\mathcal{M}_{m})_{m\ge 0}$ with respect to the natural increasing filtration $(\mathscr{G}_m)_{m\ge 0}$ defined in \eqref{def-filtra-abs} with any fixed $\tau\in(0,\mathcal{L}_F)$: 
\[
\mathcal{M}_{m}=\big(|\mathbf{n}|^{\frac{\tau}{2}}\widehat{\mu_{m}}(\mathbf{n})\big)_{\mathbf{n}\in\mathbb{N}^{d}\setminus\{\mathbf{0}\}}.
\]
And for each $\mathbf{n}=(n_1,\cdots,n_d)\in\mathbb{N}^{d}\setminus\{\mathbf{0}\}$, by the definition \eqref{def-num-abs} of $\mu_m$, we have
\[
\mathcal{M}_m(\mathbf{n})=|\mathbf{n}|^{\frac{\tau}{2}}\widehat{\mu_{m}}(\mathbf{n})=|\mathbf{n}|^{\frac{\tau}{2}}\int_{[0,1]^d}e^{-2\pi i\mathbf{n}\cdot\mathbf{t}}\mu_{m}(\mathrm{d}\mathbf{t})=|\mathbf{n}|^{\frac{\tau}{2}}\int_{[0,1]^d}\Big[\prod_{j=0}^{m}\XX_j(\mathbf{t})\Big]e^{-2\pi i\mathbf{n}\cdot\mathbf{t}}\mathrm{d}\mathbf{t}.
\]
Fix any integer $m\geq0$. For any $b$-adic sub-cube $\mathbf{I}\in \mathscr{D}_{m}^b$ defined in \eqref{def-b-dya-abs},  recall the $\mathscr{G}_m$-measurable random vector $\mathscr{Z}_\mathbf{I} = \big(\mathscr{Z}_\mathbf{I}(\mathbf{n})\big)_{\mathbf{n}\in\mathbb{N}^{d}\setminus\{\mathbf{0}\}}$ defined in \eqref{def-Z-abs}:
\[
\mathscr{Z}_\mathbf{I}(\mathbf{n})=|\mathbf{n}|^{\frac{\tau}{2}}\int_{\mathbf{I}}\Big[\prod_{j=0}^{m}\XX_j(\mathbf{t})\Big]e^{-2\pi i\mathbf{n}\cdot\mathbf{t}}\mathrm{d}\mathbf{t}.
\]
Hence as random vectors, we have the equality 
\[
	\mathcal{M}_m=\sum_{\mathbf{I}\in\mathscr{D}_{m}^b}\mathscr{Z}_\mathbf{I}.
\]

By the triangle inequality of $\ell^q$, we have
\[
\Big\{\sum_{\mathbf{n}\in\mathbb{N}^{d}\setminus\{\mathbf{0}\}}|\mathcal{M}_{m}(\mathbf{n})|^q\Big\}^{1/q}=\Big\{\sum_{\mathbf{n}\in\mathbb{N}^{d}\setminus\{\mathbf{0}\}}\Big|\sum_{\mathbf{I}\in\mathscr{D}_{m}^b}\mathscr{Z}_\mathbf{I}(\mathbf{n})\Big|^q\Big\}^{1/q}\leq\sum_{\mathbf{I}\in\mathscr{D}_{m}^b}\Big\{\sum_{\mathbf{n}\in\mathbb{N}^{d}\setminus\{\mathbf{0}\}}|\mathscr{Z}_\mathbf{I}(\mathbf{n})|^q\Big\}^{1/q}
\]
and hence
\[
\mathbb{E}\Big[\Big\{\sum_{\mathbf{n}\in\mathbb{N}^{d}\setminus\{\mathbf{0}\}}|\mathcal{M}_{m}(\mathbf{n})|^q\Big\}^{p/q}\Big]\leq\mathbb{E}\Big[\Big(\sum_{\mathbf{I}\in\mathscr{D}_{m}^b}\Big\{\sum_{\mathbf{n}\in\mathbb{N}^{d}\setminus\{\mathbf{0}\}}|\mathscr{Z}_\mathbf{I}(\mathbf{n})|^q\Big\}^{1/q}\Big)^p\Big].
\]
Then by the triangle inequality of $L^p(\PP)$, we get
\[
\Big(\mathbb{E}\Big[\Big\{\sum_{\mathbf{n}\in\mathbb{N}^{d}\setminus\{\mathbf{0}\}}    |\mathcal{M}_{m}(\mathbf{n})|^q\Big\}^{p/q}\Big]\Big)^{1/p}\leq\sum_{\mathbf{I}\in\mathscr{D}_{m}^b}\Big(\mathbb{E}\Big[\Big\{\sum_{\mathbf{\mathbf{n}}\in\mathbb{N}^{d}\setminus\{\mathbf{0}\}}|\mathscr{Z}_\mathbf{I}(\mathbf{n})|^q\Big\}^{p/q}\Big]\Big)^{1/p}.
\]
Therefore, by Lemma~\ref{UB-YZW-abs}, we obtain
\[
\Big(\mathbb{E}\Big[\Big\{\sum_{\mathbf{n}\in\mathbb{N}^{d}\setminus\{\mathbf{0}\}}    |\mathcal{M}_{m}(\mathbf{n})|^q\Big\}^{p/q}\Big]\Big)^{1/p}\leq \sum_{\mathbf{I}\in\mathscr{D}_{m}^b}C^{1/p}\cdot b^{-\frac{m}{p}[d + \frac{\zeta(p)}{\log b}-\frac{\tau p}{2}-\frac{dp}{q}]}=C^{1/p}\cdot b^{-\frac{m}{p}[d + \frac{\zeta(p)}{\log b}-\frac{\tau p}{2}-\frac{dp}{q}]}\cdot\#\mathscr{D}_{m}^b.
\]
Note that $\#\mathscr{D}_{m}^b=b^{dm}$, hence
\[
\Big(\mathbb{E}\Big[\Big\{\sum_{\mathbf{n}\in\mathbb{N}^{d}\setminus\{\mathbf{0}\}}    |\mathcal{M}_{m}(\mathbf{n})|^q\Big\}^{p/q}\Big]\Big)^{1/p}\leq C^{1/p}\cdot b^{-\frac{m}{p}[d + \frac{\zeta(p)}{\log b}-\frac{\tau p}{2}-\frac{dp}{q}]}\cdot b^{dm}=C^{1/p}\cdot b^{\frac{m}{p}[d(p-1) - \frac{\zeta(p)}{\log b}+\frac{\tau p}{2}+\frac{dp}{q}]}.
\]
This implies the desired inequality
\[
\mathbb{E}\Big[\Big\{\sum_{\mathbf{n}\in\mathbb{N}^{d}\setminus\{\mathbf{0}\}}    |\mathcal{M}_{m}(\mathbf{n})|^q\Big\}^{p/q}\Big]\leq C\cdot b^{m[d(p-1) - \frac{\zeta(p)}{\log b}+\frac{\tau p}{2}+\frac{dp}{q}]}<\infty\quad\text{for any fixed $m\geq0$}.
\]
Hence we complete the proof of Proposition~\ref{lem-lq-abs}.

\subsection{Asymptotic behavior of $\mathbb{E}[\|\mathscr{Y}_\mathbf{I} \|_{\ell^q}^p]$}
We  can now prove Lemma~\ref{UB-ZW-abs} by using Lemma~\ref{UB-YZW-abs}. 

Recall again  that $\tau$, $p$, $q$ are fixed as in \eqref{fix-exponents}.   For any integer $k\ge 1$ and any $b$-adic sub-cube $\mathbf{I}\in \mathscr{D}_{k-1}^b$ defined in \eqref{def-b-dya-abs}, recall the $\mathscr{G}_k$-measurable random vector $\mathscr{Y}_\mathbf{I}=\big(\mathscr{Y}_\mathbf{I}(\mathbf{n})\big)_{\mathbf{n}\in\mathbb{N}^{d}\setminus\{\mathbf{0}\}}$ defined in \eqref{def-Y-abs}:
\[
\mathscr{Y}_\mathbf{I}(\mathbf{n})= |\mathbf{n}|^{\tau/2}\int_{\mathbf{I}}\Big[\prod_{j=0}^{k-1}\XX_j(\mathbf{t})\Big]  \mathring{\XX}_{k} (\mathbf{t}) e^{-2\pi i\mathbf{n}\cdot\mathbf{t}}\mathrm{d}\mathbf{t}
\]
with
$
\mathring{\XX}_{k}(\mathbf{t})=\XX_{k}(\mathbf{t})-\E[\XX_{k}(\mathbf{t})]=\XX_{k}(\mathbf{t})-1.
$
Using the notation \eqref{DLk-1I-abs} for dividing $\mathbf{I}\in\mathscr{D}_{k-1}^b$ into $b^d$ equal pieces in $\mathscr{D}_{k}^b$:
\[
\mathbf{I}=\bigsqcup_{\mathbf{J}\in\mathscr{D}_{k}^b(\mathbf{I})}\mathbf{J} \quad \text{with $\mathscr{D}_{k}^b(\mathbf{I})= \Big\{\mathbf{J}\subset\mathbf{I}:\mathbf{J}\in\mathscr{D}_{k}^b\Big\}$},
\]
we obtain 
\begin{align*}
\mathscr{Y}_\mathbf{I}(\mathbf{n})& = |\mathbf{n}|^{\tau/2}\int_{\mathbf{I}}\Big[\prod_{j=0}^{k}\XX_j(\mathbf{t})\Big]e^{-2\pi i\mathbf{n}\cdot\mathbf{t}}\mathrm{d}\mathbf{t}-|\mathbf{n}|^{\tau/2}\int_{\mathbf{I}}\Big[\prod_{j=0}^{k-1}\XX_j(\mathbf{t})\Big]e^{-2\pi i\mathbf{n}\cdot\mathbf{t}}\mathrm{d}\mathbf{t}
\\
& =\Bigg( \sum_{\mathbf{J}\in\mathscr{D}_{k}^b(\mathbf{I})}|\mathbf{n}|^{\tau/2}\int_{\mathbf{J}}\Big[\prod_{j=0}^{k}\XX_j(\mathbf{t})\Big]e^{-2\pi i\mathbf{n}\cdot\mathbf{t}}\mathrm{d}\mathbf{t}\Bigg)-|\mathbf{n}|^{\tau/2}\int_{\mathbf{I}}\Big[\prod_{j=0}^{k-1}\XX_j(\mathbf{t})\Big]e^{-2\pi i\mathbf{n}\cdot\mathbf{t}}\mathrm{d}\mathbf{t}.
\end{align*}
It should be emphasized that here $\mathbf{I}\in \mathscr{D}_{k-1}^b$ and $\mathbf{J}\in\mathscr{D}_{k}^b(\mathbf{I})\subset\mathscr{D}_{k}^b$. Hence comparing the above equality with  $\mathscr{Z}_\mathbf{I}(\mathbf{n})$ defined in \eqref{def-Z-abs}, we may rewrite
\[
\mathscr{Y}_\mathbf{I}(\mathbf{n})=\Big(\sum_{\mathbf{J}\in\mathscr{D}_{k}^b(\mathbf{I})}\mathscr{Z}_\mathbf{J}(\mathbf{n})\Big)-\mathscr{Z}_\mathbf{I}(\mathbf{n}).
\]
By Lemma~\ref{UB-YZW-abs}, we have
\[
\E[\|\mathscr{Z}_\mathbf{I}\|_{\ell^q}^p]\lesssim b^{-(k-1)[d + \frac{\zeta(p)}{\log b}-\frac{\tau p}{2}-\frac{dp}{q}]}\lesssim b^{-k[d + \frac{\zeta(p)}{\log b}-\frac{\tau p}{2}-\frac{dp}{q}]}
\]
and
\[
\E[\|\mathscr{Z}_\mathbf{J}\|_{\ell^q}^p]\lesssim b^{-k[d + \frac{\zeta(p)}{\log b}-\frac{\tau p}{2}-\frac{dp}{q}]}\quad\text{for all $\mathbf{J}\in\mathscr{D}_{k}^b(\mathbf{I})$}.
\]
Therefore, note that $\#\mathscr{D}_{k}^b(\mathbf{I})=b^d$ is finite, we obtain the desired inequality
\[
\E[\|\mathscr{Y}_\mathbf{I}\|_{\ell^q}^p]\lesssim b^{-k[d+\frac{\zeta(p)}{\log b}-\frac{\tau p}{2}-\frac{dp}{q}]}.
\]
This completes the proof of Lemma~\ref{UB-ZW-abs}.

\subsection{Uniform weighted $L^p(\ell^q)$-boundedness}\label{sec-mart-type}
We now prove Proposition~\ref{Uni-Bou-Lplq-abs} by using  Proposition~\ref{lem-lq-abs}, Lemma~\ref{dec-vm-abs} and Lemma~\ref{UB-ZW-abs}. Recall again  that $\tau$, $p$, $q$ are fixed as in \eqref{fix-exponents}.

Recall the definitions of $\mathcal{L}_F$ in \eqref{def-LF-abs} and   of $\zeta(p)=\zeta_{j_0}(p)$ in \eqref{def-phi-abs} (which is repeated in \eqref{def-phi-abs-again}).  

{\flushleft \it Claim:} For any $\tau\in(0,\mathcal{L}_F)$, there exists a large enough integer $j_0>0$ and a pair  of exponents $(p,q)$ such that  $1<p\leq p_0\leq\max\{2,\frac{2d}{2\alpha_0-\tau}\}<q<\infty$ and 
\begin{align}\label{existen-p-q-abs}
	\frac{\zeta(p)}{\log b}-\frac{\tau p}{2}-\frac{dp}{q} = \frac{\zeta_{j_0}(p)}{\log b}-\frac{\tau p}{2}-\frac{dp}{q}>0.
\end{align}

Indeed, by choosing $q$ large enough, it suffices to find large enough $j_0$ and $1<p\le p_0$ such that 
\begin{align}\label{find-p-q}
\frac{\zeta(p)}{\log b}-\frac{\tau p}{2} = \frac{\zeta_{j_0}(p)}{\log b}-\frac{\tau p}{2}>0.
\end{align}
By the definition of $\mathcal{L}_F$ in \eqref{def-LF-abs}, we have 
\[
\tau<\mathcal{L}_F\leq\sup_{1<p\leq p_0 }\frac{2\Theta(p)}{p\log b}.
\]
Therefore,  there exists $1<p\leq p_0$ such that
\[
\frac{\Theta(p)}{\log b}-\frac{\tau p}{2}>0.
\]
Comparing  the definition \eqref{def-Theta-abs} of $\Theta(p)$ and the definition  \eqref{def-phi-abs-again} of $\zeta(p)=\zeta_{j_0}(p)$, we know that 
\[
\lim_{j_0\to\infty}\zeta_{j_0}(p)=\Theta(p).
\]
Hence  the inequality \eqref{find-p-q} holds for large enough $j_0$ and $q$, which completes the proof of the Claim. 

By Proposition~\ref{lem-lq-abs}, $\E[\|\mathcal{M}_m\|_{\ell^q}^p]<\infty$ for each $m=0,1,\cdots,k_0$. Therefore, by Lemma~\ref{dec-vm-abs}, to prove  Proposition~\ref{Uni-Bou-Lplq-abs}, it suffices to prove the inequality 
\[
	\sum_{k=k_0+1}^{\infty}N_{k-1}^{p-1}\sum_{\mathbf{I} \in \mathscr{D}_{k-1}^b}\mathbb{E}[\|\mathscr{Y}_\mathbf{I} \|_{\ell^q}^p]<\infty.
\]
This inequality follows from Lemma~\ref{UB-ZW-abs}. Indeed, since $\# \mathscr{D}_{k-1}^b = b^{d(k-1)}$,  by Lemma~\ref{UB-ZW-abs}, we have 
\begin{align*}
	\sum_{k=k_0+1}^{\infty}N_{k-1}^{p-1}\sum_{\mathbf{I} \in \mathscr{D}_{k-1}^b}\mathbb{E}[\|\mathscr{Y}_\mathbf{I} \|_{\ell^q}^p] &\lesssim \sum_{k=k_0+1}^{\infty}N_{k-1}^{p-1}\cdot  b^{d(k-1)}\cdot   b^{-k [d+\frac{\zeta(p)}{\log b}-\frac{\tau p}{2}-\frac{dp}{q}] }\\
	&\lesssim \sum_{k=k_0+1}^\infty N_{k-1}^{p-1}\cdot b^{-k [\frac{\zeta(p)}{\log b}-\frac{\tau p}{2}-\frac{dp}{q}] }.
\end{align*}
Note that $\frac{\zeta(p)}{\log b}-\frac{\tau p}{2}-\frac{dp}{q}>0$ from \eqref{existen-p-q-abs}, thus by the sub-exponential growth condition  \eqref{def-subexp-abs} of $N_{k-1}$, the above series is convergent. Hence we obtain the desired inequality
\[
\sup_{m\geq 0}\mathbb{E}[\|\mathcal{M}_{m}\|_{\ell^{q}}^{p}]<\infty.
\]
This completes the proof of Proposition~\ref{Uni-Bou-Lplq-abs}.

\subsection{Proof of Theorem~\ref{Fou-Dec-abs}}
We can now  prove  the unified Theorem~\ref{Fou-Dec-abs} from Proposition~\ref{Uni-Bou-Lplq-abs}. Let us fix the exponents $p$ and $q$ as in Proposition~\ref{Uni-Bou-Lplq-abs}. By the standard argument in the theory of vector-valued martingales and the weak convergence \eqref{def-lim-mea-abs}, the inequality \eqref{sup-Lplq-abs} in Proposition~\ref{Uni-Bou-Lplq-abs} is equivalent to
\[
\mathbb{E}\Big[\Big\{\sum_{\mathbf{n}\in\mathbb{N}^{d}\setminus\{\mathbf{0}\}}\big||\mathbf{n}|^{\frac{\tau}{2}} \widehat{\mu_{\infty}}(\mathbf{n})\big|^{q}\Big\}^{p/q}\Big]<\infty
\]
and hence, almost surely, 
\[
|\widehat{\mu_{\infty}}(\mathbf{n})|^2=O(|\mathbf{n}|^{-\tau})\quad\text{as $\mathbf{n}\to\infty$}. 
\]
Since $\tau\in(0, \mathcal{L}_F)$ is chosen arbitrarily, the above asymptotic relation provides the almost sure lower bound $\dim_F(\mu_{\infty})\geq \mathcal{L}_F>0$, which completes the whole proof of Theorem~\ref{Fou-Dec-abs}.

%\bibliographystyle{alpha}
%\bibliography{bib-GMC}

\newcommand{\etalchar}[1]{$^{#1}$}

\end{document}